\documentclass[11pt]{article}
\usepackage[margin=1in]{geometry}
\usepackage{setspace}

\usepackage[utf8]{inputenc}
\usepackage[numbers,sort&compress]{natbib} 
\usepackage{multirow}
\usepackage[pdftex]{graphicx}
\usepackage{booktabs}
\usepackage{array}
\usepackage{algorithm}
\usepackage{algpseudocode}

\usepackage{amsmath,amsfonts,bm, amsthm}
\usepackage{amssymb}

\usepackage{subcaption}
\usepackage{mathtools}
\usepackage{lipsum}
\usepackage{enumitem}
\usepackage{chngcntr}
\usepackage{apptools}
\usepackage{xcolor}

\usepackage{comment}
\let\hat\widehat

\usepackage[toc,page]{appendix}
\usepackage[perpage]{footmisc}

\usepackage{times}

\usepackage[export]{adjustbox}

\usepackage{cancel}
\usepackage{verbatim}
\usepackage{upgreek}

\usepackage{macros}
\usepackage{tikz}
\usetikzlibrary{calc,trees, positioning,arrows.meta}
\tikzset{
    every node/.style={font=\sffamily\small},
    main node/.style={thick,circle,draw,font=\sffamily\Large}
}

\usepackage[normalem]{ulem}
\usepackage{authblk}

\newcommand\redsout{\bgroup\markoverwith{\textcolor{red}{\rule[0.5ex]{2pt}{0.4pt}}}\ULon}
\usepackage{hyperref}
\hypersetup{colorlinks,breaklinks,
    citecolor=[RGB]{153,204,153},
    linkcolor=[RGB]{102,153,204},
    urlcolor=[RGB]{115,140,191}
}

\newsavebox\affbox
\makeatletter
\renewcommand\AB@authnote[1]{%
  \ifcase#1\or\textsuperscript{$\dagger$}\or\textsuperscript{$\ddagger$}\or
  \textsuperscript{*}\or\textsuperscript{\S}\or
  \textsuperscript{\P}\or\textsuperscript{**}\or
  \textsuperscript{$\dagger\dagger$}\fi}
  
\makeatother

\title{Multistage  Conditional Compositional Optimization}
\author[1]{Buse \c{S}en}
\author[2]{Yifan Hu}
\author[1]{Daniel Kuhn}

\affil[$\dagger$]{\small EPFL, Switzerland \authorcr 
{\href{mailto:buse.sen@epfl.ch}{\texttt{buse.sen@epfl.ch}}} \authorcr 
{\href{mailto:daniel.kuhn@epfl.ch}{\texttt{daniel.kuhn@epfl.ch}}}}

\affil[$\ddagger$]{\small Rutgers University, USA \authorcr 
\href{mailto:yifan.hu@rutgers.edu}{\texttt{yifan.hu@rutgers.edu}}}

\date{\today}

\begin{document}
\maketitle
\begin{abstract}
    We introduce Multistage Conditional Compositional Optimization (MCCO) as a new paradigm for decision-making under uncertainty that combines aspects of multistage stochastic programming and conditional stochastic optimization. MCCO minimizes a nest of conditional expectations and nonlinear cost functions. It has numerous applications and arises, for example, in optimal stopping, linear–quadratic regulator problems, distributionally robust contextual bandits, as well as in problems involving dynamic risk measures. The na\"ive nested sampling approach for MCCO suffers from the curse of dimensionality familiar from scenario tree-based multistage stochastic programming, that is, its scenario complexity grows exponentially with the number of nests. We develop new multilevel Monte Carlo techniques for MCCO whose scenario complexity grows only polynomially with the desired accuracy. 
\end{abstract}

\section{Introduction}
Multistage stochastic optimization constitutes a central and still largely unresolved challenge in the field of optimization under uncertainty. Efficient exact solution methods are out of reach even if the probability distribution of the random problem parameters is given in closed form. Indeed, even linear two-stage stochastic programs with independent and uniformly distributed random parameters are \char"0023 P-hard \cite{dyer2006complexity, hanasusanto2016comment}. Consequently, multistage stochastic programs are often studied in a black-box framework, where the decision-maker has no analytical distributional information and aims to construct an $\epsilon$-optimal solution relying only on oracles that supply samples of the random problem parameters. In this framework, multistage stochastic programs are often addressed with the {\em Sample Average Approximation} (SAA), which builds a scenario tree using conditional sampling and solves a finite-dimensional approximate stochastic program on this empirically generated tree. However, the scenario complexity of SAA for obtaining an $\epsilon$-optimal solution to a $T$-stage linear stochastic program is bounded above by $\cO(\epsilon^{-2T})$ \citep{shapiro2005complexity, shapiro2006complexity}, and this bound cannot be significantly improved by any SAA-type method even if the random parameters are serially independent \cite[Proposition~2]{reaiche2016note}. Hence, the size of this scenario tree grows exponentially with the number of stages. We are not aware of any other black-box method that overcomes this curse of dimensionality.

With this paper we hope to lay the groundwork for a new generation of black-box methods for multistage stochastic optimization whose scenario complexity depends on the desired accuracy~$\epsilon$ only through $(1/\epsilon)^p$, where the exponent~$p$ is independent of~$T$. To this end, we study \emph{Multistage Conditional Compositional Optimization} (MCCO) problems, which are reminiscent of multistage stochastic programs. In fact, we will show that MCCO encapsulates various stochastic dynamic control problems with serially dependent disturbances as special cases. In addition, we will argue that MCCO is highly expressive and even includes problem instances that cannot be reformulated as any risk-neutral multistage stochastic program. The main contribution of this paper is to propose new multilevel Monte Carlo methods for MCCO problems with~$T\geq 2$ stages whose scenario complexity scales as~$\cO(\epsilon^{-2})$ under standard regularity conditions.

The MCCO problems to be investigated in this paper can be represented as
\begin{equation}
\tag{MCCO}
    	\label{problem:MCCO}
    	\min_{x\in\mathcal{X}}  F(x),
\end{equation}
where $\mathcal{X}\subseteq\RR^{d}$ denotes the feasible set, and where the objective function $F:\RR^d\rightarrow\RR$ is defined as a nest of~$T$ conditional expectations and nonlinear integrands of the form
\begin{equation}
    \tag{MCCE}
    \label{problem:MCCE} 
    F(x) \coloneqq \EE \Bigg[f_1\Bigg(\xi_1,\EE_{1} \bigg[f_2\bigg(\xi_2,\ldots \EE_{T-2} \Big[f_{T-1}\left(\xi_{T-1},\EE_{T-1} \left[f_T\left(\xi_T,x\right)\right]\right)\Big]\ldots\bigg) \bigg]\Bigg)\Bigg].
\end{equation}
Here, $\xi_1,\ldots,\xi_T$ represents a stochastic process, $\EE[\cdot]$ denotes unconditional expectation, and $\EE_{t}[\cdot]$ denotes expectation conditional on the $\sigma$-algebra $\cF_t$ generated by the time-$t$ observation history $\xi_{[t]} \coloneqq (\xi_1,...,\xi_{t})$. The integrands $f_t:\RR^{m_t}\times\RR^{d_t}\rightarrow \RR^{d_{t-1}}$ with $d_0=1$ and $d_T=d$ can be interpreted as uncertainty-affected cost functions. A key assumption throughout the paper is that the random vectors~$\xi_t$ are {\em dependent} and that the integrands~$f_t$ are {\em nonlinear}. Hence, evaluating the objective function at a given~$x\in\cX$ is challenging, a problem which we henceforth refer to as \emph{Multistage Conditional Compositional Estimation} (MCCE).

MCCO reduces to classical (single-stage) stochastic optimization when~$T=1$, to conditional stochastic optimization when~$T=2$ \citep{dai2017learning, hu2020biased, hu2021bias, hu2020sample, goda2023constructing, he2023debiasing}, to multilevel compositional stochastic optimization when the random vectors~$\xi_t$ are mutually independent \citep{jiang2022optimal, gao2024decentralized, xiao2022projection, cong2020minimal, zhang2021multilevel, yang2019multilevel, balasubramanian2022stochastic, ruszczynski2021stochastic, chen2021solving} and to MCCE when the feasible set~$\cX$ is a singleton \citep{rainforth2018nesting,syed2023optimal,zhou2022unbiased}. As we allow~$T$ to be any integer and the random vectors~$\xi_t$ to be  dependent, MCCO strictly generalizes all of these more basic decision-making paradigms. 

Both MCCE and MCCO arise in many applications relevant for operations research, economics or engineering. For example, the problem of computing credit valuation adjustments can be naturally formulated as an instance of MCCE \cite{giles2023efficient}. MCCE also emerges in optimal stopping problems, which encompass applications such as option pricing in finance~\citep{zhou2022unbiased} or inventory replenishment in operations management~\citep{ozyoruk2022end}. In addition, we will show that linear quadratic regulator problems and general finite-state stochastic control problems can be reformulated as instances of MCCE. Remarkably, these reformulations are possible even when the disturbances are serially dependent, as is the case when the underlying system is in a transient regime or is subject to distribution shifts. Finally, MCCO naturally arises in distributionally robust contextual bandits~\citep{shen2024wasserstein} or nested risk minimization~\citep{wang2024bayesian}. All of these examples, which will be presented in detail in Section~\ref{sec:motivational_ex}, underscore the strong ties between MCCE, MCCO and multistage stochastic optimization. However, MCCO is not subsumed by multistage stochastic optimization. Indeed, the infimum of any risk-neutral multistage stochastic program is concave in the joint distribution of the random problem parameters, whereas~$F(x)$ clearly fails to be concave in the distribution of~$\xi_{[T]}$ for some choices of the integrands~$f_t$.

Much like in multistage stochastic optimization, efficient exact solution methods for MCCE and MCCO are out of reach even if the probability distribution of the underlying stochastic process is given in closed form. Indeed, solving MCCE exactly is \char"0023 P-hard even if~$T=1$, $\xi_1$ follows the uniform distribution on the unit hypercube in~$\RR^{m_1}$ and~$f_1$ is a piecewise affine function with only two pieces \citep[Corollary~1]{hanasusanto2016comment}. We thus adopt a black-box model, that is, we assume access to~$T$ oracles (``black boxes'') that one can use to generate any number of independent samples from the marginal distribution of~$\xi_1$ and the conditional distribution of~$\xi_{t+1}$ given~$\xi_{[t]}$, $t\in[T-1]$, respectively. However, we do not presume to know these distributions. In this situation, we can resort to sampling methods for solving MCCE and MCCO approximately.

In this paper we will design and analyze function value estimators~$\hat{F}(x)$ for~$F(x)$ that attain a specified mean squared error with a minimum number of scenarios ({\em i.e.}, sample paths), and we will refer to this number as the estimator's mean squared error-based scenario complexity.  Given a function value estimator~$\hat F(x)$, one can attempt to solve the following sample-based approximation of the intractable MCCO problem.
\begin{equation}
	\phantomsection \label{prob:min_hatF(x)}
    	\min_{x\in\mathcal{X}}\; \hat F(x)
\end{equation}
We will show that the minimizers of problem~\eqref{prob:min_hatF(x)} are near-optimal in~\eqref{problem:MCCO} with high probability. However, computing these minimizers can be challenging in the absence of restrictive convexity and monotonicity conditions. Thus, we will also design and analyze sample-efficient estimators for~$\nabla F(x)$ that can be employed within stochastic gradient descent-type methods for solving problem~\eqref{problem:MCCO} to stationarity. 

The SAA estimator is arguably the simplest estimator for~$F(x)$. It is obtained by replacing all conditional expectations in~\eqref{problem:MCCE} with conditional sample averages. This amounts to approximating the stochastic process of the disturbances~$\xi_t$, $t\in[T]$, by a scenario tree. The number of samples used to approximate~$\EE_t[\cdot]$ coincides with the scenario tree's branching factor at stage~$t$. We will show that the scenario complexity of an SAA estimator with optimally tuned branching factors amounts to~$\cO(\epsilon^{-2T})$ if all integrands~$f_t$ are Lipschitz continuous and to~$\cO(\epsilon^{-(T+1)})$ if the integrands are additionally smooth. In both cases, the SAA estimator suffers from a curse of dimensionality. To design more efficient estimators, we leverage \emph{Multilevel Monte Carlo} (MLMC) methods, which have originally been popularized in computational finance and uncertainty quantification \cite{giles2008multilevel, giles2015multilevel}. A basic MLMC estimator can be constructed as follows. For any level~$\ell\in\mathbb N_0$, let~$\hat Y_\ell$ denote the SAA estimator of~$F(x)$ with a uniform branching factor~$2^\ell$ across all nodes of the scenario tree, and set~$\hat Y_{-1}=0$. By the law of large numbers, $\hat Y_\ell$ converges almost surely to~$F(x)$ as~$\ell$ grows. Hence, $F(x)$ admits the telescoping sum representation $F(x) = \EE [\sum_{\ell=0}^\infty (\hat Y_\ell -\hat Y_{\ell -1}) ]$. If~$\lambda$ is a discrete random variable independent of all samples and if~$\mathbb P(\lambda=\ell)=q(\ell)>0$ for all~$\mathbb N_0$, then the law of total expectation implies
\begin{equation}
	\label{eq:MLMC-basic}
    	F(x) = \EE \left[\sum_{\ell=0}^\infty q(\ell) \frac{\hat Y_\ell -\hat Y_{\ell -1}}{q(\ell)} \right] = \EE \left[\frac{\hat Y_\lambda -\hat Y_{\lambda -1}}{q(\lambda)} \right].
\end{equation}
The inverse propensity weighted telescoping term~$\hat F(x)=(\hat Y_\lambda -\hat Y_{\lambda -1})/q(\lambda)$ associated with level~$\lambda$ constitutes a randomized MLMC estimator in the spirit of~\cite{rhee2015unbiased, blanchet2015unbiased}. By construction, it satisfies $\EE[\hat F(x)]=F(x)$ and is thus unbiased. The choice of the distribution for~$\lambda$ is critical for controlling its computational cost. To see this, note that the SAA estimator~$\hat Y_\ell$ requires~$2^{\ell(T-1)}$ scenarios---one for each leaf node of the underlying scenario tree---implying that its computational cost grows quickly with~$\ell$. Drawing~$\lambda$ from a geometric distribution, for instance, ensures that small (inexpensive) levels~$\ell$ are sampled with high probability, while large (expensive) levels are sampled rarely. This trick allows us to reduce the expected computational cost of~$\hat F(x)$ while preserving its unbiasedness. To further ensure that the variance of~$\hat F(x)$ remains bounded, we reuse all samples of the level~$\ell -1$ scenario tree when constructing the level~$\ell$ tree. Hence, $\hat Y_\ell$ and~$\hat Y_{\ell-1}$ are strongly positively correlated, and the variance of~$\hat F(x)$ remains bounded thanks to a control variate effect.

In this paper, we introduce new MLMC estimators for~$F(x)$ and~$\nabla F(x)$ whose construction departs from that of the na\"ive MLMC estimator described above in three key respects. First, we truncate the telescoping sum in~\eqref{eq:MLMC-basic} at a maximum level~$M$ in order to control the estimator's variance \cite{blanchet2017unbiased}. Truncation introduces a bias, but the bias decays exponentially with~$M$. Hence, selecting~$M$ becomes a matter of balancing bias and variance. Second, we adopt a stagewise recursive construction employing antithetic sampling~\cite{giles2014antithetic}, thereby improving the estimator's computational cost. Third, we average multiple independent and identically distributed (i.i.d.) copies of the same estimator to reduce the variance. We will prove that the scenario complexity of the resulting MLMC estimator for~$F(x)$ amounts to $\cO(\log(\epsilon^{-1})^{2T-2}\epsilon^{-2})$ if all integrands~$f_t$ are Lipschitz continuous and to~$\cO(\epsilon^{-2})$ if the integrands are additionally smooth. Under an additional Hölder continuity assumption, analogous results are obtained for our MLMC estimator of~$\nabla F(x)$.


The main contributions of this paper can be summarized as follows.

\begin{enumerate}
    
    \item {\bf SAA value estimator:} We provide a refined scenario complexity analysis of the SAA estimator for~$F(x)$. First, we prove that the mean squared error-based scenario complexity of this estimator scales as~$\cO(\epsilon^{-2T})$ and $\cO(\epsilon^{-(T+1)})$ for nonsmooth and smooth integrands, respectively. Under a standard sub-Gaussianity assumption, we further show that, for nonsmooth and smooth integrands respectively, $\cO(\log(\epsilon^{-1})\epsilon^{-2T})$ and $\cO(\log(\epsilon^{-1})\epsilon^{-(T+1)})$ scenarios suffice to ensure that the estimator differs from~$F(x)$ at most by~$\epsilon$, uniformly across all $x\in\cX$ and with high probability. All bounds are presented explicitly without any unknown constants. Thus, they provide finite-sample guarantees. 
    
    \item {\bf MLMC value estimator:} We propose a new recursive truncated MLMC estimator for~$F(x)$ together with sharp\footnote{Some of these guarantees are sharp only up to poly-logarithmic factors in $\epsilon^{-1}$.} scenario complexity guarantees. First, we prove that the mean squared error-based scenario complexity of this estimator scales as~$\cO(\log(\epsilon^{-1})^{2T-2}\epsilon^{-2})$ and~$\cO(\epsilon^{-2})$ for nonsmooth and smooth integrands, respectively. Under a standard sub-Gaussianity assumption, we further show that, for nonsmooth and smooth integrands respectively, $\cO(\log(\epsilon^{-1})^{T}\epsilon^{-2})$ and $\cO(\log(\epsilon^{-1})\epsilon^{-2})$ scenarios suffice to ensure that the estimator differs from~$F(x)$ at most by~$\epsilon$, uniformly across all $x\in\cX$ and with high probability. Consequently, we provide the first scenario complexity results for MCCE with nonsmooth integrands that eliminate the usual exponential dependence on~$T$ in~$\epsilon^{-1}$.
    
    \item {\bf MCCO paradigm:} We introduce MCCO as a new paradigm for optimization under uncertainty and show that it unifies and generalizes numerous traditional decision-making models from the literature.
    
    \item {\bf MLMC gradient estimator:} We propose new recursive truncated as well as {\em un}truncated MLMC estimators for~$\nabla F(x)$ together with sharp scenario complexity guarantees. Under the assumption that the integrands are smooth and have H\"older continuous Hessians, we prove that the mean squared error-based scenario complexities of both estimators scale as~$\cO(\epsilon^{-2})$. If the Hölder continuity assumption is relaxed, the scenario complexity increases to~$\cO(\log(\epsilon^{-1})^{2T-2}\epsilon^{-2})$ for the truncated and diverges for the untruncated estimator. We develop the first MLMC gradient estimators for MCCO with $T\geq 2$. 
    
    \item {\bf SGD algorithm:} We integrate our new MLMC gradient estimators into a stochastic gradient descent (SGD) algorithm for finding a stationary point of problem~\eqref{problem:MCCO}. We prove that if all integrands are smooth and have H\"older continuous Hessians, then this algorithm needs $\cO(\epsilon^{-4})$ scenarios to reach an $\epsilon$-stationary point both when equipped with the truncated and untruncated variant of the estimator. 
    
    \item {\bf Numerical experiments:} We validate our theoretical results with numerical experiments revolving around Bermudan option pricing (for MCCE) and distributionally robust contextual bandits (for MCCO). These experiments confirm that the proposed MLMC estimators achieve the same mean square error as an SAA estimator using a significantly smaller number of scenarios.
\end{enumerate}

\paragraph{Literature Review} MCCE with dependent disturbances has been widely studied in the literature. However, most works focus on problems with $T=2$ stages, which arise in derivative pricing \cite{bujok2015multilevel, haji2023nested} or risk management \cite{giles2015multilevel, gordy2010nested, broadie2011efficient, giles2019multilevel}. For~$T=2$, it is well known that SAA estimators typically achieve a root mean squared error of~$\epsilon$ using~$\cO(\epsilon^{-3})$ scenarios, while MLMC estimators require only~$\cO(\epsilon^{-2})$ scenarios \citep{giles2018mlmc}. The literature on MCCE with~$T>2$ stages is far more limited. It is shown in~\citep{rainforth2018nesting} that the scenario complexity of SAA estimators amounts to~$\cO(\epsilon^{-2T})$ for nonsmooth integrands and to~$\cO(\epsilon^{-(T+1)})$ for smooth ones. This work, however, provides only asymptotic rates without explicit constants. Moreover, it only studies mean squared errors but provides no uniform convergence analysis, which would be necessary to ensure that the minimizers of~\eqref{prob:min_hatF(x)} approximate those of~\eqref{problem:MCCO} with high probability. To the best of our knowledge, only two papers develop MLMC estimators for MCCE with~$T>2$. The first focuses on optimal stopping problems and imposes restrictive assumptions on the integrands \citep{zhou2022unbiased}. The second primarily considers smooth integrands. While it also discusses nonsmooth integrands, the corresponding MLMC estimators have infinite variance, making problem~\eqref{prob:min_hatF(x)} numerically unstable \citep{syed2023optimal}. No prior work develops MLMC estimators for MCCE with~$T>2$ stages and general nonsmooth integrands.

To our best knowledge, all existing work on MCCO assumes that~$T=2$. In this case, MCCO reduces to conditional stochastic optimization, which arises in many applications. In machine learning, for instance, it appears in reinforcement learning \citep{dai2018sbeed, hu2020sample}, model-agnostic meta-learning, robust supervised learning or invariant learning~\citep{hu2020sample, hu2021bias}. In statistics and econometrics, applications include instrumental variable regression \citep{bennett2019deep, goda2023constructing, hartford2017deep, muandet2020dual, singh2019kernel,chen2024stochastic} or Bayesian experimental design \citep{goda2020multilevel, goda2022unbiased, beck2020multilevel}. The simplest approach to MCCO is to approximate~$F(x)$ by an SAA estimator~$\hat F(x)$ and to solve~\eqref{prob:min_hatF(x)}. For~$T=2$, the number of scenarios needed to guarantee that the resulting minimizer of~\eqref{prob:min_hatF(x)} constitutes an $\epsilon$-optimal solution to~\eqref{problem:MCCO} with high probability amounts to~$\mathcal{O}(\epsilon^{-4})$ if the integrands are nonsmooth and can be reduced to~$\mathcal{O}(\epsilon^{-3})$ if the integrands are smooth \citep{hu2020sample}. However, solving~\eqref{prob:min_hatF(x)} may be computationally infeasible without restrictive convexity or monotonicity assumptions. To address this, subsequent work develops SAA estimators for~$\nabla F(x)$ that can be employed within a stochastic gradient descent algorithm for solving problem~\eqref{problem:MCCO} to stationarity. For smooth integrands, this method requires~$\cO(\epsilon^{-6})$ scenarios to find an $\epsilon$-stationary point, or~$\cO(\epsilon^{-5})$ with standard variance reduction techniques~\citep{hu2020biased}. MLMC estimators for $\nabla F(x)$ introduced in~\citep{hu2024multi} further reduce this to $\cO(\epsilon^{-4})$ scenarios. MCCO with $T>2$ stages has not yet been investigated.


The remainder of the paper is structured as follows. Section~\ref{sec:motivational_ex} illustrates motivating applications of MCCE and MCCO. Sections~\ref{sec:mcco_estimation} and~\ref{sec:mcco_optimization} then introduce and analyze SAA and MLMC estimators for~$F(x)$ and~$\nabla F(x)$, respectively. Numerical results are presented in Section~\ref{sec:experiments}, and all proofs as well as some auxiliary technical results are relegated to the appendix.

\paragraph{Notation} The entrywise $p$-norm of a matrix $A\in\RR^{n\times m}$ is denoted by $\|A\|_p \coloneqq (\sum_{j=1}^m\sum_{i=1}^n| A_{ij}|^p)^{1/p}$. If $p=2$, $\|A\|_p$ reduces to the Frobenius norm. We use~$\NN_0$ and~$\NN$ to denote the sets of natural numbers with and without~$0$, respectively, and for any~$n\in\NN$, we set~$[n]=\{1,2,\ldots,n\}$. The transpose of the Jacobian of a differentiable function $f:\RR^n\rightarrow \RR^m$ is denoted by $\nabla f:\RR^n\rightarrow \RR^{n\times m}$. If~$m=1$, $\nabla f$ thus reduces to the gradient of~$f$. All (in)equalities involving random variables are understood to hold almost surely. The variance of a random variable $\xi$ is denoted by~$\VV(\xi)$. We designate estimators by the superscript~‘\;$\hat{}$\;’ and define~$\CC(\hat{F})$ as the number of scenarios (sample paths) needed to construct an estimator~$\hat{F}$. The truncated geometric distribution with rate~$r\in(0,1)$ and truncation point~$M\in\mathbb N$ is denoted by~$\text{Geo}(r|M)$. Thus, a random variable $\lambda$ governed by $\text{Geo}(r|M)$ satisfies $\PP(\lambda=\ell)=r(1-r)^\ell /(1-(1-r)^{M+1})$ for every $\ell=0,\ldots, M$. If~$a(\epsilon)$ and~$b(\epsilon)$ are real-valued functions of a tolerance~$\epsilon>0$, then~$a(\epsilon)=\cO(b(\epsilon))$ means that there exist~$C>0$ and~$\bar\epsilon>0$ such that $a(\epsilon)\leq C b(\epsilon)$ for all~$\epsilon\in(0,\bar\epsilon)$. Note that $a(\epsilon)$ and $b(\epsilon)$ could depend on other variables. However, $\cO(\cdot)$ captures only their asymptotic dependence on~$\epsilon$.

\section{Motivating Applications}
\label{sec:motivational_ex}
Instances of~\eqref{problem:MCCE} and~\eqref{problem:MCCO} with~$T>2$ stages arise in numerous applications. Examples include optimal stopping, linear-quadratic regulator and finite-state stochastic control problems for~\eqref{problem:MCCE} as well as policy learning in contextual bandits and risk measurement with nested risk measures for~\eqref{problem:MCCO}.

\paragraph{Optimal Stopping}
\label{sec:application_optimal_stopping} 
Suppose that an agent observes a sequence of random variables $\xi_t$, $t\in[T]$, which are revealed one at a time, and let $g_t: \RR^{m_t}\rightarrow \RR$, $t\in[T]$, be Borel measurable payoff functions. At each time~$t\in[T]$, the agent must decide whether to stop observing, in which case an immediate payoff~$g_t(\xi_t)$ is received, or to continue observing, in which case no immediate payoff is received. The process must stop at time~$T$ at the latest. If the goal is to maximize the expected payoff, then the optimal stopping time can be found via dynamic programming \citep{zhou2022unbiased}. To this end, use $U_t(\xi_{[t]})$ to denote the agent's expected future payoff conditional on all observations up to time~$t$. One readily verifies that $U_T(\xi_{[T]})=g_T(\xi_T)$ and
\begin{align*}
    U_{t}(\xi_{[t]}) = \max \left\{g_t(\xi_{t}), \EE_{t} \left[U_{t+1}(\xi_{[t+1]}) \right]\right\} \quad \forall t\in[T-1].
\end{align*} 
Unraveling this recursion yields
\[
    \EE\left[U_1(\xi_{1})\right]=\EE \bigg[\max\bigg\{g_1(\xi_1), \EE_{1}\Big[\max\Big\{g_2(\xi_2), \ldots, \EE_{T-1}[ g_T(\xi_T)] \ldots  \Big\}   \Big] \bigg\}  \bigg].
\]
Thus, the problem of computing the expected value of the optimal stopping time can be viewed as an instance of~\eqref{problem:MCCE} with $f_T({\xi_T},x_T)=g_T(\xi_T)$ and $f_t(\xi_t,x_t)=\max\{g_t(\xi_t),x_t\} $ for all $t\in[T-1]$. Optimal stopping problems emerge in various applications across many domains~\citep{ciocan2022interpretable}. For example, they arise in the pricing of American-style options~\citep{zhou2022unbiased} or in sequential hypothesis testing~\citep{lai1973optimal}. In addition, they also arise in supply chain management~\citep{van2023intuitive} to determine the timing of inventory replenishment~\citep{ozyoruk2022end}, and in game theory to model the reasoning about the probabilistic strategies of players~\citep{rainforth2018nesting}. 

\paragraph{Linear-Quadratic Regulator}
Consider the finite-horizon stochastic linear-quadratic regulator problem
\begin{equation}
	\label{eq:lqr}
	\begin{array}{cl}
	\displaystyle \min_{a,s} & \displaystyle \mathbb E\left[\sum_{t=1}^{T-1} \left(s_t^\top Q s_t+a_t^\top Ra_t \right) +s_T^\top P_T s_T\right] \\
	\text{s.t.} & s_{t+1}= A s_t + B a_t + \xi_{t+1}\quad\forall  t\in[T-1] \\ & a_t \text{ is $\mathcal F_t$-measurable}\quad\forall  t\in[T-1],
	\end{array}
\end{equation}
where $s_t\in\mathbb{R}^m$ and $a_t\in\mathbb{R}^n$ denote the state and the control action at time $t$. We assume that the system matrices $A\in\mathbb{R}^{m\times m}$ and $B\in\mathbb{R}^{m\times n}$ as well as the symmetric positive semidefinite cost matrices $Q\in\mathbb R^{m\times m}$, $R\in\mathbb R^{n\times n}$ and $P_T\in\mathbb R^{m\times m}$ are deterministic and that $R$ is invertible. We also assume that~$s_1=\xi_1$ and that $\{\xi_t\}_{t\in \mathbb N}$ is a (possibly serially dependent) stochastic process with finite second moments taking values in $\mathbb{R}^m$. Control problems of the form~\eqref{eq:lqr} are ubiquitous in economics~\cite{hansen2008robustness}, engineering~\cite{bertsekas1995dynamic} or neuroscience~\cite{todorov2002optimal}, among many others. Traditionally, the stochastic disturbances are assumed to be i.i.d. However, in an era where control systems must navigate continuously shifting environments and interact with adaptive, data-driven agents, models that abandon the fragile i.i.d.\ noise assumption are essential for ensuring robustness under real-world nonstationarity and persistent distribution shift.

The Q-function at time~$t$ associated with problem~\eqref{eq:lqr} is defined via the Bellman-type recursion 
\begin{equation}
	\label{eq:Qdef}
	Q_t(s,a) = s^\top Q s + a^\top R a + \mathbb{E}_t \left[ \min_{a'\in\mathbb R^n} Q_{t+1}(As+Ba+\xi_{t+1},a') \right]
\end{equation}
with terminal value $Q_T(s,a)=s^\top P_T s$ (which is deterministic and independent of $a$); see also~\cite[\S~4.1]{bertsekas1995dynamic}. Due to the serial dependencies of the noise process, $Q_t(s,a)$ is an $\mathcal F_t$-measurable random object for any fixed~$(s,a)$. But we notationally suppress the dependence of $Q_t(s,a)$ on $\xi_{[t]}$ to avoid clutter. The optimal value of~\eqref{eq:lqr} can be expressed in terms of the time-$1$ Q-function as
\begin{equation}
	\label{eq:min-lqr}
	\mathbb E\left[ \min_{a\in\mathbb R^n} Q_1(\xi_1,a)\right].
\end{equation}
The following lemma shows that the Q-functions are quadratic in the states and actions with random coefficients. To keep the paper self-contained, we sketch the proof of this elementary result in Appendix~\ref{sec:proofs-section-2}.

\begin{lemma}
\label{lem:Q-lqr}
For each $t\in [T]$, the Q-function is quadratic in $(s,a)$ and can be represented as
\begin{subequations}
\begin{align}
	\label{eq:Q-Bellman}
	Q_t(s,a) = 
	\begin{bmatrix}s\\ a\end{bmatrix}^{\top}
	\begin{pmatrix}
	Q_t^{ss} & Q_t^{sa} \\
	(Q_t^{sa})^\top & Q_t^{aa}
	\end{pmatrix}
	\begin{bmatrix}s\\ a\end{bmatrix} 
	+2 \begin{bmatrix}b_t^s \\ b_t^a \end{bmatrix}^{\!\top}
	\begin{bmatrix}s\\ a\end{bmatrix}
	+c_t \quad \mathbb P\text{-a.s.},
\end{align}
where the $\mathcal F_t$-measurable coefficient matrices and vectors are defined through the backward recursions
\begin{equation}
\label{eq:Q-recursion}
\begin{aligned}
	Q_t^{ss} &= \mathbb{E}_t\big[Q + A^\top P_{t+1}A\big]\\
	Q_t^{sa} &= \mathbb{E}_t\big[A^\top P_{t+1} B\big]\\
	Q_t^{aa} &=  \mathbb{E}_t\big[R +B^\top P_{t+1} B\big]\\
	b_t^s &= \mathbb{E}_t\big[A^\top P_{t+1} \xi_{t+1} + A^\top g_{t+1} \big]\\
	b_t^a &= \mathbb{E}_t\big[B^\top P_{t+1} \xi_{t+1} + B^\top g_{t+1}\big]\\
	c_t &= \mathbb{E}_t\big[\xi_{t+1}^\top P_{t+1} \xi_{t+1} + 2\,g_{t+1}^\top \xi_{t+1} + d_{t+1} \big]
\end{aligned}
\end{equation}
for all $t\in[T-1]$ initialized by 
\begin{equation}
	\label{eq:Q-initialization}
	Q_T^{ss} =P_T,  \quad Q_T^{sa} =0, \quad Q_T^{aa} =0, \quad b_T^s = 0, \quad b_T^a = 0 \quad\text{and} \quad  c_T = 0.
\end{equation}
The underlying auxiliary variables are defined by
\begin{equation}
	\label{eq:Q-Schur}
	P_{t} = Q_{t}^{ss} - Q_{t}^{sa}(Q_{t}^{aa})^{-1}(Q_{t}^{sa})^\top,\quad g_{t} = b_{t}^s - Q_{t}^{sa}(Q_{t}^{aa})^{-1}b_{t}^a  \quad\text{and} \quad  d_{t} = c_{t} - (b_{t}^a)^\top (Q_{t}^{aa})^{-1} b_{t}^a
\end{equation}
\end{subequations}
for all $t\in[T-1]$, while $P_T$ is the terminal state cost matrix, $g_T=0$ and $d_T=0$.
\end{lemma}

For all $t\in[T]$ we define an aggregate coefficient vector encoding the time-$t$ Q-function as
\[
	x_t = \bigl(Q_t^{ss},\,Q_t^{sa},\,Q_t^{aa},\,b_t^s,\,b_t^a,\,c_t\bigr) \in\mathbb R^{d},
\]
where the components are the $\mathcal F_t$-measurable matrices and vectors appearing in~\eqref{eq:Q-Bellman}, and where $d=m^2+mn+n^2+m+n+1$. Also, for all $t=2,\ldots,T$ we define $f_t:\mathbb{R}^{m}\times\mathbb R^d \to\mathbb R^d$ through
\[
	f_t(\xi_t,x_t) = \Bigl( Q + A^\top P_t A,\; A^\top P_t B,\; R + B^\top P_t B,\; A^\top P_t \xi_t + A^\top g_t,\; B^\top P_t \xi_t + B^\top g_t,\; \xi_t^\top P_t \xi_t + 2g_t^\top\xi_t + d_t \Bigr),
\]
where $P_t$, $g_t$ and $d_t$ are defined as in~\eqref{eq:Q-Schur} and thus constitute rational functions of~$x_t$. In addition, we define $f_1:\mathbb R^m\times\mathbb R^d\to\mathbb R$ through
\[
	f_1(\xi_1,x_1) = \xi_1^\top \Big(Q_1^{ss}-Q_1^{sa}(Q_1^{aa})^{-1}(Q_1^{sa})^\top\Big) \xi_1 +2\Big(b_1^s - Q_1^{sa}(Q_1^{aa})^{-1}b_1^a\Big)^\top \xi_1 + \Big(c_1 - (b_1^a)^\top (Q_1^{aa})^{-1} b_1^a\Big).
\]
The backward recursions~\eqref{eq:Q-recursion} for the Q-function coefficients now take the compact form
\[
	x_{t} = \mathbb{E}_{t}\!\bigl[f_{t+1}(\xi_{t+1},x_{t+1})\bigr] \quad \forall t\in[T-1]
\]
with terminal condition $x_T=(P_T,0,0,0,0,0)$. Using similar arguments as in the proof of Lemma~\ref{lem:Q-lqr}, one readily verifies that the optimal value~\eqref{eq:min-lqr} of the control problem~\eqref{eq:lqr} can be represented compactly as $\mathbb E[f_1(\xi_1,x_1)]$. In view of the above recursion, this implies that the problem of computing the optimal value of~\eqref{eq:lqr} is equivalent to solving problem~\eqref{problem:MCCE}, where $x$ is identified with the deterministic vector~$x_T$. 

A similar reasoning as shown here can be used to reduce linear exponential quadratic regulator and risk-sensitive control problems~\cite{whittle1990risk} to instances of~\eqref{problem:MCCE}. We omit the derivations for brevity.

\paragraph{Stochastic Control with Finite State and Action Spaces}
We use the same symbols as in the last example to denote states, actions and disturbances and consider the finite-horizon stochastic control problem
\begin{equation}
	\label{eq:control}
	\begin{array}{cl}
	\displaystyle \min_{a,s} & \displaystyle \mathbb E\left[\sum_{t=1}^{T-1} c_t(s_t,a_t) + c_T(s_T)\right] \\
	\text{s.t.} & a_t \in \cA_t(\xi_t) \quad \forall t\in[T-1]\\
    & s_{t+1}= \psi(s_t,a_t,\xi_{t+1}) \quad\forall  t\in[T-1] \\ & a_t \text{ is $\mathcal F_t$-measurable}\quad\forall  t\in[T-1].
	\end{array}
\end{equation}
All states and actions take values in finite sets~$\cS$ and~$\cA$, respectively, while $\{\xi_t\}_{t\in \mathbb N}$ represents a generic (possibly serially dependent) stochastic process of disturbances taking values in $\mathbb{R}^m$. The system evolves according to a generic state transition function $\psi : \cS \times \cA \times \RR^m \to \cS$, and at each time~$t$, the controller must select an action from a 
generic random and time-varying feasible set~$\cA_t(\xi_t)\subseteq \cA$. The stage-costs  $c_t : \mathcal S \times \mathcal A \to \mathbb R$, $t\in[T-1]$, and the terminal cost $c_T : \mathcal S \to \mathbb R$ are also captured by generic functions. 

If the disturbances are serially independent, then problem~\eqref{eq:control} reduces to a Markov decision process, a problem class that occupies central stage in operations research and machine learning~\cite{bertsekas1995dynamic}. Allowing the disturbances to display serial correlations enlarges this problem class significantly.

The time-$t$ Q-function of problem~\eqref{eq:control} is defined by the Bellman-type recursion
\begin{equation}\label{eq:Q-bellman}
 Q_t(s,a)=c_t(s,a) + \EE_t\left[\min_{a^\prime\in\cA_{t+1}(\xi_{t+1})} Q_{t+1}\left(\psi(s,a', \xi_{t+1}),a' \right)\right]
\end{equation}
with terminal condition $Q_T(s,a) = c_T(s)$. As in the previous example, $ Q_t(s,a)$ is an $\cF_t$-measurable random object for any fixed~$(s,a)$, but we suppress its dependence on~$\xi_{[t]}$ to avoid clutter. The optimal value of problem~\eqref{eq:control} can be expressed in terms of the time-$1$ $Q$-function as
\begin{equation}\label{eq:min-control}
    \EE\left[\min_{a\in\cA_1(\xi_1)} Q_1(\xi_1,a)\right].
\end{equation}
For all $t\in[T]$ we define~$x_t \in\RR^{d}$ with $d=|\cS|\times|\cA|$ as the vectorized time-$t$ $Q$-function with components $(x_t)_{(s,a)}= Q_t(s,a)$, $(s,a)\in\cS\times\cA$. Also, for all $t=2,\ldots,T$, we define $f_{t}:\RR^{m}\times \RR^{d}\rightarrow\RR^{d}$ through
\[
    \big( f_{t}(\xi_{t}, x_{t})\big)_{(s,a)} = c_{t-1}(s,a)  + \min_{a^\prime \in \mathcal A_{t}(\xi_{t})} x_{t}\big(\psi(s,a,\xi_{t}),\,a^\prime\big)\quad \forall (s,a)\in \mathcal S \times \mathcal A.
\]
In addition, we define $f_1:\RR^{m}\times \RR^{d}\rightarrow \RR$ through $f_1 (\xi_1,x_1) =  \min_{a^\prime \in \mathcal A_{1}(\xi_{1})} x_{1}(s,a^\prime)$, where we use again the convention that $s=\xi_1$. With this notation, the Bellman recursions~\eqref{eq:Q-bellman} can be written compactly as
\[
    x_t = \mathbb E_t\big[f_{t+1}(\xi_{t+1}, x_{t+1})\big] \quad \forall t \in [T-1],
\]
with terminal condition $(x_T)_{(s,a)}=c_T(s)$ for every $s\in\cS$ and $a\in\cA$. The optimal value~\eqref{eq:min-control} of problem~\eqref{eq:control} thus coincides with $\EE[f_1(\xi_1,x_1)]$, and computing this value is equivalent to solving problem~\eqref{problem:MCCE}.

\paragraph{Distributionally Robust Policy Learning for Offline Contextual Bandits}
\label{sec:application_dro}
    The contextual bandit problem is a sequential decision problem under uncertainty. In each round, an agent first observes a context $c\in\mathcal C$, then selects an action $a\in\mathcal A$ (metaphorically thought of as an arm of a slot machine) and finally incurs a random cost $y_a\in\mathcal Y$ associated with the chosen action. The distribution~$\mathbb Q_{y_a|c}$ of~$y_a$ conditional on~$c$ is unknown to the agent. However, the pairs $(c,y)$ with $y=(y_a)_{a\in\mathcal A}$ are known to be independent and identically distributed across rounds. The crux of multi-armed bandit problems is that, in each round, the agent observes only the cost of the chosen action but not those of the other actions. Thus, each action $a\in\mathcal A$ must be selected repeatedly for each given~$c$ in order to learn the distribution $\mathbb Q_{y_a|c}$. 
    
    Consider now a parametric family of policies~$\pi_\theta$, $\theta\in\Theta$, which map any context~$c\in\mathcal C$ to a probability distribution over the actions~$a\in\mathcal A$. We will henceforth assume that~$\mathcal A$ is finite. Thus, given a context~$c$, the policy~$\pi_\theta$ selects action~$a$ with probability~$\pi_\theta(a|c)$. The agent aims to find~$\theta\in\Theta$ such that $\pi_\theta$ minimizes the average cost. If the joint distribution~$\mathbb Q_{(c,y)}$ of~$c$ and~$y$ was known, this could be achieved by solving
    \begin{align*}
        \min_{\theta\in\Theta} \int_{\mathcal C\times\mathcal Y} \sum_{a\in\mathcal A} y_a \, \pi_\theta(a|c) \, {\rm d}\mathbb Q_{(c,y)}(c,y) = \min_{\theta\in\Theta} \int_{\mathcal C} \sum_{a\in\mathcal A} \pi_\theta(a|c) \int_{\mathcal Y} y_a \,{\rm d}\mathbb Q_{y|c}(y) \, {\rm d}\mathbb Q_{c}(c).
    \end{align*}
    In off-policy learning, however, $\mathbb Q_{(c,y)}$ is unknown, and the agent has only access to context, action and cost data collected under a fixed behavioral policy. Any estimator~$\PP_{(c,y)}$ constructed from this data invariably differs from~$\mathbb Q_{(c,y)}$. To combat the detrimental effects of estimation errors, it has been proposed to solve the following distributionally robust variant of the off-policy learning problem~\citep{shen2024wasserstein}.
    \begin{align}
        \label{eq:dro-offline-bandits}
        \min_{\theta\in\Theta} \sup_{\substack{\mathbb Q_c\in\mathcal P(\mathcal C): \\ d(\mathbb Q_c,\PP_c)\leq r_c}} \int_{\mathcal C} \sum_{a\in\mathcal A} \pi_\theta(a|c) \sup_{\substack{\mathbb Q_{y|c}\in\mathcal P(\mathcal Y): \\ d(\mathbb Q_ {y|c},\PP_{y|c})\leq r_y}} \int_{\mathcal Y} y_a \,{\rm d}\mathbb Q_{y|c}(y) \, {\rm d}\mathbb Q_{c}(c)
    \end{align}
    Here, we use the 2-Wasserstein metric $d(\cdot,\cdot)$ to measure distances between probability distributions. Note that the distributionally robust off-policy learning problem replaces the expectations with respect to the unknown distributions $\mathbb Q_c$ and $\mathbb Q_{y|c}$ by worst-case expectations over all distributions in an $r_c$-neighborhood around~$\PP_c$ and an $r_y$-neighborhood around~$\PP_{y|c}$, respectively. By~\cite[Proposition~6.20]{kuhn2024distributionally}, the maximization problem over~$\mathbb Q_{y|c}$ in~\eqref{eq:dro-offline-bandits} can be solved in closed form, and its supremum equals $\int_{\mathcal Y} y_a\, {\rm d} \PP_{y|c}(y) +r_y$. In addition, by~\cite[Theorem~4.18]{kuhn2024distributionally}, the maximization problem over~$\mathbb Q_c$ can be dualized and thus be converted to a minimization problem over a scalar dual decision variable~$\lambda$. Hence, problem~\eqref{eq:dro-offline-bandits} can be recast as
    \begin{align}
        \label{eq:dro-offline-bandits-dual}
        \min_{\theta\in\Theta,\lambda\geq 0} \; \int_{\mathcal C} \sup_{u\in\mathcal C} \left(\sum_{a\in\mathcal A} \pi_\theta(a|u) \int_{\mathcal Y} y_a\, {\rm d} \PP_{y|c=u}(y) +r_y + r_c^2\lambda -\lambda\|u-c'\|_2^2 \right) {\rm d}\PP_{c}(c').
    \end{align}
    Note that~\eqref{eq:dro-offline-bandits-dual} involves a maximization problem over all contexts~$u\in\mathcal C$, which may be difficult to solve because the expected reward of action~$a$ conditional on~$c=u$ is a measurable function of~$u$ and because maximizing a generic measurable function is hard. As a remedy, the supremum over all $u\in\mathcal C$ can be approximated by a softmax function involving the uniform distribution~$\PP_u$ on~$\mathcal C$. Thus, \eqref{eq:dro-offline-bandits-dual} simplifies to
    \begin{equation}
    \label{eq:dro-offline-bandits-dual-softmax}
    \begin{aligned}
        & \min_{\theta\in\Theta,\lambda\geq 0} \; \int_{\mathcal C} \frac{1}{\mu} \log \Bigg(\int_{\mathcal C} \exp\Bigg( \mu \Bigg(\sum_{a\in\mathcal A} \pi_\theta(a|u) \int_{\mathcal Y} y_a\, {\rm d} \PP_{y|c=u}(y) \\
        & \hspace{7cm} +r_y +r_c^2\lambda -\lambda\|u-c'\|^2_2 \Bigg) \Bigg) {\rm d} \PP_{u}(u) \Bigg) {\rm d}\PP_{c}(c'),
    \end{aligned}
    \end{equation}
    where the parameter~$\mu>0$ controls the smoothness of the softmax function. For further details see~\citep{shen2024wasserstein}. Problem~\eqref{eq:dro-offline-bandits-dual-softmax} is readily recognized as an instance of~\eqref{problem:MCCO} with~$T=3$, $x=(\theta,\lambda)$, $\mathcal X=\Theta\times \RR_+$, $\xi_1=c'$, $\xi_2=u$ and~$\xi_3=(c',u, y)$. The integrands are given by $f_1(\xi_1,x_1)= \frac{1}{\mu} \log(x_1)$, $f_2(\xi_2,x_2)= \exp(\mu x_2)$ and
    \[
        f_3(\xi_3,x) = \sum_{a\in\mathcal A} \pi_\theta(a|u)\, y_a +r_y +r_c^2\lambda -\lambda\|u-c'\|^2_2.
    \]
    Variants of the distributionally robust off-policy learning problem~\eqref{eq:dro-offline-bandits} use the 2-Wasserstein distance with a Kullback-Leibler regularization term \citep{azizian2023regularization} or the Sinkhorn divergence \citep{wang2021sinkhorn} instead of the plain 2-Wasserstein distance to measure the discrepancy of different distributions. These alternative formulations also give rise to instances of MCCO with $T=3$. For further details we refer to~\citep[Appendix~D]{wang2021sinkhorn}.

\paragraph{Nested Risk Measures}
\label{sec:application_nested_risk}
Consider a stream of random costs $\xi_t$, for $t \in [T]$, that are $p$-integrable for some exponent $p \geq 1$ ({\em i.e.}, $\mathbb{E}[\|\xi_t\|^p_2] < \infty$ for all $t \in [T]$). In this situation, it makes sense to assign the total cost $\sum_{t=1}^T \xi_t$ a risk that depends on the temporal structure of the underlying stochastic process. This can be done by using a nested dynamic risk measure $\rho = \rho_0 \circ \cdots \circ \rho_{T-1}$, where $\rho_t: \mathcal{L}^p_{t+1} \to \mathcal{L}^p_t$ is a conditional risk mapping and $\mathcal{L}^p_t$ represents the family of all $p$-integrable functions of the history $\xi_{[t]}$ of costs up to time $t$. The defining properties of a conditional risk mapping are translation invariance, convexity and monotonicity; see~\cite{ruszczynski2006conditional} for details. By exploiting these properties, the risk of the total cost can then be represented as
\begin{equation*}
    \rho\left[ {\textstyle\sum_{t=1}^T \xi_t }\right] = \rho_0\left[ \xi_1 + \rho_1\left[ \xi_2 + \cdots +\rho_{T-1}[\xi_T] \cdots\right] \right].  
\end{equation*}
A widely studied choice is to set $\rho_t$ to a conditional entropic risk measure, defined as
\[
    \rho_{t}(\xi_{t+1}) = \frac{1}{{\mu}_{t}} \log\left(\mathbb{E}_{t}[\exp(-{\mu}_{t} \xi_{t+1} )] \right),
\]
where $\mu_t > 0$ represents the risk aversion level at time $t$. Notably, the risk aversion level may vary over time. Such scenarios arise, for instance, in pension planning, where the goal is to assess the risk of a pension contract. In this context, tolerance for uncertainty in cash flows typically decreases over time as retirement approaches. It is now easy to see that evaluating $\exp(\mu_0 \rho(\sum_{t=1}^T \xi_t))$ reduces to an instance of~\eqref{problem:MCCE}, with $f_t(\xi_t,x_t)=\exp(-\mu_{t-1} (\xi_t + \mu_{t}^{-1} \log(x_t)))$ for $t\in[T-1]$ and $f_T(\xi_T,x)=\exp(-\mu_{T-1} \xi_T)$. Similarly, optimization problems involving nested risk measures can often be reformulated as instances of~\eqref{problem:MCCO}. Such problems arise, for example, in Bayesian distributionally robust optimization \citep{shapiro2023bayesian} and in the analysis of Bayesian risk Markov decision processes \citep{lin2022bayesian}, which emerge in betting or control problems \citep{lin2022bayesian, wang2024bayesian}. Solving Bayesian risk Markov decision processes presents significant challenges due to the complexity of computing nested risk functionals \citep{lin2022bayesian}. To address these difficulties, approximation techniques are commonly used. However, traditional estimation and optimization methods can be computationally expensive. This paper proposes cost-efficient methods to mitigate this issue and resolves an open question raised in~\cite{lin2022bayesian}.


\section{Numerical Solution of MCCE Problems}
\label{sec:mcco_estimation}

We begin by presenting the fundamental assumptions that underpin this paper.
\begin{assumption}[Basic Regularity Conditions] 
\label{assump:general}
\leavevmode\begin{enumerate}[label=(\roman*)]
    \item The integrand $f_t(\xi_t,x_t)$ is Borel-measurable in~$\xi_t$, and $\EE[\|f_t(\xi_t,0)\|_2]<\infty$ for every $t\in[T]$. 
    \item The feasible set $\mathcal{X}$ is convex and compact. Thus, it has a finite diameter~$D_\mathcal{X}$.
\end{enumerate}
\end{assumption}

Neither $F$ nor the underlying integrands $f_t, t\in[T]$, are assumed to display any convexity properties. However, we assume these functions to be Lipschitz continuous, which is common practice in sample complexity analysis (see {\em e.g.}, \citep{shapiro2005complexity}). Sometimes, we additionally assume the integrands to be smooth. 

\begin{assumption}[Lipschitz Continuity] 
\label{assump:lipschitz}
   $f_t(\xi_t,x_t)$ is $L_t$-Lipschitz continuous in $x_t$ uniformly for all $\xi_t$, i.e.,
    $$ \|f_t(\xi_t,x_t) -  f_t(\xi_t,y_t)\|_2 \leq L_t\|x_t-y_t\|_2 \quad \forall x_t,y_t \in \RR^{d_t},\; \forall \xi_t\in\RR^{m_t}.$$
\end{assumption}

\begin{assumption}[Smoothness] 
\label{assump:smooth}
 $f_t(\xi_t,x_t)$ is $S_t$-smooth in $x_t$ for across all $\xi_{t}$, i.e., 
  $$\|\nabla_{x_t} f_t(\xi_t,x_t) - \nabla_{x_t} f_t(\xi_t,y_t) \|_2\leq S_t\|x_t-y_t\|_2\quad \forall x_t,y_t\in \RR^{d_t},\; \forall \xi_t\in\RR^{m_t}.$$
\end{assumption}
Note that if Assumption~\ref{assump:smooth} holds, then the fundamental theorem of calculus implies that 
\begin{align}\label{eq:inequality_for_smooth_funs}
    \left\|f_t(\xi_t,x_t)-f_t(\xi_t,y_t)-\nabla_{x_t} f_t(\xi_t,y_t)^\top(x_t-y_t) \right\|_2\leq \frac{S_{t}}{2}\|x_t-y_t\|_2^2\quad \forall x_t,y_t\in\RR^{d_t},\;\forall \xi_t\in\RR^{m_t}.
\end{align}

Assumptions~\ref{assump:general} and~\ref{assump:lipschitz} ensure that
\eqref{problem:MCCE} and~\eqref{problem:MCCO} are well-defined. Indeed, they imply that all conditional expectations in the definition of~$F(x)$ exist and are finite, which is an immediate consequence of Lemma~\ref{lem:well-defined} in the appendix. Thus, problem~\eqref{problem:MCCE} is well-defined. In addition, Assumptions~\ref{assump:general} and~\ref{assump:lipschitz} imply that $F(x)$ is Lipschitz continuous with Lipschitz constant~$\prod_{t=1}^T L_t$. This is an immediate corollary of Lemma~\ref{lem:lipschitz} in the appendix. By Assumption~\ref{assump:general}\,(ii), we can thus invoke Weierstrass' extreme value theorem to conclude that the minimum in~\eqref{problem:MCCO} is attained. Hence, problem~\eqref{problem:MCCO} is well-defined.

Problem \eqref{problem:MCCE} is challenging due to the inherent difficulty of high-dimensional integration. Consequently, attempting to compute~$F(x)$ {\em exactly} is overly ambitious. In this section we thus pursue the more modest but achievable goal of {\em approximating}~$F(x)$ with an estimator~$\hat F(x)$ constructed from samples. Specifically, we will construct an SAA estimator in Section~\ref{sec:saa} and an MLMC estimator in Section~\ref{sec:rtmlmc}. We will show that both estimators converge to~$F(x)$ in mean squared error and in probability, uniformly over all~$x \in \cX$. Moreover, we will demonstrate that the MLMC estimator outperforms the SAA estimator in terms of both mean squared error-based and high-probability scenario complexity.

\subsection{SAA Estimator}
\label{sec:saa}
Perhaps the most straightforward estimator for $F(x)$ is the SAA estimator defined below.

\begin{defn}[SAA Estimator]
\label{def:MCCO_sample_average_approximation}
We construct a forest of~$n_1$ scenario trees as follows. The root nodes of the trees are indexed by $i_1\in[n_1]$. The children of $i_1$ are indexed by $\itwo=(i_1,i_2)$ for $i_2\in[n_2]$, and the children of $\itwo$ are indexed by $\ithree=(\itwo,i_3)$ for $i_3\in[n_3]$, etc. Thus, the multi-index $\itt=(i_1,i_2,\ldots,i_t)$ of a stage-$t$ node encodes the tree that accommodates the node $\itt$ as well as the path from the root node in that tree to the node $\itt$; see Figure~\ref{fig:scenario-tree}. The SAA estimator for $F(x)$ is then given by
\begin{subequations}
\begin{equation}\label{eq:saa_estimator}
    \hat{F}(x)\coloneqq \frac{1}{n_1}\sum_{i_1=1}^{n_1} \hat F_{i_1}(x),
\end{equation}
where the term corresponding to the $i_1$-th scenario tree is set to
\begin{equation}
    \label{eq:sub-saa_estimator}
    \hat F_{i_1}(x)\coloneqq f_1\left(\xi_1^{i_1},\frac{1}{n_2}\sum_{i_{2}=1}^{n_2} f_2\left(\xi_2^{\itwo},\ldots, \frac{1}{n_T}\sum_{i_T=1}^{n_T} f_T\left({\xi_T^{\iT}},x\right)\right)\ldots\right)\quad\forall i_1\in[n_1],
\end{equation}
\end{subequations}
and where the sample $\xi_t^{\itt}$ associated with node $\itt$ is constructed by the forward recursion
\[
    \xi_t^{\itt} \, \stackrel{\text i.i.d.}{\sim} \, \PP_{\xi_t}\left(\cdot \left|\xi_1=\xi_1^{i_1} ,\ldots, \xi_{t-1}=\xi_{t-1}^{i_{[t-1]}} \right. \right) \quad \forall i_1\in[n_1],\ldots,\forall i_t\in[n_t],~\forall t\in[T].
\]
\end{defn}
We introduce some key notations used throughout the paper. Given any~$t\in\mathbb N_0$ with $t+1\in[T]$ and any Borel-measurable function~$\psi(\xi_{[t+1]})$, we henceforth use $\hat\EE_t[\psi(\xi_{[t+1]})]$ as shorthand for the sample average of~$\psi(\xi_{[t+1]})$ over~$n_{t+1}$ samples $\xi_{t+1}^{i}$, $i\in[n_{t+1}]$, drawn independently from the marginal distribution of~$\xi_1$ (if $t=0$) or the conditional distribution of $\xi_{t+1}$ given $\xi_{[t]}$ (if $t>0$). That is, we set 
\begin{equation}\label{eq:expected_psi}
    \hat\EE_t[\psi(\xi_{[t+1]})] \coloneqq \frac{1}{n_{t+1}} \sum_{i=1}^{n_{t+1}} \psi(\xi_{[t]},\xi_{t+1}^{i}).
\end{equation}
Below, we sometimes use $\hat\EE[\cdot]$ as shorthand for $\hat\EE_0[\cdot]$.
Using these conventions, the SAA estimator of Definition~\ref{def:MCCO_sample_average_approximation} can be recast more compactly as
\[
    \hat F(x) = \hat \EE \Big[f_1\Big(\xi_1, \hat \EE_{1} \big[f_2\big(\xi_2,\ldots \hat \EE_{T-1} \left[f_T\left(\xi_T,x\right)\right] \ldots\big) \big]\Big)\Big].
\]

\begin{figure}[htb]
    \centering
\begin{tikzpicture}[->,>={Stealth[round,sep]},auto, level distance=1.1cm,
level 1/.style={sibling distance=5cm},
level 2/.style={sibling distance=1.8cm},
level 3/.style={sibling distance=0.5cm}
]

\node {$\xi_1^{i_1}$}
  child {
    node (a) {$\xi_2^{(i_1,1)}$}
      child {node (a1) {$\xi_3^{(i_1,1,1)}$}} 
      child {node (a2) {$\xi_3^{(i_1,1,i_3)}$}}
      child {node (a3) {$\xi_3^{(i_1,1,n_3)}$}}
  }
  child {
    node (b) {$\xi_2^{(i_1,i_2)}$}
      child {node (b1) {$\xi_3^{(i_1,i_2,1)}$}}
      child {node (b2) {$\xi_3^{(i_1,i_2,i_3)}$}}
      child {node (b3) {$\xi_3^{(i_1,i_2,n_3)}$}}
  }
  child {
    node (c) {$\xi_2^{(i_1,n_2)}$}
      child {node (c1) {$\xi_3^{(i_1,n_2,1)}$}}
      child {node (c2) {$\xi_3^{(i_1,n_2,i_3)}$}}
      child {node (c3) {$\xi_3^{(i_1,n_2,n_3)}$}}
  };

\node at ($(a)!.5!(b)$) {\ldots};
\node at ($(b)!.5!(c)$) {\ldots};

\node at ($(a1)!.5!(a2)$) {\ldots};
\node at ($(a2)!.5!(a3)$) {\ldots};

\node at ($(b1)!.5!(b2)$) {\ldots};
\node at ($(b2)!.5!(b3)$) {\ldots};

\node at ($(c1)!.5!(c2)$) {\ldots};
\node at ($(c2)!.5!(c3)$) {\ldots};

\end{tikzpicture}

\caption{Visualization of the $i_1$-th scenario tree underlying the SAA estimator when $T=3$.}
\label{fig:scenario-tree}
\end{figure}
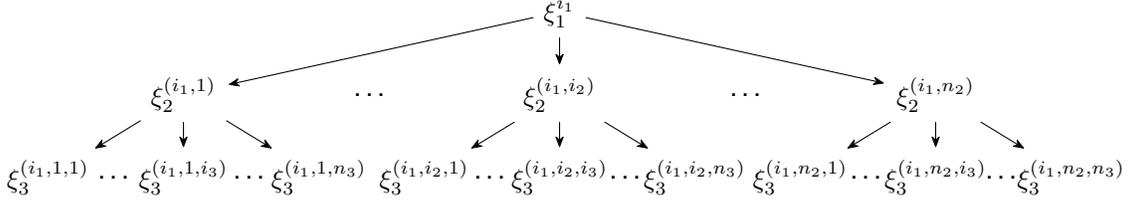

Note that the SAA estimator of Definition~\ref{def:MCCO_sample_average_approximation} requires $\CC(\hat F(x)) = \prod_{t=1}^T n_t$ scenarios. In the following we will prove that, as the sample sizes $n_t$, $t\in[T]$, tend to infinity, the estimator~$\hat{F}(x)$ converges in mean squared error and in probability to~$F(x)$ uniformly across all $x\in\cX$. To this end, we first rewrite the mean squared error of~$\hat{F}(x)$ with respect to~$F(x)$ as the sum of the squared bias and the variance of~$\hat F(x)$. 
\begin{align}\label{eq:mse_decomposition}
    \text{MSE}(\hat{F}(x))=\EE\Big[ \big(\hat{F}(x)-F(x)\big)^2\Big] =\Big(\EE\big[\hat{F}(x) \big] -F(x)\Big)^2 + \VV\big(\hat{F}(x) \big)
\end{align}
This formula shows that, for any fixed~$\epsilon>0$, we can drive the mean squared error of $\hat{F}(x)$ below $\epsilon^2$ by ensuring that $|\EE [\hat{F}(x)] -F(x)|\leq \epsilon/\sqrt{2}$ and $\VV(\hat{F}(x))\leq \epsilon^2/2$. In the following, we construct bounds on the bias and the variance of the SAA estimator and analyze their dependence on the sample sizes~$n_t$, $t\in[T]$. These bounds are obtained under the assumption that all integrands have finite variance.

\begin{assumption}\label{assump:finite_variance}
    Each integrand $f_t(\xi_t,x_t)$ has a finite conditional variance in the sense that $$\sigma_t^2\coloneqq\sup_{x_t\in\RR^{d_t}}\esssup \; \EE_{t-1} \left[ \big\|f_t(\xi_t,x_t) - \EE_{t-1}[f_t(\xi_t,x_t)] \big\|^2_2 \right]<\infty\quad\forall t\in[T].$$
\end{assumption}

We are now ready to derive an upper bound on the bias of the SAA estimator. This upper bound is expressed in terms of product Lipschitz constants of the form $L_{[t]}\coloneqq\prod_{s=1}^t L_s$ for $t\in[T]$ and $L_{[0]}=1$.

\begin{lemma}
\label{lemma:bound_on_bias_saa}
If Assumptions~\ref{assump:general}, \ref{assump:lipschitz} and \ref{assump:finite_variance} hold, then the bias of the SAA estimator satisfies 
\begin{align*}
    \left|  \EE \big[ \hat{F}(x)\big] - F(x) \right| \leq \begin{cases}
    \sum_{t=1}^{T-1} \frac{L_{[t]}\sigma_{t+1}}{\sqrt{n_{t+1}}} \; &\text{if Assumption~\ref{assump:smooth} does not hold}, \\ 
    \sum_{t=1}^{T-1}  \frac{L_{[t-1]}S_t\sigma_{t+1}^2}{2n_{t+1}} \; &\text{if Assumption~\ref{assump:smooth} holds}.
    \end{cases}
    \end{align*}
\end{lemma}

Lemma~\ref{lemma:bound_on_bias_saa} shows that the bias of the SAA estimator decays faster with the sample sizes if the integrands are smooth ({\em i.e.}, if Assumption~\ref{assump:smooth} holds). If~$T=2$, then Lemma~\ref{lemma:bound_on_bias_saa} recovers the bounds of \citep[Lemma~3.1]{hu2020sample} for conditional stochastic optimization. Next, we bound the variance of the SAA estimator.

\begin{lemma}
\label{lemma:bound_on_variance_saa} 
If Assumptions~\ref{assump:general},~\ref{assump:lipschitz} and~\ref{assump:finite_variance} hold, then the variance of the SAA estimator satisfies
    \begin{align*}
        \VV(\hat{F}(x)) \leq \frac{1}{n_1}\Big(\sum_{t=1}^{T-1} \frac{L_{[t]} \sigma_{t+1}}{\sqrt{n_{t+1}}} + \sigma_1\Big)^2.
    \end{align*}
\end{lemma}
\begin{rem}\label{remark:bound_on_variance_saa}
   The variance bound of Lemma~\ref{lemma:bound_on_variance_saa} can be strengthened when the integrands are smooth. However, the improved bound is cumbersome and does not decay faster with the sample sizes. To keep our analysis transparent, we therefore use the simpler bound of Lemma~\ref{lemma:bound_on_variance_saa} even if the integrands are smooth.
\end{rem}

Combining Lemmas~\ref{lemma:bound_on_bias_saa} and~\ref{lemma:bound_on_variance_saa} with~\eqref{eq:mse_decomposition} shows that the mean squared error of~$\hat F(x)$ satisfies
\begin{align*}
\text{MSE}(\hat{F}(x)) \leq 
\begin{cases}
      \big( \sum_{t=1}^{T-1} \frac{L_{[t]}\sigma_{t+1}}{\sqrt{n_{t+1}}} \big)^2 + \frac{1}{n_1}\big(\sum_{t=1}^{T-1} \frac{L_{[t]} \sigma_{t+1}}{\sqrt{n_{t+1}}} + \sigma_1\big)^2  &\text{if Assumption~\ref{assump:smooth} does not hold},\\ 
          \big( \sum_{t=1}^{T-1} \frac{L_{[t-1]}S_t \sigma_{t+1}^2}{2n_{t+1}}  \big)^2 + \frac{1}{n_1}\big(\sum_{t=1}^{T-1} \frac{L_{[t]} \sigma_{t+1}}{\sqrt{n_{t+1}}} + \sigma_1\big)^2  &\text{if Assumption~\ref{assump:smooth} holds}.
\end{cases} 
\end{align*}
We define the mean squared error-based scenario complexity of any estimator for~$F(x)$ as the total number of scenarios needed in order to ensure a root mean squared error of at most~$\epsilon$ uniformly across all $x\in \mathcal{X}$. The above bound enables us to characterize the mean squared error-based scenario complexity of the SAA estimator. To this end, recall that the construction of~$\hat F(x)$ requires a total of~$\prod_{t=1}^{T}n_t$ scenarios. The next theorem shows that setting the sample sizes to $n_t = \mathcal{O}(1/\epsilon^2)$ for nonsmooth integrands and $n_t = \mathcal{O}(1/\epsilon)$ for smooth integrands, for all $t = 2, \ldots, T$, ensures that the bias is bounded by~$\epsilon/\sqrt{2}$. Moreover, choosing $n_1 = \mathcal{O}(1/\epsilon^2)$ guarantees that the standard deviation is bounded by~$\epsilon/\sqrt{2}$. This implies that the mean squared error-based scenario complexity of the SAA estimator is at most~$\mathcal{O}(\epsilon^{-2T})$ for nonsmooth integrands and~$\mathcal{O}(\epsilon^{-(T+1)})$ for smooth integrands. The proof of this theorem is straightforward and therefore omitted. 
\begin{theorem}[Mean Squared Error-Based Scenario Complexity of the SAA Estimator]
\label{thm:saa_sample_complexity_mse} 
Suppose that Assumptions~\ref{assump:general}, \ref{assump:lipschitz} and~\ref{assump:finite_variance} hold. Assume also  that $n_1=\lceil1+2\sqrt{2}\sigma_1/\epsilon +2\sigma_1^2/\epsilon^2\rceil$ and that
\begin{align*}
    n_{t} = \begin{cases} \Big\lceil\Big(\frac{\sqrt{2}L_{[t-1]}\sigma_{t}(T-1)}{\epsilon}\Big)^2 \Big\rceil & \text{if Assumption~\ref{assump:smooth} does not hold,}\\[1ex]
    \Big\lceil\frac{\sqrt{2}L_{[t-2]}S_{t-1}\sigma_{t}^2(T-1)}{2\epsilon}\Big\rceil &\text{if Assumption~\ref{assump:smooth} holds,}
    \end{cases}
\end{align*}
for all $t=2,\ldots,T$. Then, the SAA estimator achieves a root mean squared error of at most $\epsilon$ uniformly across all~$x\in\mathcal X$. Thus, the mean squared error-based scenario complexity of the SAA estimator amounts at most to $\cO(\epsilon^{-2T})$ if Assumption~\ref{assump:smooth} does not hold and to $\cO(\epsilon^{-(T+1)})$ if Assumption~\ref{assump:smooth} holds.
\end{theorem}

Theorem~\ref{thm:saa_sample_complexity_mse} shows that although smoothness improves the exponent in the accuracy parameter~$\epsilon$, the scenario complexity of the SAA estimator still exhibits an exponential dependence on the horizon~$T$ in both the smooth and nonsmooth cases. This highlights a fundamental limitation of nested estimation. Achieving high accuracy becomes increasingly sample-intensive as~$T$ grows. It is important to note that these guarantees hold only \emph{in expectation}; that is, the root mean squared error of the SAA estimator is at most~$\epsilon$ \emph{on average} over repeated independent simulation runs. To obtain stronger guarantees that hold for a \emph{single} run with high confidence, we next turn to high-probability bounds. These results are based on concentration inequalities and therefore require an additional standard sub-Gaussianity assumption. 

\begin{defn}[Sub-Gaussian Random Vector] An $m$-dimensional random vector $\xi$ with finite mean is called sub-Gaussian with variance proxy $\zeta^2>0$ if
\begin{align*}
    \EE \big[\exp(\lambda^\top (\xi-\EE[\xi]))\big]\leq \exp\left(\|\lambda\|_2^2\,\zeta^2/2\right) \quad \forall \lambda \in\RR^{m}.
\end{align*}    
\end{defn}

\begin{assumption}[Sub-Gaussianity of $\hat F_{i_1}(x)$]
\label{assump:sub-gaussianity_saa} 
For each $i_1\in[n_1]$, the estimator $\hat F_{i_1}(x)$ corresponding to the $i_1$-th scenario tree in Definition~\ref{def:MCCO_sample_average_approximation} is sub-Gaussian with variance proxy $\zeta^2>0$ for any fixed~$x\in\cX$.
\end{assumption}

Assumption~\ref{assump:sub-gaussianity_saa} is hard to verify in general, but Remark~\ref{rem:sub-gaussian} below outlines conditions when it holds.

\begin{rem}
\label{rem:sub-gaussian}
Assumption~\ref{assump:sub-gaussianity_saa} holds under either of the following easily verifiable conditions.
\begin{enumerate}[label=(\roman*)]
    \item If $\xi_t$ is sub-Gaussian with variance proxy $\zeta_t^2$ and $f_t(\xi_t,x_t)$ is jointly $L_t$-Lipschitz continuous in $\xi_t$ and $x_t$ for all~$t\in[T]$, then Assumption~\ref{assump:finite_variance} is automatically satisfied \cite[Exercise~2.5]{wainwright2019high}. In addition, one readily verifies that $\hat F_{i_1}(x)$ is sub-Gaussian. Indeed, $\hat F_{i_1}(x)$ can be constructed by iteratively averaging and concatenating sub-Gaussian random variables and transforming them under Lipschitz continuous functions. By Lemma~\ref{lemma:sub-gaussianity-preserving-operations} in the appendix, all of these operations preserve sub-Gaussianity.

    \item If there exists a constant $M>0$ such that $|f_1(\xi_1,x_1)|\leq M$ for all $\xi_1\in\RR^{m_1}$ and $x_1\in\RR^{d_1}$, then \cite[Example~2.4]{wainwright2019high} readily implies that $\hat F_{i_1}(x)$ is sub-Gaussian with variance proxy $\zeta^2=4M^2$.
\end{enumerate}
\end{rem}

We now establish the uniform convergence of the SAA estimator to the true objective function.

\begin{lemma}\label{lemma:unif_conv_saa} Suppose that
 Assumptions~\ref{assump:general}, \ref{assump:lipschitz}, \ref{assump:finite_variance} and~\ref{assump:sub-gaussianity_saa} hold and that $\epsilon>0$ is a prescribed tolerance. Assume also that the stagewise sample sizes satisfy
\begin{align*}
    n_t = \begin{cases}
    \Big \lceil \Big( \frac{4L_{[t-1]}\sigma_{t} (T-1)}{\epsilon} \Big)^2 \Big \rceil&\text{if Assumption~\ref{assump:smooth} does not hold,} \\[1ex]
    \Big\lceil\frac{2L_{[t-2]}S_{t-1}\sigma_{t}^2 (T-1)}{\epsilon}\Big\rceil &\text{if Assumption~\ref{assump:smooth} holds,}
    \end{cases}
\end{align*}
for all $t=2,\ldots,T$. 
Then, we have
\begin{align*}
    \PP\left(\sup_{x\in \mathcal{X}} \; \lvert \hat{F}(x) - F(x) \rvert > \epsilon \right) \leq
    2\left\lceil 8L_{[T]} D_\mathcal{X}/\epsilon +1\right\rceil^{d} \exp\bigg(- \frac{n_1\epsilon^2}{32\zeta^2} \bigg).
    \end{align*}
\end{lemma}
Recall that~$\zeta^2$ is the variance proxy of~$\hat{F}_{i_1}(x)$, $d$ and~$D_\cX$ are the dimension and the diameter of~$\cX$, respectively, and~$L_{[T]}$ is the Lipschitz constant of~$F(x)$. Lemma~\ref{lemma:unif_conv_saa} establishes a high probability bound on the worst possible difference between the SAA estimator and the true objective function $F(x)$ uniformly across all $x\in\mathcal{X}$. It asserts that the probability that this difference exceeds any given threshold~$\epsilon>0$ decays exponentially fast with~$\epsilon$ and with the sample size~$n_1$. 

The uniform convergence result of Lemma~\ref{lemma:unif_conv_saa} not only characterizes the number of scenarios needed to approximately solve the estimation problem~\eqref{problem:MCCE}, which is the central focus of this section, but also inspires a straightforward method to approximately solve the optimization problem~\eqref{problem:MCCO}. Indeed, instead of minimizing the function~$F(x)$, whose values and gradients are difficult to compute exactly, one may minimize the SAA estimator~$\hat F(x)$, that is, one may solve an instance of~\eqref{prob:min_hatF(x)}. Below, we will use Lemma~\ref{lemma:unif_conv_saa} to show that the exact minimizers of the SAA problem~\eqref{prob:min_hatF(x)} yield near-optimal solutions to~\eqref{problem:MCCO} with high probability. The practicality of this approach ultimately depends on the tractability of~\eqref{prob:min_hatF(x)}.

\begin{rem}\label{rem:mcce_mcco_transition_saa}
    Global optimization of the SAA estimator~$\hat F(x)$ is tractable in settings where~$\hat{F}(x)$ is convex and admits an efficient separation oracle. A sufficient condition for convexity of both~$F(x)$ and~$\hat F(x)$ is that the integrand~$f_t(\xi_t, x_t)$ is convex in~$x_t$ for every~$t \in [T]$ and non-decreasing in~$x_t$ for every~$t \in [T-1]$. These conditions ensure that the global minimum of~\eqref{prob:min_hatF(x)} can be found efficiently. However, they are restrictive. 
\end{rem}

We now show that the minimizers of~\eqref{prob:min_hatF(x)} are near-optimal in~\eqref{problem:MCCO} with high probability.

\begin{corollary}
\label{cor:high_prob_bound_saa_epsilon_optimal} Assume that all conditions of Lemma~\ref{lemma:unif_conv_saa} hold for some~$\epsilon>0$, and let $x^*$ and $\hat{x}^*$ be minimizers of \eqref{problem:MCCO} and the SAA problem \eqref{prob:min_hatF(x)}, respectively. Then, we have
\begin{align*}
    \PP\big( F(\hat{x}^*) - F(x^*) > \epsilon \big) \leq 4\left\lceil 8L_{[T]} D_\mathcal{X}/\epsilon +1\right\rceil^{d} \exp\bigg(- \frac{n_1\epsilon^2}{128\zeta^2} \bigg).
\end{align*}   
\end{corollary}

We define the high-probability scenario complexity of a method for solving~\eqref{problem:MCCO} as the total number of scenarios that this method needs to identify an $\epsilon$-optimal solution of~\eqref{problem:MCCO} with probability~$1-\beta$. 

\begin{theorem}[High-Probability Scenario Complexity of the SAA Estimator]
    \label{thm:SAA_sample_complexity_high_probability} Suppose that all conditions of Lemma~\ref{lemma:unif_conv_saa} hold for some~$\epsilon>0$ and that
\begin{align*}
    n_1 =\left\lceil \frac{128\zeta^2}{\epsilon^2} \left(d \log\left( \left\lceil 8L_{[T]}D_\mathcal{X}/\epsilon +1\right\rceil\right) + \log\left(4/\beta \right) \right) \right\rceil
\end{align*}
for some $\beta\in(0,1)$. Then any solution of the SAA problem \eqref{prob:min_hatF(x)} constitutes an $\epsilon$-optimal solution of \eqref{problem:MCCO} with probability at least $1-\beta$. Thus, the high-probability scenario complexity of the SAA estimator equals $\cO(\log(\epsilon^{-1})\epsilon^{-2T})$ if Assumption~\ref{assump:smooth} does not hold and $\cO(\log(\epsilon^{-1})\epsilon^{-(T+1)})$ if Assumption~\ref{assump:smooth} holds.
\end{theorem}

Theorem~\ref{thm:SAA_sample_complexity_high_probability} generalizes \cite[Corollary~4.2]{hu2020sample} to general $T>2$. It shows that, like the mean squared error-based scenario complexity, the high-probability scenario complexity of the SAA method grows exponentially with~$T$. This aligns with the observation in multistage stochastic programming that SAA suffers from a curse of dimensionality \cite{shapiro2006complexity, reaiche2016note}. Note also that the exponential growth rate is reduced by a factor of two if the integrands are smooth in the sense of Assumption~\ref{assump:smooth}. In contrast to Theorem~\ref{thm:saa_sample_complexity_mse}, the bound on~$n_1$ depends now on the dimension~$d$ of the decision variable and the variance proxy $\zeta^2$ of the tree-wise SAA estimator. This captures the cost of enhancing guarantees from convergence in expectation to convergence with high probability. If the SAA estimator~$\hat{F}(x)$ further satisfies the Hölderian error bound condition proposed in \citep[Assumption~4.1]{hu2020sample}, then the bound on~$n_1$ in Theorem~\ref{thm:SAA_sample_complexity_high_probability} can be made independent of the dimension~$d$. In the special case where~$T=2$, the scenario complexity bounds of Theorem~\ref{thm:SAA_sample_complexity_high_probability} reduce to the known bounds for conditional stochastic optimization reported in~\citep{hu2020sample}. These bounds scale quartically with the inverse accuracy $1/\epsilon$ for nonsmooth integrands and cubically for smooth integrands, demonstrating the efficiency of the SAA method when~$T$ is small. However, for larger values of~$T$, the method becomes increasingly inefficient, underscoring the need for more advanced and scalable approaches to solve~\eqref{problem:MCCO}. 


\subsection{MLMC Estimator}
\label{sec:rtmlmc}
We now develop a recursive MLMC estimator for the nested expectation problem~\eqref{problem:MCCE}. Our approach generalizes the unbiased MLMC method of \cite{blanchet2015unbiased} to arbitrary depths $T\geq 2$, but introduces a controlled bias to simultaneously guarantee finite variance and computational cost under weaker conditions ({\em e.g.}, Lipschitz continuity). This contrasts with the unbiased recursive MLMC estimator of \cite{syed2023optimal}, which requires smoothness to ensure a finite variance. To this end, we need additional notation and definitions. For any $t+1\in[T]$ and Borel-measurable function~$\psi(\xi_{[t+1]})$, we continue to use the empirical conditional expectation operator~$\hat\EE_t[\psi(\xi_{[t+1]})]$ with $n_{t+1}$ samples defined in~\eqref{eq:expected_psi}. We further introduce empirical conditional expectation operators that run over $2^\ell$ samples as well as over the even or odd indices in $[2^\ell]$. In other words, we define
\begin{align*}
    \begin{array}{l@{\,}c@{\,}r@{\,}l@{\qquad}l}
        \displaystyle \hat\EE_t^{\ell}[\psi(\xi_{[t+1]})] & \coloneqq &\displaystyle \frac{1}{2^\ell} & \displaystyle\sum_{i=1}^{2^\ell} \psi(\xi_{[t]},\xi_{t+1}^{i}) & \forall\ell\in\NN_0,\\
        \hat\EE_t^{\ell,\rm o}[\psi(\xi_{[t+1]})] & \coloneqq &\displaystyle  \frac{1}{2^{\ell-1}} & \displaystyle\sum_{i=1}^{2^{\ell-1}} \psi(\xi_{[t]},\xi_{t+1}^{2i-1}) & \forall\ell\in\NN,\\
        \hat\EE_t^{\ell,\rm e} [\psi(\xi_{[t+1]})] & \coloneqq &\displaystyle  \frac{1}{2^{\ell-1}} & \displaystyle \sum_{i=1}^{2^{\ell-1}} \psi(\xi_{[t]},\xi_{t+1}^{2i}) & \forall\ell\in\NN,
    \end{array}
\end{align*}
where the samples $\xi_{t+1}^{i}$, $i\in[2^\ell]$, are drawn independently from the marginal distribution of~$\xi_1$ (if $t=0$) or the conditional distribution of $\xi_{t+1}$ given $\xi_{[t]}$ (if $t>0$). By construction, we thus have
\begin{align}
    \label{eq:even-odd-averages-identity}
    \hat\EE_t^{\ell}[\psi(\xi_{[t+1]})] =\frac{1}{2}\hat\EE_t^{\ell,{\rm o}}[\psi(\xi_{[t+1]})] + \frac{1}{2} \hat\EE_t^{\ell, {\rm e}}[\psi(\xi_{[t+1]})].
\end{align}

\begin{defn}[MLMC Estimator]
\label{def:MCCO_RTMLMC}
For any $t\in[T-1]$, select a rate parameter~$r_t \in(0,1)$ and truncation point $M_t\in\mathbb N$, and let the random variable~$\lambda_t$ follow the truncated geometric distribution $\mathrm{Geo}(r_t | M_t)$. Assume that~$\lambda_t$ is independent of all other random objects, and set $q_t(\ell)\coloneqq\PP(\lambda_t=\ell)$ for~$\ell=0,\ldots,M_t$ and~$t\in[T-1]$. We then construct a sequence of estimators $\hat H_t(x)$, $t\in[T]$, using the following backward recursion. We initialize the recursion by setting $\hat H_T(x)\coloneqq f_T(\xi_T,x)$. Next, for any $t\in[T-1]$ we set
\begin{align}
    \hat{H}_t(x)\coloneqq \frac{1}{q_{t}(\lambda_{t})}\left( \hat{h}_t^{\lambda_t}(x)-\frac{1}{2}\hat{h}_t^{\lambda_t, {\rm e}}(x)-\frac{1}{2}\hat{h}_t^{\lambda_t,{\rm o}}(x)\right),
    \label{eq:H_hat}
\end{align}
where
\begin{align*}
    \hat{h}_t^{\ell}(x)\coloneqq f_t \Big(\xi_t,\hat \EE_t^{\ell} \big[ \hat{H}_{t+1}(x) \big] \Big) \quad \forall \ell=0,\ldots, M_t
\end{align*}
as well as
\begin{align*}
    \hat{h}_t^{\ell,{\rm e}}(x) \coloneqq f_t \Big(\xi_t, \hat \EE_t^{\ell,\rm e} \big[ \hat H_{t+1}(x) \big] \Big) \quad \text{and} \quad \hat{h}_t^{\ell,{\rm o}}(x) \coloneqq f_t \Big(\xi_t, \hat \EE_t^{\ell,\rm o} \big[ \hat{H}_{t+1}(x) \big] \Big) \quad \forall \ell\in [M_t].
\end{align*}
We also set $\hat{h}_t^{0,{\rm e}}(x)\coloneqq\hat{h}_t^{0,{\rm o}}(x)\coloneqq 0$. The MLMC estimator for $F(x)$ is then defined as $\hat{F}(x)\coloneqq \hat\EE [\hat{H}_{1}(x)]$. 
\end{defn}

The MLMC estimator is built on a forest of~$n_1$ scenario trees that display a random branching structure with $2^{\lambda_t}$ branches per stage-$t$ node, where $\lambda_t$ is sampled from a truncated geometric distribution. For this reason we will sometimes refer to $\lambda_t$ as the log-branching factor of stage~$t$. If $M_t=\infty$ for all $t\in[T-1]$ and the integrands $f_t$ map to~$\RR$, then our MLMC estimator reduces to the untruncated MLMC estimator proposed in \cite{syed2023optimal}. If additionally $T=2$, then it further reduces to the estimator discussed in \cite[\S~3.2]{wang2023unbiased}.

Working with random (log-)branching factors offers significant advantages. We will show that the bias of the MLMC estimator is of the same order as that of the SAA estimator with a large fixed branching factor~$2^{M_t}$ at each stage~$t\in[T-1]$. Since the MLMC estimator is constructed on a scenario tree with a random branching factor~$2^{\lambda_t}$ that never exceeds but is typically much smaller than~$2^{M_t}$ at every stage, it achieves this bias using substantially fewer samples in total. As~$\lambda_t$ concentrates near zero, the branching factor~$2^{\lambda_t}$ often equals~$1$, so the MLMC estimator requires only a single sample in most branches. Consequently, the scenario tree grows more slowly with the time horizon~$T$, making the evaluation of the MLMC estimator significantly more efficient than that of the SAA estimator with comparable bias.

We now prove that the MLMC estimator requires fewer samples than the SAA estimator to achieve a mean squared error no greater than~$\epsilon^2$. To do so, we derive upper bounds on the variance, bias, and expected sampling cost of the MLMC estimator in Lemmas~\ref{lemma:bound_on_variance_rtmlmc}, \ref{lemma:bias_bound_rtmlmc}, and~\ref{lemma:bound_on_cost_rtmlmc}, respectively. These bounds depend on the estimator’s hyperparameters, that is, the number~$n_1$ of scenario trees as well as the truncation points~$M_t$ and the rates~$r_t$ that define the truncated geometric distributions of the log-branching factors~$\lambda_t$, $t \in [T-1]$. We then tune these hyperparameters to ensure that the bias and the variance of the MLMC estimator are bounded by~$\epsilon/\sqrt{2}$. This calibration guarantees an overall mean squared error of at most~$\epsilon^2$, enabling a fair comparison of the expected sampling costs of the optimally tuned SAA and MLMC estimators.

From this point onward, we denote the $p$-th moment of~$\hat{H}_t(x)$ conditional on time~$t-1$ information by
\begin{align}
\label{eq:mu}
    \mu_t^{p}(x) \coloneqq \EE_{t-1}\big[\|\hat{H}_t(x)\|_p^p \big],
\end{align}
where $p \geq 1$. To analyze the MLMC estimator, we also define the conditional moment bound of order~$p$ as
\begin{align}\label{eq:mu_bar_def}
    \overline\mu_T^p\coloneqq \sup_{x\in\RR^{d}}\esssup \; \EE_{T-1} \left[ \big\|f_T(\xi_T,x)\big\|^p_p \right].
\end{align}
By using Jensen's inequality, one readily verifies that if $\overline\mu_T^p<\infty$, then $\overline\mu_T^{p'}<\infty$ for every~$p'\in[1,p]$. The following assumption is crucial for our subsequent results.
\begin{assumption}\label{assump:mu_bar}
    The conditional moment bound $\overline{\mu}_T^p $ of the integrand $f_T(\xi_T,x)$ is finite up to order $p=2$ if Assumption~\ref{assump:smooth} does not hold and up to order $p=2^T$ if Assumption~\ref{assump:smooth} holds.
\end{assumption}

The following lemma establishes bounds on~$\mu_t^2(x)$ and~$\mu_t^{2^t}(x)$ for all~$t \in [T]$, corresponding to the cases of nonsmooth and smooth integrands, respectively. In particular, for~$t = 1$, these bounds can be used to derive upper bounds on the unconditional variance of the MLMC estimator in both settings. For any $t\in[T]$ and $s\in[t+1]$ we henceforth set $L_{[s:t]}\coloneqq\prod_{i=s}^t L_i$ if $s<t+1$; $\coloneqq1$ if $s=t+1$; see also Lemma~\ref{lem:lipschitz} in the appendix. In particular, note that $L_{[1:t]}$ coincides with the Lipschitz constant~$L_{[t]}$ defined before Lemma~\ref{lemma:bound_on_bias_saa}.

\begin{lemma}
\label{lemma:bound_on_variance_rtmlmc}
Suppose that Assumptions~\ref{assump:general}, \ref{assump:lipschitz} and~\ref{assump:mu_bar} hold. Then, for all $x \in \cX$ and $t \in [T]$, the conditional moments of~$\hat{H}_t(x)$ satisfy
\begin{align*}
    \begin{array}{r@{\;}ll}
    \mu_t^2(x)& \displaystyle \leq 
         C_t \displaystyle\prod_{s=t}^{T-1} \bigg(\displaystyle\sum_{\ell=0}^{M_s}\frac{1}{2^{\ell} q_s(\ell)} \bigg) & \text{if Assumption~\ref{assump:smooth} does not hold,}\\
    \mu_t^{2^t}(x) & \displaystyle \leq 
        D_t\displaystyle\prod_{s=t}^{T-1}           \bigg(\displaystyle\sum_{\ell=0}^{M_s}\frac{1}{2^{2^s \ell} q_s(\ell)^{2^{s}-1}} \bigg)& \text{if Assumption~\ref{assump:smooth} holds,}
        \end{array}
\end{align*}
where
\begin{align*}
    C_t\coloneqq \overline\mu_T^2 (2B_2)^{T-t}L_{[t:T-1]}^2\quad\text{and}\quad
    D_t\coloneqq\overline\mu_T^{2^T} \prod_{s=t}^{T-1} \bigg(\bigg(\frac{3S_s}{2}\bigg)^{2^s} d_s^{2^s-1} B_{2^{s+1}}\bigg),
\end{align*}
and where $B_2$ and $B_{2^{s+1}}$ are universal constants defined in Lemma~\ref{lemma:moment-inequality-multidimensional} in the appendix.
\end{lemma}
Lemma~\ref{lemma:bound_on_variance_rtmlmc} allows us to bound the variance of the MLMC estimator~$\hat F(x)=\hat\EE[\hat H_1(x)]$. Indeed, as~$\hat F(x)$ constitutes an average of~$n_1$ independent copies of~$\hat H_1(x)$, its variance admits the bound
\begin{equation}
    \label{eq:variance-rtmlmc-estimator}
    \VV(\hat F(x)) = \frac{1}{n_1} \VV(\hat H_1(x)) \leq \frac{\mu_1^2(x)}{n_1},
\end{equation}
while Lemma~\ref{lemma:bound_on_variance_rtmlmc} provides a bound on~$\mu_1^2(x)$.
In fact, Lemma~\ref{lemma:bound_on_variance_rtmlmc} constructs bounds on all conditional moments~$\mu_t^2(x)$ (for nonsmooth integrands that may violate Assumption~\ref{assump:smooth}) and~$\mu_t^{2^t}(x)$ (for smooth integrands that satisfy Assumption~\ref{assump:smooth}) by backward recursion on~$t$. Note that if the integrands are nonsmooth, then the upper bound on~$\mu_1^2(x)$ is finite only if~$f_T(\xi_T,x)$ has a finite second moment. If the integrands are smooth, on the other hand, then the upper bound on~$\mu_1^2(x)$ is finite only if~$f_T(\xi_T,x)$ has a finite moment of order~$2^T$. The requirement that moments of order higher than two of~$f_T(\xi_T, x)$ be finite arises as an artifact of the bound's construction, which involves iterative application of the inequality~\eqref{eq:inequality_for_smooth_funs} for smooth functions. This inequality is conservative because it neglects the Lipschitz continuity of~$f_T(\xi_T, x)$ with respect to~$x$. Indeed, the Lipschitz continuity ensures that the global growth of the left hand side of~\eqref{eq:inequality_for_smooth_funs} is at most linear. In spite of this, we do not attempt to improve the conditional moment bounds for smooth integrands in Lemma~\ref{lemma:bound_on_variance_rtmlmc}. Deriving such bounds is cumbersome and unlikely to improve our scenario complexity results below. Moreover, to establish high-probability scenario complexity guarantees for the MLMC estimator, we will anyway assume that~$f_T(\xi_T, x)$ is sub-Gaussian, which implies that all moments of~$f_T(\xi_T, x)$ are finite. 

Lemma~\ref{lemma:bound_on_variance_rtmlmc} readily extends to the untruncated MLMC estimator. In addition, if Assumption~\ref{assump:smooth} holds while~$d_t=1$ and~$M_t = \infty$ for all~$t\in[T-1]$, then our upper bound on the $2^t$-th conditional moment of~$\hat{H}_{t}$ coincides with the $2^t$-th \emph{un}conditional moment of the untruncated MLMC estimator analyzed in~\cite{syed2023optimal}.

The next result uses Lemma~\ref{lemma:bound_on_variance_rtmlmc} to derive bounds on the bias of the MLMC estimator.

\begin{lemma}
\label{lemma:bias_bound_rtmlmc}
If Assumptions~\ref{assump:general}, \ref{assump:lipschitz} and~\ref{assump:mu_bar} hold, then the bias of the MLMC estimator satisfies
    \begin{align*}
      \big|\EE[\hat{F}(x)] - F(x)\big| \leq \begin{cases} \sum_{t=1}^{T-1}\frac{L_{[t]}\EE[\mu_{t+1}^2(x)]^\frac{1}{2}}{2^{M_t/2}}&\text{if Assumption~\ref{assump:smooth} does not hold,}\\
       \sum_{t=1}^{T-1}\frac{L_{[t-1]}S_t \EE[\mu_{t+1}^{2}(x)]}{2^{M_t + 1}} &\text{if Assumption~\ref{assump:smooth} holds.}\end{cases}
    \end{align*}
\end{lemma}


Lemma~\ref{lemma:bias_bound_rtmlmc} shows that the bias of the MLMC estimator decays faster with the truncation points if the integrands are smooth ({\em i.e.}, if Assumption~\ref{assump:smooth} holds) and if $\EE[\mu_{t+1}^2(x)]$ is finite. Lemma~\ref{lemma:bias_bound_rtmlmc_H_t} in the appendix establishes a stronger result than Lemma~\ref{lemma:bias_bound_rtmlmc} in that it provides a bound on the bias of~$\hat H_t(x)$ conditional on time~$t-1$ information for every~$t\in[T]$. As~$\EE[\hat{F}(x)]=\EE[\hat{H}_1(x)]$, Lemma~\ref{lemma:bias_bound_rtmlmc} follows immediately from Lemma~\ref{lemma:bias_bound_rtmlmc_H_t} for~$t=1$, and therefore its proof is omitted for brevity.

Lemmas~\ref{lemma:bias_bound_rtmlmc} and~\ref{lemma:bound_on_bias_saa} reveal that the bias of the MLMC estimator with random branching factors~$2^{\lambda_t}$ for~$t \in [T-1]$ is of the same order of magnitude as that of the SAA estimator with deterministic branching factors~$n_{t+1}= 2^{M_t}$ for~$t \in [T-1]$. In addition, Lemma~\ref{lemma:bias_bound_rtmlmc} generalizes \cite[Proposition 4.1]{hu2021bias}, which focuses on conditional stochastic programs with~$T=2$, and \cite[Theorem~2.2]{syed2023optimal}, which focuses on untruncated and thus unbiased MLMC estimators with $M_t=\infty$ for all $t\in[T-1]$. Note that when~$T=2$, the bias of the SAA and MLMC estimators can be analyzed in the same manner. However, when $T>2$, the bias of the MLMC estimator requires a more involved analysis due to its recursive construction, and it also involves a bound on the variance of~$\hat H_t(x)$.

\begin{rem}
The bias bounds stated in Lemma~\ref{lemma:bias_bound_rtmlmc} are derived by first expressing the bias of~$\hat F(x)$ as a telescoping sum. Then, Jensen's inequality is applied to bound the absolute value of this sum by the sum of the absolute values of its terms. A tighter bias bound can be obtained using an inductive proof technique that, at each step, leverages a bias-variance decomposition. If the integrands are only known to be Lipschitz continuous (i.e., if Assumption~\ref{assump:smooth} does not hold), then this refined technique yields the bias bound
\begin{align*}
    \big|\EE[\hat{F}(x)] - F(x)\big| \leq \bigg(\sum_{t=1}^{T-1} \frac{L_{[t]}^2 \, \EE[\mu_{t+1}^2(x)]}{2^{M_t}} \bigg)^{1/2}.
\end{align*}
This bound on the bias of~$\hat F(x)$ matches the $2$-norm of the vector with entries $L_{[t]} \, \EE[\mu_{t+1}^2(x)]^\frac{1}{2}/2^{M_t/2}$, $t\in[T-1]$. Thus, it strengthens the corresponding bias bound of Lemma~\ref{lemma:bias_bound_rtmlmc},
which matches the $1$-norm of the same vector. In addition, if the integrands are known to be smooth (i.e., if Assumption~\ref{assump:smooth} holds), it is still possible to recursively construct a bias bound that improves upon the one given in Lemma~\ref{lemma:bias_bound_rtmlmc}. However, this refined bound does not admit a simple closed-form expression. As the recursively constructed bounds (for nonsmooth as well as smooth integrands) exhibit the same asymptotic dependence on the truncation points~$M_t$, $t \in [T-1]$, we choose to work with the simpler and more interpretable bounds of Lemma~\ref{lemma:bias_bound_rtmlmc}.
\end{rem}

The next lemma provides an explicit formula for the expected sampling cost of the MLMC estimator.

\begin{lemma}
\label{lemma:bound_on_cost_rtmlmc}
The expected number of scenarios needed to construct the MLMC estimator amounts to 
\begin{equation*}
   \EE\big[\CC\big(\hat{F}(x)\big)\big]= n_1 \prod_{t=1}^{T-1} \sum_{\ell=0}^{M_t} q_t(\ell)2^{\ell}.
\end{equation*}
\end{lemma}

Lemmas~\ref{lemma:bound_on_variance_rtmlmc}, \ref{lemma:bias_bound_rtmlmc} and~\ref{lemma:bound_on_cost_rtmlmc} provide bounds on the variance, the bias and the expected computational cost of the MLMC estimator, respectively. They hold for any choices of the estimator's hyperparameters, that is, the number~$n_1$ of scenario trees as well as the rate parameter~$r_t$ and truncation point~$M_t$ of the geometric distribution governing the log-branching factor~$\lambda_t$ for every~$t\in[T-1]$. These hyperparameters should be tuned to guarantee a mean squared error of at most~$\epsilon^2$ at minimal computational cost. To achieve this, we select sufficiently large truncation points~$M_t$, $t\in[T-1]$, to ensure that the bias bound of Lemma~\ref{lemma:bias_bound_rtmlmc} is at most~$\epsilon/\sqrt{2}$. We also select a sufficiently large number~$n_1$ of scenario trees to ensure that the variance is at most~$\epsilon^2/2$. In view of~\eqref{eq:variance-rtmlmc-estimator}, this is achieved by setting~$n_1 \approx 2\mu_1^2(x)/\epsilon^2$, which implies that the expected sampling cost satisfies
\[
    \CC(\hat F(x)) \approx \frac{2 \mu_1^2(x)}{\epsilon^2} \prod_{t=1}^{T-1} \sum_{\ell=0}^{M_t} q_t(\ell)2^{\ell};
\]
see Lemma~\ref{lemma:bound_on_cost_rtmlmc}. Recall that we only have access to upper bounds on~$\mu_1^2(x)$ that depend indirectly on the rate parameters~$r_t$ through the probabilities~$q_t(\ell)$. We thus construct approximate upper bounds on~$\CC(\hat F(x))$ uniformly across all~$x\in\cX$ for both nonsmooth and smooth integrands by replacing~$\mu_1^2(x)$ on the right hand side of the above expression with the respective bounds given in Lemma~\ref{lemma:bound_on_variance_rtmlmc}. We denote the resulting upper bounds by~$c_{\mathrm{ns}}(r)$ and~$c_{\mathrm{s}}(r)$ for nonsmooth and smooth integrands, respectively, where~$r = (r_1, \ldots, r_{T-1})$. Ideally, we should choose~$r\in(0,1)^{T-1}$ so as to minimize these upper bounds. However, the exact minimizers are difficult to obtain in closed form. The next assumption provides near-optimal choices for~$r$.

\begin{assumption}[Rate Parameters] 
\label{assump:rate-parameters-rtmlmc-funval}
For any $t\in[T-1]$ we set
\begin{align*}
    r_t=\begin{cases} 2^{-1}&\text{if Assumption~\ref{assump:smooth} does not hold,}\\
    1-2^{-1-2^{-t}} &\text{if Assumption~\ref{assump:smooth} holds.}
    \end{cases}
\end{align*}
\end{assumption}

\begin{rem}\label{remark:optimal-prob} 
Lemma~\ref{lemma:optimal_b} in the appendix shows that if the integrands are nonsmooth (Assumption~\ref{assump:smooth} does not hold), then the rate parameters of Assumption~\ref{assump:rate-parameters-rtmlmc-funval} exactly minimize the approximate upper bound~$c_{\mathrm{ns}}(r)$ on the expected sampling cost over all~$r\in(0,1)$. If~$M_t = \infty$ for all $t\in[T-1]$, however, then the choice $r_t = 2^{-1}$ leads to infinite variance and cost bounds. In fact, no value of~$r_t$ can simultaneously ensure finite second moment and finite expected sampling cost; see also \cite{syed2023optimal}. These findings underscore the merits of truncated MLMC estimators. If the integrands are smooth (Assumption~\ref{assump:smooth} holds), then the rate parameters of Assumption~\ref{assump:rate-parameters-rtmlmc-funval} exactly minimize the approximate upper bound~$c_{\mathrm{s}}(r)$ only for~$T=2$. For $T > 2$, the rate parameters in Assumption~\ref{assump:rate-parameters-rtmlmc-funval} jointly control both the geometric growth of the second moment and the expected sampling cost of~$\hat{F}(x)$. Specifically, Lemma~\ref{lemma:bound_on_variance_rtmlmc} implies that $\mu_1^2$ is bounded in~$M_t$ whenever $r_t < 1 - 2^{-1 - (2^t - 1)^{-1}}$, while Lemma~\ref{lemma:bound_on_cost_rtmlmc} implies that the expected sampling cost is bounded in~$M_t$ whenever $r_t > 2^{-1}$ for all $t \in [T-1]$. These bounds remain valid in the special case when~$M_t=\infty$ for all~$t\in[T-1]$ addressed in~\citep{syed2023optimal}. The choice $r_t = 1 - 2^{-1 - 2^{-t}}$ for all $t \in [T-1]$ satisfies both of these strict inequalities and provides a balanced trade-off between the geometric growth of the level-$\ell$ contributions to the variance and the expected sampling cost. Hence, this choice is natural. 
\end{rem}

The MLMC estimator is fully determined by the number~$n_1$ of scenario trees as well as the rates~$r_t$ and the truncation points~$M_t$ that define the truncated geometric distributions of the log-branching factors~$\lambda_t$ for all~$t \in [T-1]$. The next theorem specifies a particular choice of these hyperparameters that guarantees the mean squared error of the MLMC estimator to be bounded by~$\epsilon^2$. It also characterizes the estimator's mean squared error-based scenario complexity, which we define as the expected value of its sampling cost.

\begin{theorem}[Mean Squared Error-Based Scenario Complexity of the MLMC Estimator]
\label{thm:rtmlmc_sample_complexity_mse} 
Suppose that Assumptions~\ref{assump:general}, \ref{assump:lipschitz}, \ref{assump:mu_bar} and~\ref{assump:rate-parameters-rtmlmc-funval} hold and that $\epsilon>0$ is a given tolerance. Set $n_1=\lceil \sup_{x\in\cX} 2\mu_1^2(x)/\epsilon^2\rceil$, and define the truncation points~$M_t$, $t\in[T-1]$, through the backward recursion
\begin{align*}
     M_{t} = \bigg\lceil 2\log_2\Big( \textstyle \sqrt{2}L_{[t]}\sqrt{C_{t+1}} \, \Big( \prod_{s=t+1}^{T-1} \sqrt{Z_s} \sqrt{M_s+1} \Big) (T-1)/\epsilon \Big) \bigg\rceil
\end{align*}
if Assumption~\ref{assump:smooth} does not hold and 
\begin{align*}
     \textstyle M_{t}  = \bigg\lceil\log_2\bigg( \sqrt{2}L_{[t-1]}S_{t}D_{t+1}^{2^{-t}} \, \bigg( \prod_{s=t+1}^{T-1}Z_s^{\frac{2^s-1}{2^t}} \Big(\frac{1-2^{-\frac{M_s+1}{2^s}}}{1-2^{-\frac{1}{2^s}}} \Big)^\frac{1}{2^t} \bigg)  (T-1)/(2\epsilon)\bigg)\bigg\rceil
\end{align*}
if Assumption~\ref{assump:smooth} holds. Here, $C_t$ and~$D_t$ are defined as in Lemma~\ref{lemma:bound_on_variance_rtmlmc}, and $Z_t\coloneqq ( 1-(1-r_t)^{M_t+1} ) / r_t$. Then, the  MLMC estimator achieves a root mean squared error of at most~$\epsilon$ uniformly across all~$x\in\cX$. Thus, the mean squared error-based scenario complexity of the MLMC estimator amounts at most to $\cO(\log(\epsilon^{-1})^{2(T-1)}\epsilon^{-2})$ if Assumption~\ref{assump:smooth} does not hold and to $\cO(\epsilon^{-2})$ if Assumption~\ref{assump:smooth} holds.
\end{theorem}

Theorems~\ref{thm:saa_sample_complexity_mse} and~\ref{thm:rtmlmc_sample_complexity_mse} show that the MLMC estimator substantially outperforms the SAA estimator in terms of mean squared error-based scenario complexity. Specifically, for nonsmooth integrands, the MLMC estimator achieves a scenario complexity of~$\mathcal{O}(\log(\epsilon^{-1})^{2(T-1)}\epsilon^{-2})$, representing a significant improvement over the~$\mathcal{O}(\epsilon^{-2T})$ complexity of the SAA estimator. Similarly, for smooth integrands, the MLMC estimator improves the scenario complexity from~$\mathcal{O}(\epsilon^{-(T+1)})$ for SAA to~$\mathcal{O}(\epsilon^{-2})$. The MLMC estimator of Definition~\ref{def:MCCO_RTMLMC}, which discards all terms~$\hat{h}_t^{\ell}$, $\hat{h}_t^{\ell,{\rm e}}$ and~$\hat{h}_t^{\ell,{\rm o}}$ beyond the truncation point~$M_t$ for each~$t \in [T-1]$, also outperforms the \emph{untruncated} MLMC estimator proposed in~\cite{syed2023optimal}. While the untruncated estimator achieves a scenario complexity of~$\mathcal{O}(\epsilon^{-2})$ for smooth integrands, it suffers from infinite scenario complexity for nonsmooth integrands. In contrast, our MLMC estimator avoids this divergence by using truncation.

Note that while the scenario complexity of our MLMC estimator scales gracefully with~$1/\epsilon$, it exhibits exponential dependence on~$T$ when the integrands are nonsmooth. As a result, the expected computational cost grows rapidly with~$T$. This scaling behavior reflects the inherent difficulty of evaluating nested compositions of conditional expectations and nonsmooth functions. Nevertheless, for applications characterized by small values of~$T$, as discussed in Section~\ref{sec:motivational_ex}, the MLMC estimator can be computed to a moderate relative accuracy ({\em e.g.}, $10\%$ or $1\%$) at a practical computational cost. We emphasize, however, that Theorem~\ref{thm:rtmlmc_sample_complexity_mse} provides the first finite scenario complexity guarantee for problem~\eqref{problem:MCCE} with nonsmooth integrands that avoids exponential terms of the order $\epsilon^{-T}$. As these scenario complexity guarantees pertain only to expected errors, it is expedient to complement them with bounds that hold with high probability. As in Section~\ref{sec:saa}, such bounds are available for estimators with finite second moments under a sub-Gaussianity assumption.

\begin{assumption}[Sub-Gaussianity of $\hat H_1(x)$]
\label{assump:sub-gaussianity_mlmc} 
The estimator $\hat H_1(x)$, which was introduced in Definition~\ref{def:MCCO_RTMLMC}, is sub-Gaussian with variance proxy $\zeta^2>0$ for any fixed~$x\in\cX$.
\end{assumption}

Remark~\ref{rem:sub-gaussian-2} below outlines easily verifiable conditions under which Assumption~\ref{assump:sub-gaussianity_mlmc} holds.

\begin{rem}
\label{rem:sub-gaussian-2}
    \!Assumption~\ref{assump:sub-gaussianity_mlmc} holds under the same conditions that imply Assumption~\ref{assump:sub-gaussianity_saa} in Remark~\ref{rem:sub-gaussian}. Indeed, the sub-Gaussianity of $\hat H_t(x)$ can be shown by backward recursion on~$t\in[T-1]$, using the assumed sub-Gaussianity of~$\xi_t$ together with Lemma~\ref{lemma:sub-gaussianity-preserving-operations}, which ensures that sub-Gaussianity is preserved under finite sums, concatenations, Lipschitz continuous transformations and mixtures.
\end{rem}

Using Assumption~\ref{assump:sub-gaussianity_mlmc}, we now demonstrate that if the rates~$r_t$ and truncation points~$M_t$ for all~$t \in [T-1]$ are chosen as in Theorem~\ref{thm:rtmlmc_sample_complexity_mse}, then the worst-case absolute error of the MLMC estimator is bounded by~$\epsilon$ with high probability, where the failure probability tends to~$0$ as~$n_1$ increases. In the remainder of this section, we use~$L'\in\RR$ to denote the Lipschitz constant of the MLMC estimator~$\hat{F}(x)$ with respect to~$x$, which can be shown to exist by using similar arguments as in Lemma~\ref{lem:lipschitz}. Details are omitted for brevity.

\begin{lemma}
\label{lemma:unif_conv_rt-mlmc}
Suppose that Assumptions~\ref{assump:general}, \ref{assump:lipschitz}, \ref{assump:rate-parameters-rtmlmc-funval} and~\ref{assump:sub-gaussianity_mlmc} hold and that $\epsilon>0$. Assume also that the truncation points~$M_t$, $t\in[T-1]$ satisfy
\begin{align*}
     M_{t} = \bigg\lceil 2\log_2\Big( \textstyle 4 L_{[t]}\sqrt{C_{t+1}} \, \Big( \prod_{s=t+1}^{T-1} \sqrt{Z_s} \sqrt{M_s+1} \Big) (T-1)/\epsilon \Big) \bigg\rceil
\end{align*}
if Assumption~\ref{assump:smooth} does not hold and
\begin{align*}
     \textstyle M_{t}  = \bigg\lceil\log_2\bigg( 2 L_{[t-1]}S_{t}D_{t+1}^{2^{-t}} \, \bigg( \prod_{s=t+1}^{T-1}Z_s^{\frac{2^s-1}{2^t}} \Big(\frac{1-2^{-\frac{M_s+1}{2^s}}}{1-2^{-\frac{1}{2^s}}} \Big)^\frac{1}{2^t} \bigg)  (T-1)/\epsilon\bigg)\bigg\rceil
\end{align*}
if Assumption~\ref{assump:smooth} holds. Then, we have   
    \begin{align*}
    \PP\left(\sup_{x\in \mathcal{X}} \; \lvert \hat{F}(x) - F(x) \rvert > \epsilon \right) \leq
    2\left\lceil 8 L' D_\mathcal{X}/\epsilon +1\right\rceil^{d} \exp\bigg(- \frac{n_1\epsilon^2}{32\zeta^2} \bigg).
    \end{align*}
\end{lemma}
Recall that~$\zeta^2$ is the variance proxy of $\hat{H}_1(x)$ (see Assumption~\ref{assump:sub-gaussianity_mlmc}), while~$d$ and~$D_{\cX}$ are the dimension and the diameter of~$\cX$, respectively. The proof of Lemma~\ref{lemma:unif_conv_rt-mlmc} parallels that of Lemma~\ref{lemma:unif_conv_saa} and is thus omitted. Lemma~\ref{lemma:unif_conv_rt-mlmc} readily implies that the minimizers of problem~\eqref{prob:min_hatF(x)}, which optimizes the MLMC estimator over~$\cX$, are near-optimal in~\eqref{problem:MCCO} with high probability. Unfortunately, these minimizers are often difficult to compute even if~$F(x)$ is a convex function.

\begin{rem}
    Recall that~$F(x)$ as well as the SAA estimator from Section~\ref{sec:saa} are convex if the integrand~$f_t(\xi_t, x_t)$ is convex in~$x_t$ for every~$t \in [T]$ and non-decreasing in~$x_t$ for every~$t \in [T-1]$. Perhaps surprisingly, however, these conditions fail to guarantee that the MLMC estimator is convex. To see this, assume for example that~$T=2$, $f_1(\xi_1,x_1)=\exp(x_1)$ and $f_2(\xi_2,x_2)=\xi_2+x_2$. As $f_t(\xi_t,x_t)$ is convex and non-decreasing in~$x_t$ for every~$t\in[T]$, we conclude that~$F(x)$ and its SAA estimator are convex. Next, we construct the MLMC estimator with~$n_1=1$ scenario trees in the event~$\lambda_1=1$. This estimator requires a sample~$\xi_1^1$ from the marginal distribution of~$\xi_1$ and~$2^{\lambda_1}=2$ samples~$\xi_2^1$ and~$\xi_2^2$ from the conditional distribution of~$\xi_2$ given~$\xi_1=\xi_1^1$. By Definition~\ref{def:MCCO_RTMLMC}, the estimator can be represented as
    \begin{align*}
        \hat F(x)& = \frac{1}{q_1(\lambda_1)}\Big[\exp\big((\xi_2^1 + \xi_2^2)/2 +x \big) - \frac{1}{2}\exp\big(\xi_2^1+x \big) - \frac{1}{2}\exp\big( \xi_2^2+x \big) \Big]\\
        &= \frac{\exp(x)}{q_1(1)} \Big[ \exp\big((\xi_2^1 + \xi_2^2 )/2 \big) - \frac{1}{2}\exp\big(\xi_2^1 \big) - \frac{1}{2}\exp\big(\xi_2^2 \big) \Big],
    \end{align*}
    where the second equality holds because we condition on the event~$\lambda_1=1$. Assuming that $\xi_2^1\neq \xi_2^2$, Jensen's inequality implies that the last term in square brackets is negative, which in turn implies that~$\hat{F}(x)$ is nonconvex. Hence, minimizing the MLMC estimator over~$\cX$ may, with positive probability, lead to an intractable nonconvex optimization problem. We emphasize, however, that the MLMC estimator constitutes an unbiased estimator for the SAA estimator with branching factor~$2^{M_1}$, which is convex in this example. 
\end{rem}

We now demonstrate that the minimizers of the MLMC problem~\eqref{prob:min_hatF(x)} are, regardless of whether they can be computed efficiently, near-optimal in~\eqref{problem:MCCO} with high probability.

\begin{corollary}\label{cor:high_prob_conv_rtmlmc}
Assume that all conditions of Lemma~\ref{lemma:unif_conv_rt-mlmc} hold for some~$\epsilon>0$, and let $x^*$ and $\hat{x}^*$ be minimizers of \eqref{problem:MCCO} and the MLMC problem \eqref{prob:min_hatF(x)}, respectively. Then, we have
\begin{align*}
    \PP\big( F(\hat{x}^*) - F(x^*) > \epsilon \big) \leq 4\left\lceil 8L' D_\mathcal{X}/\epsilon +1\right\rceil^{d} \exp\bigg(- \frac{n_1\epsilon^2}{128\zeta^2} \bigg).
\end{align*}   
\end{corollary}

The proof of Corollary~\ref{cor:high_prob_conv_rtmlmc} parallels that of Corollary~\ref{cor:high_prob_bound_saa_epsilon_optimal} and is thus omitted. As the computational cost of the MLMC estimator is {\em random}, we generalize our definition of the high-probability scenario complexity as the
{\em expected} number of scenarios used by a method that is guaranteed to find an $\epsilon$-optimal solution of problem~\eqref{problem:MCCO} with probability~$1-\beta$. The next result follows immediately from Corollary~\ref{cor:high_prob_conv_rtmlmc}.

\begin{theorem}[High-Probability Scenario Complexity of the MLMC Estimator]
    \label{thm:rtmlmc_sample_complexity_high_probability}
  Suppose that all conditions of Lemma~\ref{lemma:unif_conv_rt-mlmc} hold for some~$\epsilon>0$ and that
    \begin{align*}
    n_1 =\left\lceil \frac{128\zeta^2}{\epsilon^2} \left(d \log\left( \left\lceil8L' D_\mathcal{X}/\epsilon +1\right\rceil\right) + \log\left(4/\beta \right) \right) \right\rceil
\end{align*}
for some $\beta\in(0,1)$. Then any solution of the MLMC problem \eqref{prob:min_hatF(x)} is an $\epsilon$-optimal solution of~\eqref{problem:MCCO} with probability at least $1-\beta$. Thus, the high-probability scenario complexity of the MLMC estimator equals $\cO(\log(\epsilon^{-1})^{T}\epsilon^{-2})$ if Assumption~\ref{assump:smooth} does not hold and $\cO(\log(\epsilon^{-1})\epsilon^{-2})$ if Assumption~\ref{assump:smooth} holds.
\end{theorem}

Theorems~\ref{thm:SAA_sample_complexity_high_probability} and~\ref{thm:rtmlmc_sample_complexity_high_probability} show that the MLMC estimator significantly outperforms the SAA estimator in terms of high-probability scenario complexity. Indeed, it improves the high-probability scenario complexity from~$\cO(\log(\epsilon^{-1})\epsilon^{-2T})$ to~$\cO(\log(\epsilon^{-1})^{T}\epsilon^{-2})$ for nonsmooth integrands and from~$\cO(\log(\epsilon^{-1})\epsilon^{-(T+1)})$ to~$\cO(\log(\epsilon^{-1})\epsilon^{-2})$ for smooth integrands. In contrast, the \emph{un}truncated MLMC estimator studied in~\cite{syed2023optimal} lacks a finite Lipschitz constant. Indeed, as is evident from~\eqref{eq:H_hat}, the Lipschitz modulus of~$\hat{H}_t(x)$ scales with~$1/q(\lambda_t)$. As the random variable~$\lambda_t$ has unbounded support when~$M_t=\infty$, $1/q(\lambda_t)$ can be arbitrarily large. Hence, $\hat{H}_t(x)$ admits no Lipschitz constant that applies uniformly across all realizations of~$\lambda_t$. Our bound on the high-probability scenario complexity of the untruncated MLMC estimator thus diverges.

Note that the high-probability guarantees of Theorem~\ref{thm:rtmlmc_sample_complexity_high_probability} for smooth integrands require a larger number~$n_1$ of scenario trees than the mean squared error-based scenario complexity results in Theorem~\ref{thm:rtmlmc_sample_complexity_mse}. Specifically, the high-probability scenario complexity incurs an additional multiplicative factor of~$\log(\epsilon^{-1})$. Despite this increase, the $\epsilon$-scaling of the MLMC estimator remains better than that of the SAA estimator.

For non‐smooth integrands, the high‐probability scenario complexity $\cO(\log(\epsilon^{-1})^T \epsilon^{-2})$ even grows slower with $\epsilon^{-1}$ than the mean squared error‐based sample complexity $\cO(\log(\epsilon^{-1})^{2(T-1)} \epsilon^{-2})$ from Theorem~\ref{thm:rtmlmc_sample_complexity_mse}. This phenomenon arises because of the sub-Gaussian tail assumption on the estimator~$\hat{H}_1(x)$ in Theorem~\ref{thm:rtmlmc_sample_complexity_high_probability}, which yields tighter control over the worst-case behavior than variance-based bounds. 

Finally, we emphasize that while the scenario complexity of the truncated MLMC estimator grows gracefully with~$\epsilon^{-1}$ for (moderate) fixed values of~$T$, it usually grows exponentially with~$T$. Such exponential growth appears to be unavoidable even for smooth integrands, in which case the scenario complexity can be shown to include a constant term that is exponential in~$T$. In addition, the untruncated MLMC estimator from~\cite{syed2023optimal} as well as the SAA estimator from Section~\ref{sec:saa} also display an exponential cost growth in~$T$. 



\section{Gradient Methods for MCCO}
\label{sec:mcco_optimization}
From Section~\ref{sec:mcco_estimation} we know that the minimizers of problem~\eqref{prob:min_hatF(x)}, equipped with the SAA or the MLMC estimator, are near-optimal in~\eqref{problem:MCCO} with high probability. Despite their favorable statistical properties, computing these minimizers can be challenging in the absence of restrictive convexity and monotonicity conditions. This motivates us in this section to develop efficient estimators for~$\nabla F(x)\in\RR^d$ that can be employed within first-order optimization methods for solving problem~\eqref{problem:MCCO} to stationarity. 

We aim to recursively construct MLMC gradient estimators using similar ideas as in Section~\ref{sec:rtmlmc}. These estimators are again based on a forest of scenario trees with a random branching structure, and the underlying log-branching factors are modeled by truncated geometric distributions. The resulting estimators are biased, but the bias can be reduced by increasing the truncation points. In the special case where the truncation points are set to infinity, we obtain {\em un}truncated, and therefore {\em un}biased, estimators. Both the truncated and untruncated gradient estimators proposed here are novel for problems with more than two stages ($T>2$). Truncated as well as untruncated MLMC gradient estimators for two-stage problems are discussed in~\citep{hu2021bias, goda2023constructing}. As will become clear below, the generalization to~$T\geq 2$ poses significant technical challenges.

Throughout, we work under Assumptions~\ref{assump:lipschitz} and~\ref{assump:smooth}, which guarantee the existence and well-defined\-ness of the gradients. We first show that the gradient~$\nabla F(x)$ can be constructed via the backward recursion
\begin{align}\label{eq:G_t}
    G_t(x) \coloneqq \EE_{t-1}\left[G_{t+1}(x) \nabla_{x_t} f_t(\xi_t, F_{t+1}^T(x))\right]\in\RR^{d\times d_{t-1}}\quad \forall t\in[T-1]
\end{align}
initialized by $G_T(x) \coloneqq \EE_{T-1}[\nabla_x f_T(\xi_T, x)]\in\RR^{d\times d_{T-1}}$, where~$F_{t+1}^T(x)$ represents the nest of the last $T-t$ conditional expectations and nonlinear integrands in~\eqref{problem:MCCO}. Hence, we have $F(x)=F^T_1(x)$; see also Definition~\ref{def:risk-mapping} in the appendix. All gradients in~\eqref{eq:G_t} exist by Assumption~\ref{assump:smooth}. Moreover, since $\nabla_{x_t} f_t(\xi_t, x_t)$ is uniformly bounded under Assumption~\ref{assump:lipschitz}, the same holds for $G_t(x)$. Consequently, the conditional expectation in~\eqref{eq:G_t} is well-defined. We can now prove that~$\nabla F(x)= G_1(x)$.

\begin{lemma}
\label{lem:gradient-recursion}
If Assumptions~\ref{assump:general}, \ref{assump:lipschitz} and~\ref{assump:smooth} hold, then $\nabla F(x)= G_1(x)$ for all $x\in\RR^{d}$.
\end{lemma}

Lemma~\ref{lem:gradient-recursion} shows that the gradient~$\nabla F(x)$ admits a recursive representation analogous to that of~$F(x)$ itself. This observation allows us to construct gradient estimators using methods similar to those in Section~\ref{sec:mcco_estimation}. In particular, following the analogy with Definition~\ref{def:MCCO_RTMLMC}, we can define an MLMC estimator for~$\nabla F(x)$.

\begin{defn}[MLMC Gradient Estimator]
\label{def:rumlmc-grad}
For any $t\in[T-1]$, select a rate parameter~$r_t \in (0,1)$ and truncation point $M_t\in\mathbb N$, and let $\lambda_t$ be a random variable that follows the truncated geometric distribution~$\mathrm{Geo}(r_t | M_t)$. Assume that~$\lambda_t$ is independent of all other random objects, and set $q_t(\ell)\coloneqq\PP(\lambda_t=\ell)$ for~$\ell=0,\ldots,M_t$ and~$t\in[T-1]$. We construct gradient estimators $\hat G_t(x)$, $t\in[T]$, using the following backward recursion. We initialize the recursion by $\hat G_T(x)\coloneqq\nabla_x f_T(\xi_T,x)$. Next, for any $t\in[T-1]$, we~set
\begin{align}
    \hat{G}_t(x) \coloneqq \frac{1}{q_{t}(\lambda_{t})}\left( \hat{g}_t^{\lambda_t}(x)-\frac{1}{2}\hat{g}_t^{\lambda_t, {\rm e}}(x)-\frac{1}{2}\hat{g}_t^{\lambda_t, {\rm o}}(x)\label{eq:G_hat}\right),
\end{align}
where
\begin{equation*}
    \hat{g}_t^{\ell}(x) \coloneqq \hat\EE_t^{\ell}[\hat{G}_{t+1}(x) ]\; \nabla_{x_t} f_t \big(\xi_t, \hat\EE_t^{\ell}[ \hat{H}_{t+1}(x)] \big)\quad \forall \ell=0,\ldots,M_t
\end{equation*}
as well as
\begin{equation*}
\begin{aligned}
     \hat{g}_t^{\ell, {\rm e}}(x) \coloneqq& \hat\EE_t^{\ell, {\rm e}}[ \hat{G}_{t+1}(x)] \; \nabla_{x_t} f_t \big(\xi_t, \hat\EE_t^{\ell, {\rm e}}[\hat{H}_{t+1}(x)] \big)\quad \forall \ell\in[M_t], \\
     \hat{g}_t^{\ell, {\rm o}}(x) \coloneqq& \hat\EE_t^{\ell, {\rm o}}[ \hat{G}_{t+1}(x)] \; \nabla_{x_t} f_t \big(\xi_t, \hat\EE_t^{\ell, {\rm o}}[\hat{H}_{t+1}(x)] \big)\quad \forall \ell\in[M_t].
\end{aligned}
\end{equation*}
We also set $\hat{g}_t^{0,{\rm e}}(x)\coloneqq\hat{g}_t^{0,{\rm o}}(x)\coloneqq0$. The MLMC estimator for $\nabla F(x)$ is then defined as $ \hat G(x)\coloneqq\hat\EE [\hat{G}_{1}(x)]$. 
\end{defn}

A special case arises when the truncation points in Definition~\ref{def:rumlmc-grad} are set to $M_t = \infty$ for all $t \in [T-1]$. In this case, we refer to the resulting estimator as the \emph{untruncated} MLMC gradient estimator. Recall that~$\hat{H}_t(x)$ was introduced in Definition~\ref{def:MCCO_RTMLMC} and can be viewed as an estimator for the conditional risk mapping~$F_t^T(x)$ from Definition~\ref{def:risk-mapping}. Similarly, $\hat{G}_t(x)$ constitutes an estimator for~$G_t(x)$. Crucially, the construction above reveals an important structural feature, namely that the gradient estimator $\hat{G}_t(x)$ depends on the function value estimator $\hat{H}_{t+1}(x)$ as an argument at each stage, and both sequences of estimators must therefore be maintained and propagated simultaneously. This coupling is a novel structural feature of our construction, \emph{absent} from classical MLMC gradient estimators. Both $\hat{G}_t(x)$ and $\hat{H}_t(x)$ are built on the same forest of~$n_1$ scenario trees displaying a random branching structure with $2^{\lambda_t}$ branches per stage-$t$ node, where~$\lambda_t$ is sampled from a truncated geometric distribution. Note that the forest of trees is defined implicitly through the nesting of the empirical conditional expectations $\hat\EE[\cdot]$ and $\hat\EE_t^{\lambda_t}[\cdot]$ for~$t\in[T-1]$, which are used both for the construction of the function value and gradient estimators. This implies that~$\hat{G}_t(x)$ and~$\hat{H}_t(x)$, $t\in[T-1]$, are constructed from the same samples and are therefore correlated. Reusing the same samples across both estimators reduces the sampling cost, but it complicates the analysis of the trade-off between variance and computational cost. Our analysis uses the following H\"older condition.

\begin{assumption}[Hölder Continuous Hessians]\label{assump:holder}
     The Hessian matrix $\nabla^2_{x_t} f_{t,i} (\xi_t,x_t)$ is $\rho_t$-Hölder continuous in $x_t$ across all $\xi_t$. That is, for all $t\in[T-1]$, there exist $R_t>0$ and $\rho_t\in( 1-2^{1-t}, 1]$ such that
    \begin{equation*}
        \|\nabla^2_{x_t} f_{t,i} (\xi_t,x_t)-\nabla^2_{x_t}  f_{t,i}(\xi_t,y_t)\|_2 \leq R_t \|x_t-y_t\|_2^{\rho_t} \quad \forall x_t,y_t\in\RR^{d_t}, \;\forall \xi_t\in\RR^{m_t}, \; \forall i\in[d_{t-1}].
    \end{equation*}
\end{assumption}

Note that when $\rho_t=1$, Assumption~\ref{assump:holder} requires the Hessian matrices of the components of $f_t(\xi_t,x_t)$ to be Lipschitz continuous in~$x_t$. Assumption~\ref{assump:holder} has been used in~\cite{goda2023constructing} to control the variance and the sampling costs of MLMC gradient estimators for conditional stochastic optimization problems, which arise as special cases of~\eqref{problem:MCCO} with~$T=2$ nests. It implies via the fundamental theorem of calculus that
\begin{align}\label{eq:inequality_for_holder_funs}
    \left\|\nabla_{x_t} f_{t,i} (\xi_t,x_t)-\nabla_{x_t} f_{t,i} (\xi_t,y_t) -\nabla^2_{x_t} f_{t,i} (\xi_t,y_t) (x_t-y_t) \right\|_2 \leq \frac{R_t}{\rho_t +1}\|x_t-y_t\|_2^{\rho_t+1}
\end{align}
for all $x_t,y_t\in\RR^{d_t}$ and $\xi_t\in\RR^{m_t}$. 

In order to derive bounds on the mean squared error-based scenario complexity of the MLMC gradient estimator, we first analyze its variance, bias, and sampling cost. In analogy with~\eqref{eq:mu}, for any $p \geq 1$, we define the $p$-th conditional moment of~$\hat{G}_t(x)$ given the information available up to time~$t-1$ by
\begin{align}\label{eq:nu}
    \nu_t^{p}(x) \coloneqq \EE_{t-1}\!\left[ \|\hat{G}_t(x)\|_2^p \right].
\end{align}
Recall that $\hat{G}_t(x)\in\RR^{d \times d_{t-1}}$ is a matrix, and thus $\|\hat{G}_t(x)\|_2$ stands for its Frobenius norm. Note also that the $p$-th conditional moment of~$\hat{H}_t(x)$ in~\eqref{eq:mu} was defined with respect to the vector $p$-norm. To facilitate the analysis of the MLMC gradient estimator, we define the conditional moment bound of order~$p$ as
\begin{align*}
    \overline\nu_T^{p}\coloneqq \sup_{x\in\RR^{d}}\esssup \; \EE_{T-1}\! \left[ \big\|\nabla_x f_T(\xi_T,x)\big\|^{p}_{p} \right].
\end{align*}
The following assumption is crucial for our subsequent results.
\begin{assumption}\label{assump:nu_bar}
    The conditional moment bound $\overline{\mu}_T^p $ of the integrand $ f_T(\xi_T,x)$ is finite up to order $p=2^T(\rho_{T-1}+1)$, and the conditional moment bound $\overline{\nu}_T^p $ of the gradient $\nabla f_T(\xi_T,x)$ is finite up to order $p=2^T$.
\end{assumption}
Note that Assumption~\ref{assump:nu_bar} implies Assumption~\ref{assump:mu_bar} because~$\rho_{T-1}>0$. The following lemma establishes a bound on~$\nu_t^{2^t}(x)$ for all~$t \in [T]$. In particular, for~$t = 1$, this bound can be used to derive an upper bound on the unconditional variance of the MLMC gradient estimator. 

\begin{lemma}
\label{lemma:bound_on_variance_rumlmc_grad} 
Suppose that Assumptions~\ref{assump:general}, \ref{assump:lipschitz}, \ref{assump:smooth}, \ref{assump:holder} and~\ref{assump:nu_bar} hold. Then, for every~$t\in[T-1]$ there exists~$E_t$ that depends only on~$T$, the moment bounds from Assumption~\ref{assump:nu_bar}, the dimensions~$d_s$ and smoothness constants~$S_s$ for $s = t+1,\ldots,T-1$, and the Hölder constants~$R_s$ and~$\rho_s$ for~$s = t,\ldots,T$ such that
    \begin{align*}
        \nu_t^{2^t}(x) \leq  E_t\prod_{s=t}^{T-1}\max\bigg\{ \sum_{\ell=0}^{M_s}\frac{1}{2^{(2^{s-1})(\rho_s+1)\ell} q_s(\ell)^{2^{s}-1}}, \sum_{\ell=0}^{M_s}\frac{1}{2^{2^s(\rho_{s-1}+1) \ell} q_s(\ell)^{2^{s}(\rho_{s-1}+1)-1}} \bigg\}
    \end{align*}
for all $x\in\cX$ and $t \in[T]$. Here, we adopt the convention~$\rho_0 = 0$ for notational convenience.
\end{lemma}

The exact closed-form expressions for the constants $E_t$ are cumbersome and thus omitted from the statement of Lemma~\ref{lemma:bound_on_variance_rumlmc_grad}. Instead, an explicit recursive definition is given in the proof. Lemma~\ref{lemma:bound_on_variance_rumlmc_grad} allows us to bound the variance of the MLMC gradient estimator $\hat G(x)=\hat\EE[\hat G_1(x)]$. Indeed, as~$\hat G(x)$ constitutes an average of~$n_1$ independent copies of~$\hat G_1(x)$, its variance admits the bound
\begin{align}\label{eq:variance-rumlmc-estimator}
    \VV(\hat{G}(x))=\frac{1}{n_1}\VV(\hat{G}_1(x))\leq  \frac{\nu_1^2(x)}{n_1}, 
\end{align}
while Lemma~\ref{lemma:bound_on_variance_rumlmc_grad} provides a bound on~$\nu_1^2(x)$. Note that this bound is finite only if $\nabla f_T(\xi_T, x)$ admits finite moments of order $2^T$, and $f_T(\xi_T, x)$ admits finite moments of order $2^T(\rho_{T-1}+1)$. We further emphasize that Lemma~\ref{lemma:bound_on_variance_rumlmc_grad} also applies to the {\em untruncated} MLMC gradient estimator obtained by setting $M_t = \infty$ for all $t \in [T-1]$. In this case, however, the resulting variance bound remains finite only if the rate parameters $r_t$, $t \in [T-1]$, are sufficiently large to ensure convergence of all infinite series appearing in the bound.

Lemma~\ref{lemma:bound_on_variance_rumlmc_grad} is closely related to Lemma~\ref{lemma:bound_on_variance_rtmlmc}, as both establish bounds on conditional moments of MLMC estimators by means of backward recursions in time. The key differences lie in the underlying assumptions and in the complexity of the resulting recursions. For function-value estimators, Lemma~\ref{lemma:bound_on_variance_rtmlmc} requires only Lipschitz continuity of the integrands, yielding relatively simple product-form bounds that admit closed-form expressions. If the integrands are also smooth, the bounds strengthen while retaining a simple closed form. By contrast, for gradient estimators, Lemma~\ref{lemma:bound_on_variance_rumlmc_grad} requires not only smoothness of the integrands but also Hölder continuity of their Hessian matrices, which introduces the Hölder exponents~$\rho_t$ into the analysis and makes the proof technically more demanding. Nevertheless, both lemmas serve the same purpose. They provide explicit upper bounds on conditional moments that can be used to control the growth of the estimator’s variance and to establish scenario complexity guarantees. When $\rho_t=1$ for all $t\in[T-1]$ ({\em i.e.}, when all Hessian matrices are Lipschitz continuous), the second term in the max operator of the moment bound in Lemma~\ref{lemma:bound_on_variance_rumlmc_grad} dominates the first term for all $s\in[T-1]$—recall that $\rho_{0}=0$—which implies that the conditional moment bounds for the gradient estimator impose stricter requirements on the rate parameters $r_s$ governing $q_s(\ell)$ than those for the function-value estimator in Lemma~\ref{lemma:bound_on_variance_rtmlmc}.

Next, we analyze the bias of the MLMC gradient estimator using the moment bounds established in Lemmas~\ref{lemma:bound_on_variance_rtmlmc} and~\ref{lemma:bound_on_variance_rumlmc_grad}. This analysis is cumbersome because~$\hat{g}_t^\ell$ and its variants based on only even- or odd-indexed samples, $\hat{g}_t^{\ell, {\rm e}}$ and~$\hat{g}_t^{\ell, {\rm o}}$, respectively (see Definition~\ref{def:rumlmc-grad}), exhibit a product structure. More precisely, any bias in~$\hat{H}_{t+1}$ is propagated through the gradient mapping~$\nabla_{x_t} f_t(\xi_t,\cdot)$, and its magnitude is thus controlled by the smoothness of~$f_t$. Furthermore, the biases of~$\hat{G}_{t+1}$ and~$\hat{H}_{t+1}$ interact multiplicatively, giving rise to cross terms that must be carefully controlled. This interplay significantly complicates the analysis compared with the case of the function value estimators studied in Section~\ref{sec:rtmlmc}.

\begin{lemma}\label{lemma:bias_bound_rtmlmc_grad}
    If Assumptions~\ref{assump:general}, \ref{assump:lipschitz}, \ref{assump:smooth} and~\ref{assump:nu_bar} hold, the bias of the MLMC gradient estimator satisfies
    \begin{align*}
        \big\| \EE[\hat{G}(x)]- G_1(x) \big\|_2 \leq \sum_{t=1}^{T-1}L_{[t-1]}S_t \EE[\nu_{t+1}^{2}(x)]^\frac{1}{2}\bigg(  \frac{ \EE[\mu_{t+1}^{2}(x)]^\frac{1}{2}}{2^{M_t/2}}   + \sum_{s=t+1}^{T-1} \frac{L_{[t+1:s-1]}S_s\EE[(\mu_{s+1}^{2}(x))^2]^\frac{1}{2}}{2^{M_s+1}} \bigg).
    \end{align*}
\end{lemma}


As in the case of the MLMC function value estimator discussed in Section~\ref{sec:rtmlmc}, Lemma~\ref{lemma:bias_bound_rtmlmc_grad} shows that the MLMC gradient estimator is unbiased if $M_t=\infty$ for all~$t\in[T-1]$ provided that all expectations in the bias bound are finite. Hence, the \emph{untruncated} MLMC gradient estimator is \emph{unbiased}. We now quantify the expected sampling cost of the MLMC gradient estimator.

\begin{lemma}\label{lemma:bound_on_cost_rumlmc_grad} The expected number of scenarios needed to construct the MLMC gradient estimator equals
\begin{equation*}
    \EE\big[\CC\big(\hat{G}(x)\big)\big]= n_1\prod_{t=1}^{T-1} \sum_{\ell=0}^{M_t} q_{t}(\ell) 2^{\ell}.
\end{equation*}
\end{lemma}
The proof of Lemma~\ref{lemma:bound_on_cost_rumlmc_grad} widely parallels that of Lemma~\ref{lemma:bound_on_cost_rtmlmc} and is thus omitted. As in Section~\ref{sec:rtmlmc}, we now tune the hyperparameters~$n_1$, $r_t$ and~$M_t$ of the MLMC estimator to guarantee a mean squared error of at most~$\epsilon^2$ at minimal computational cost. To achieve this, we select sufficiently large truncation points~$M_t$, $t\in[T-1]$, to ensure that the bias bound of Lemma~\ref{lemma:bias_bound_rtmlmc_grad} is at most~$\epsilon/\sqrt{2}$. We also select a sufficiently large number~$n_1 \approx 2\nu_1^2(x)/\epsilon^2$ of scenario trees to ensure that the variance is at most~$\epsilon^2/2$; see~\eqref{eq:variance-rumlmc-estimator}. By Lemma~\ref{lemma:bound_on_cost_rumlmc_grad}, the expected sampling cost can thus be approximated by
\[
    \CC(\hat G(x)) \approx \frac{ 2\nu_1^2(x)}{\epsilon^2} \prod_{t=1}^{T-1} \sum_{\ell=0}^{M_t} q_t(\ell)2^{\ell}.
\]
We can construct an approximate upper bound on~$\CC(\hat G(x))$ uniformly across all~$x\in\cX$ by replacing~$\nu_1^2(x)$ with the respective bound given in Lemma~\ref{lemma:bound_on_variance_rumlmc_grad}. The resulting upper bound~$c_{\mathrm{h}}(r)$ depends on the rate parameters~$r = (r_1, \ldots, r_{T-1})$ of the truncated geometric distributions used to model the log-branching factors. Here, the subscript ``h" indicates that the bound was obtained under the assumption of Hölder continuous Hessians. Ideally, $r\in (0,1)^{T-1}$ should be selected to minimize the given upper bound. As in Section~\ref{sec:rtmlmc}, closed-form minimizers are elusive. The next assumption specifies a simple near-optimal choice for~$r$.

\begin{assumption}[Rate Parameters] 
\label{assump:rate-parameters-rumlmc-grad}
For any $t\in[T-1]$, we select $r_t$ satisfying
    \[
       2^{-1} < r_t < \min\bigg\{1-2^{-\frac{2^{t-1}(\rho_t+1)}{2^t-1}},\; 1-2^{-\frac{2^t(\rho_{t-1}+1)}{2^t(\rho_{t-1}+1)-1}}\bigg\}.
    \]
\end{assumption}
A rate parameter~$r_t$ that satisfies both strict inequalities is guaranteed to exist whenever Assumption~\ref{assump:holder} holds. Indeed, in this case the H\"older exponent satisfies~$\rho_t > 1-2^{1-t}$, which guarantees that the first term in the minimum is strictly greater than~$2^{-1}$. The second term is strictly greater than~$2^{-1}$ for any~$\rho_t\geq 0$. Assumption~\ref{assump:rate-parameters-rumlmc-grad} ensures that the given rate parameters jointly control the geometric growth of the second moment as well as the expected sampling cost of~$\hat{G}(x)$. Specifically, Lemma~\ref{lemma:bound_on_variance_rumlmc_grad} implies that~$\nu_1^2$ remains bounded in~$M_t$ if $r_t$ obeys the given upper bound in Assumption~\ref{assump:rate-parameters-rumlmc-grad}, and Lemma~\ref{lemma:bound_on_cost_rumlmc_grad} implies that the expected sampling cost is bounded in~$M_t$ if $r_t > 2^{-1}$ for all~$t \in [T-1]$.  Note also that if~$M_t=\infty$ for all~$t\in[T-1]$, then Assumption~\ref{assump:rate-parameters-rumlmc-grad} guarantees that the infinite sums in Lemmas~\ref{lemma:bound_on_variance_rumlmc_grad} and~\ref{lemma:bound_on_cost_rumlmc_grad} converge. Hence, the variance and the expected sampling cost of the untruncated MLMC gradient estimator are finite. 

The MLMC gradient estimator is fully characterized by the number~$n_1$ of scenario trees as well as the rates~$r_t$ and truncation points~$M_t$ governing the geometric distributions of the log-branching factors~$\lambda_t$ for~$t \in [T-1]$. The following theorem presents a specific configuration of these hyperparameters that ensures the root mean squared error of the truncated MLMC gradient estimator to be bounded by~$\epsilon$.

\begin{theorem}[Mean Squared Error-Based Scenario Complexity of the MLMC Gradient Estimator]\label{thm:rtmlmc_grad_sample_complexity_mse}
    Let Assumptions~\ref{assump:general}, \ref{assump:lipschitz}, \ref{assump:smooth}, \ref{assump:holder}, \ref{assump:nu_bar} and~\ref{assump:rate-parameters-rumlmc-grad} hold, and let~$\epsilon>0$ be a given tolerance. Then, for every~$t\in[T-1]$ there exists~$W_t$ that depends only on~$T$, the moment bounds from Assumption~\ref{assump:nu_bar}, the dimensions~$d_s$ and smoothness constants~$S_s$ for~$s = t+1,\ldots,T-1$, the Lipschitz constants~$L_s$ and the rate parameters~$r_s$ for~$s \in[T-1]$, and the Hölder constants~$R_s$ and~$\rho_s$ for~$s = t,\ldots,T$ such that the following holds. If
    \[
        n_1=\lceil \sup_{x\in\cX} 2\nu_1^2(x)/\epsilon^2\rceil\quad \text{and} \quad M_{t}  =\lceil 2\log_2(W_t /\epsilon)\rceil\quad \forall t\in[T-1],
    \]
    then the MLMC gradient estimator achieves a root mean squared error of at most~$\epsilon$ uniformly for all~$x\in\cX$. Thus, its mean squared error-based scenario complexity is at most~$\cO(\epsilon^{-2})$.
\end{theorem}

The exact closed-form expressions for the constants~$W_t$ are cumbersome and thus omitted from the statement of Theorem~\ref{thm:rtmlmc_grad_sample_complexity_mse}. The explicit construction of these constants will become clear in the proof.

Theorem~\ref{thm:rtmlmc_grad_sample_complexity_mse} establishes that, if the  integrands have Hölder continuous Hessians, then the MLMC gradient estimator achieves the optimal scenario complexity of~$\cO(\epsilon^{-2})$. This may be surprising because the MLMC gradient estimator has a significantly more complex structure than the MLMC function value estimator, implying that it is more difficult to control its variance; see also the discussion after Lemma~\ref{lemma:bound_on_variance_rumlmc_grad}. We emphasize that Theorem~\ref{thm:rtmlmc_grad_sample_complexity_mse} remains valid for the \emph{untruncated} MLMC gradient estimator, which is obtained by setting~$M_t = \infty$ for all~$t \in [T-1]$. This result is  new for~$T > 2$. When~$T = 2$, Theorem~\ref{thm:rtmlmc_grad_sample_complexity_mse} recovers the~$\cO(\epsilon^{-2})$ scenario complexity for estimating the gradients of conditional stochastic optimization problems established in~\cite[Theorem~2]{goda2023constructing}. The H\"older continuity of the Hessians is a crucial assumption for Theorem~\ref{thm:rtmlmc_grad_sample_complexity_mse}. 
Indeed, if the integrands are merely Lipschitz continuous and smooth, then \emph{no} choice of~$r_t$, $t \in [T-1]$, guarantees that the variance and the expected sampling cost remain simultaneously bounded~\citep{hu2021bias}. 
Hence, the scenario complexity of the \emph{untruncated} MLMC gradient estimator diverges.

\begin{rem} \label{rem:extension_to_untruncated_grad}
    The MLMC gradient estimator introduced in Definition~\ref{def:rumlmc-grad} employs the same set of samples to construct both~$\hat{G}_t(x)$ and~$\hat{H}_t(x)$. This shared use of samples reduces the expected sampling cost. However, it also induces correlation between~$\hat{G}_t(x)$ and~$\hat{H}_t(x)$, which complicates the analysis of the scenario complexity and necessitates the assumption that the integrands possess Hölder continuous Hessians. Alternatively, one could construct a simpler MLMC gradient estimator that employs independent sets of samples to compute the components~$\hat{G}_t(x)$ and~$\hat{H}_t(x)$. The independence between these components simplifies the moment bounds in Lemma~\ref{lemma:bound_on_variance_rumlmc_grad} and, in fact, eliminates the need to establish moment bounds of orders~$p > 2$. Consequently, it suffices to derive a second-order moment bound analogous to that obtained in Lemma~\ref{lemma:bound_on_variance_rtmlmc} for nonsmooth integrands. However, this alternative estimator entails a higher expected sampling cost. One can show that its mean squared error-based scenario complexity amounts to $\cO(\log(\epsilon^{-1})^{2(T-1)}\epsilon^{-2})$, which matches that of the function value estimator for nonsmooth integrands from Theorem~\ref{thm:rtmlmc_sample_complexity_mse}. The proof follows similar arguments to those used in Theorem~\ref{thm:rtmlmc_grad_sample_complexity_mse}, albeit in a simpler form. Detailed derivations are therefore omitted for brevity. We emphasize that, unless the integrands have Hölder continuous Hessians, the truncation points of the simpler MLMC gradient estimator cannot be sent to infinity without causing either the variance or the expected sampling cost to diverge. Hence, in this simpler setting, it is impossible to construct an untruncated unbiased estimator without imposing additional regularity conditions.  
\end{rem}

We can now show that the MLMC gradient estimator introduced in Definition~\ref{def:rumlmc-grad} can be incorporated into an SGD algorithm (see Algorithm~\ref{algo:sgd}) to efficiently compute an $\epsilon$-stationary point of problem~\eqref{problem:MCCO}. In the following, we use $S_{[T]}\coloneqq\sum_{t=1}^{T} L_{[1:t-1]} S_t L_{[t+1:T]}^2$ as a measure for the smoothness of~$F(x)$, where~$L_{[s:t]}$ for $t\in[T]$ and $s\in[T+1]$ is defined as before Lemma~\ref{lemma:bound_on_variance_rtmlmc}. For more details see Lemma~\ref{lem:lipschitz}.

\begin{theorem}
\label{theorem:sgd-num-steps_rt-mlmc}
Let Assumptions~\ref{assump:general}, \ref{assump:lipschitz}, \ref{assump:smooth}, \ref{assump:holder}, \ref{assump:nu_bar} and~\ref{assump:rate-parameters-rumlmc-grad} hold, and set $\overline{\nu}_1^2:=\sup_{x\in\cX} \nu_1^2(x)$. Assume that Algorithm~\ref{algo:sgd} uses a constant stepsize $\eta=\overline{\nu}_1 \sqrt{n_1/K}$ and uses the MLMC gradient estimator with~$n_1=\cO(1)$ and~$M_t$, $t\in[T-1]$, defined as in Theorem~\ref{thm:rtmlmc_grad_sample_complexity_mse}. If Algorithm~\ref{algo:sgd} is initialized at~$x_1\in\cX$ and runs over
\begin{align*}
    K\geq \frac{\overline{\nu}_1^2}{n_1\epsilon^4} \Big(2 \Big[F(x_1)-\min_{x\in\cX}F(x) \Big]+S_{[T]} \Big)^2 
\end{align*}
iterations, then it outputs an $\epsilon$-stationary point~$\bar x_K$ with $\EE[\|\nabla F(\bar{x}_K)\|^2_2] \leq \epsilon^2$. 
\end{theorem}

The last part of the proof of Theorem~\ref{thm:rtmlmc_grad_sample_complexity_mse} shows that if $n_1 = \cO(1)$, then the expected sampling cost of constructing a single MLMC gradient estimator is independent of~$\epsilon$ (\textit{i.e.}, $\cO(1)$). Theorem~\ref{theorem:sgd-num-steps_rt-mlmc} thus implies that Algorithm~\ref{algo:sgd} requires~$\cO(\epsilon^{-4})$ scenarios to compute an $\epsilon$-stationary point of~\eqref{problem:MCCO}. Theorem~\ref{theorem:sgd-num-steps_rt-mlmc} generalizes a known scenario complexity result for conditional stochastic optimization problems, which correspond to instances of~\eqref{problem:MCCO} with~$T=2$ stages \citep[Corollary~4.1]{hu2021bias}. When employing {\em un}truncated MLMC gradient estimators, obtained by setting $M_t = \infty$ for all $t \in [T-1]$, a general complexity result for constrained nonconvex stochastic optimization with unbiased gradient oracles~\citep[Corollary~4]{ghadimi2016mini} implies that an $\epsilon$-stationary point of problem~\eqref{problem:MCCO} can still be computed in~$\cO(\epsilon^{-4})$ iterations. This convergence rate is known to be optimal for unbiased gradient oracles in the absence of variance reduction~\citep{arjevani2023lower}. Variance-reduction techniques can further improve the optimal rate for unbiased gradient oracles to~$\cO(\epsilon^{-3})$~\citep{arjevani2023lower}.

\begin{algorithm}
\caption{Projected SGD}\label{algo:sgd}
\begin{algorithmic}
\Require Number of iterations $K$, stepsize schedule $\{\eta_k\}_{k=1}^K$, initial iterate~$x_1$
\For{$k=1,\ldots,K$}
\State Construct an estimator $\hat{G}(x_k)$ for the gradient $\nabla F(x_k)$
\State Update $x_{k+1}=\Pi_\cX (x_k-\eta_k\hat{G}(x_k))$, where $\Pi_\cX$ is the Euclidean projection on~$\cX$
\EndFor
\Ensure $\bar{x}_K$ chosen uniformly at random from  $\{x_k\}_{k=1}^K$
\end{algorithmic}
\end{algorithm}

\section{Numerical Experiments} \label{sec:experiments}
We now assess the benefits of our MLMC function value and gradient estimators for solving MCCE and MCCO problems, and we experimentally validate our theoretical results. All experiments are conducted on a MacBook Pro equipped with an Apple M1 Max chip and 32 GB RAM, running macOS Sonoma~14.5. Code and data for reproducing the numerical experiments are available at \url{https://github.com/busesen/MCCO}. 


\paragraph{MCCE with a Known Ground-Truth Solution}
The first experiment is centered around a synthetic instance of~\eqref{problem:MCCE}, originally introduced in \citep[\S~3.1]{syed2023optimal}, with~$T=3$ stages and integrands $f_1(\xi_1,x_1)=\sin(\xi_1+x_1)$, $f_2(\xi_2,x_2)=\sin(\xi_2-x_2)$ and $f_3(\xi_3,x)=\xi_3$. The disturbances follow a Markovian structure with $\xi_1\sim \cN(\pi/2,1)$, $\xi_2|\xi_1\sim \cN(\xi_1,1)$ and $\xi_3|\xi_2\sim \cN(\xi_2,1)$. This problem has the closed-form solution $F(x)=\exp(-1/2)\approx 0.6065$, against which the accuracy of different estimators can be assessed.

We compare two SAA estimators in the sense of Definition~\ref{def:MCCO_sample_average_approximation} with two MLMC estimators in the sense of Definition~\ref{def:MCCO_RTMLMC}. The first SAA estimator (SAA1) employs a uniform branching structure with $n_1=n_2=n_3$, while the second SAA estimator (SAA2) emphasizes the first stage by setting $n_1=n_2^2=n_3^2$; see \cite{syed2023optimal,rainforth2018nesting} for motivation and discussion. The first MLMC estimator, taken from \cite{syed2023optimal}, is untruncated, uses the recommended rate parameters $r_1=0.74$ and $r_2=0.60$, and sets $M_1=M_2=\infty$, resulting in an expected sampling cost of $4.6250$ per each of the $n_1$ underlying scenario trees. The second MLMC estimator is truncated, adopts the rates $r_1=1-2^{-3/2}$ and $r_2=1-2^{-5/4}$ proposed in Assumption~\ref{assump:rate-parameters-rtmlmc-funval} for smooth integrands, and sets $M_1=6$ and $M_2=5$, yielding an expected sampling cost of $4.7674$ per scenario tree. The truncation points are chosen to satisfy $M_1> M_2$, in accordance with Lemma~\ref{lemma:unif_conv_rt-mlmc}, while maintaining an expected sampling cost per tree comparable to that of the untruncated MLMC estimator.

For each of the four estimators, we consider a sequence of instances indexed by~$n_1$, yielding decreasing mean squared errors as~$n_1$ increases. For the SAA estimators, the associated scenario trees become increasingly bushy, and arguments analogous to those in Theorem~\ref{thm:saa_sample_complexity_mse} show that their mean squared errors converge to~$0$, implying asymptotic unbiasedness. By Lemma~\ref{lemma:bias_bound_rtmlmc}, the untruncated MLMC estimator is unbiased for every fixed~$n_1$, and similar reasoning as in Theorem~\ref{thm:rtmlmc_sample_complexity_mse} establishes that its mean squared error, which coincides with its variance, also vanishes as~$n_1$ grows. In contrast, the truncated MLMC estimator exhibits a nonvanishing bias as~$n_1$ grows. Thus, increasing~$n_1$ serves only to reduce its variance. 

Figure~\ref{fig:toy_example_estim}\,(a) displays the two MLMC estimators (lines) together with their $95\%$ confidence intervals (shaded regions) as functions of~$n_1$. Since any MLMC estimator $\hat F(x)$ is formed by averaging $n_1$ independent tree-wise estimators with finite second moments (see Definition~\ref{def:MCCO_RTMLMC} and Lemma~\ref{lemma:bound_on_variance_rtmlmc}), the central limit theorem implies approximate normality, and we therefore report confidence bounds of the form $\hat F(x)\pm 1.96\,\hat\sigma_{n_1}/\sqrt{n_1}$, where $\hat\sigma_{n_1}$ denotes the empirical standard deviation of the $n_1$ tree-wise estimators. For $n_1<500$ the variance becomes excessively large, and the plot is therefore restricted to $n_1\ge 500$. Although unbiased, the untruncated MLMC estimator exhibits substantially larger variance across all values of~$n_1$ and is notably fragile, as illustrated by the discontinuity at $n_1=17{,}212$. Such jumps are caused by the occasional sampling of large log-branching factors, which can occur because the underlying geometric distribution is {\em not} truncated.\footnote{The discontinuity in Figure~\ref{fig:toy_example_estim}\,(a) occurs because a large log-branching factor equal to $\ell=10$ was sampled in the $17{,}212$-th tree, which incurs a sampling cost that scales with $2^{\ell}$ and also introduces a large importance weight $1/q_t(\ell)$ in~\eqref{eq:H_hat}.} The figure shows the first of $10$ independent simulation runs, with similarly pronounced discontinuities observed in $4$ out of the $10$ runs for the untruncated estimator. Overall, truncation yields markedly more stable MLMC estimators with reduced variance, and although it introduces bias, this bias remains negligible at least up to $n_1=10^5$. The two SAA estimators are known to converge slower to~$F(x)$ than the MLMC estimators and are thus not shown in Figure~\ref{fig:toy_example_estim}\,(a); see, {\em e.g.}, \cite[\S~3]{syed2023optimal}. 

Figure~\ref{fig:toy_example_estim}\,(b) illustrates the trade-off between mean squared error and sampling cost for the four estimators. As before, different instances of each estimator are parametrized by~$n_1$. Since the truncated MLMC estimator is biased, its mean squared error is eventually dominated, for sufficiently large~$n_1$, by the squared bias. This causes the error to saturate as~$n_1$ (and thus the total sampling cost) increases. By Lemma~\ref{lemma:bias_bound_rtmlmc}, the squared bias is bounded above by $(2^{-7}+2^{-6})^2\approx 5.4932\times10^{-4}$, implying that saturation occurs below a log-mean squared error of $\log_{10}(5.4932\times10^{-4})\approx -3.2602$. It actually lies outside of the plotting range. The untruncated MLMC estimator is unbiased for any~$n_1$, yet it attains a larger mean squared error for all shown computational budgets due to its substantially higher variance. Although for sufficiently large~$n_1$ this variance can be reduced below the squared bias of the truncated MLMC estimator, so that the untruncated estimator may eventually outperform its truncated counterpart with fixed truncation levels, the required sampling costs typically exceed practical computational limits. Moreover, with sufficiently large computational budgets, one could jointly increase~$n_1$ as well as the truncation points, thereby reducing both variance and bias of the truncated estimator and preserving its advantage. The mean squared errors of the two MLMC estimators decay at comparable empirical rates of $-0.8707$ and $-0.7871$ for the truncated and untruncated variants, respectively, in agreement with the theoretical $n_1^{-1}$ rate for smooth integrands (see Theorem~\ref{thm:rtmlmc_sample_complexity_mse} and \cite[Corollary~2.3]{syed2023optimal}). In contrast, the SAA estimators perform substantially worse, with decay rates of $-0.3352$ for SAA1 and $-0.4936$ for SAA2, consistent with the theoretical rates $n_1^{-1/3}$ and $n_1^{-1/2}$, respectively (see Theorem~\ref{thm:saa_sample_complexity_mse} and \cite[Appendix~G]{rainforth2018nesting}). Taken together, these findings indicate that the truncated MLMC estimator performs best when computational resources are limited.

\begin{figure}[htbp]
    \centering
    \begin{subfigure}{0.4973\textwidth} 
        \includegraphics[ width=\textwidth, keepaspectratio]{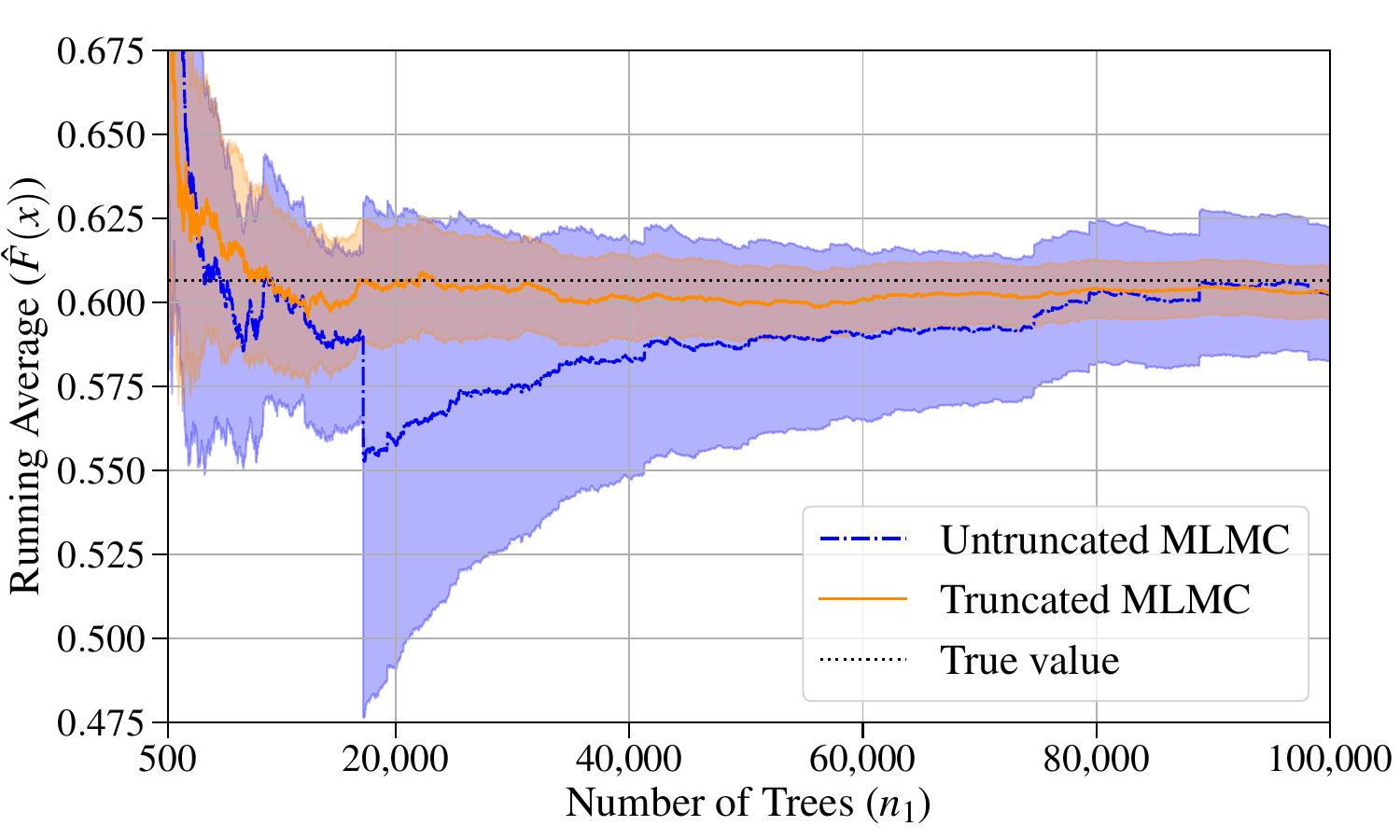}
        \caption{MLMC estimators with 95\% confidence intervals  vs~$n_1$.}
    \end{subfigure}
    \hfill
    \begin{subfigure}{0.4973\textwidth} 
        \includegraphics[ width=\textwidth, keepaspectratio]{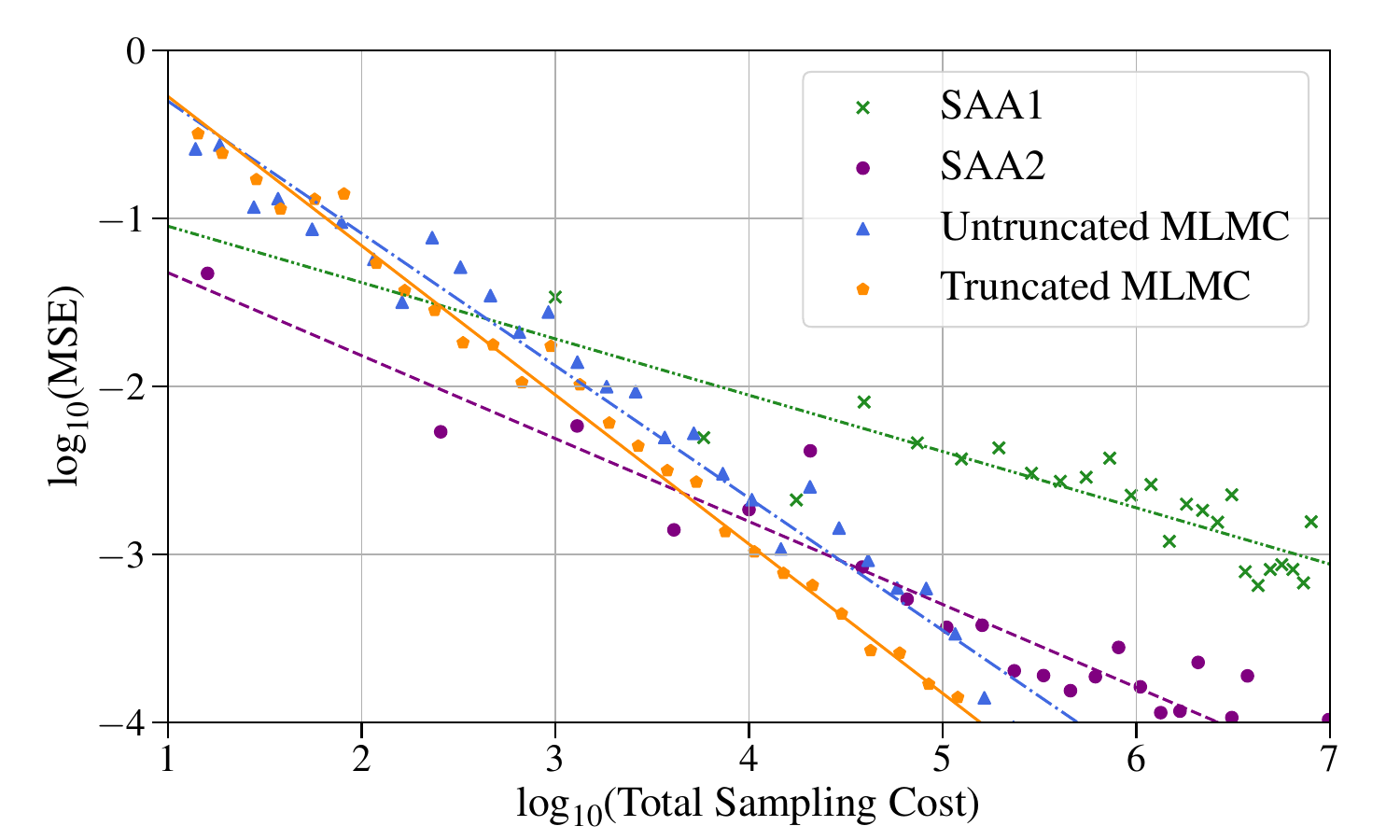}
        \caption{Mean squared error (MSE) vs total sampling cost.}
    \end{subfigure}
    \caption{Left panel: dependence of the untruncated and truncated MLMC estimators on~$n_1$. Right panel: trade-off between the mean squared error and the total sampling cost of all MLMC and SAA estimators. Each marker corresponds to an average of $10$ independent estimators of the same type for a fixed~$n_1$.}
    \label{fig:toy_example_estim}
\end{figure}

\paragraph{Pricing Bermudan Options} Consider a Bermudan basket option that grants its holder the right to sell an equally weighted basket of~$m=5$ assets at a fixed strike price~$K=100$ on one of~$T=4$ predetermined exercise dates $t\in[T]$. We use~$\xi_t\in\mathbb R^m$ to denote the random vector of asset prices at time~$t$ and assume that the risk-free interest rate amounts to~$\gamma=0.05$. Exercising the option at time~$t\in[T]$ yields an immediate payoff $g_t(\xi_t)=\max\left\{0,\, K-\frac{1}{m}\mathrm e_m^\top\xi_t\right\}$, where $\mathrm e_m\in\RR^m$ denotes the vector of ones. Under the risk-neutral measure, the asset prices follow independent geometric Brownian motions with drift~$\gamma$ and volatility~$\sigma=0.2$. Hence, the asset prices at the exercise dates satisfy $\xi_{t+1}=\operatorname{diag}(\varepsilon_t)\,\xi_t$, where the disturbances $\varepsilon_t\in\RR^m$, $t\in[T-1]$, are serially independent and identically distributed with $\ln(\varepsilon_t)\sim\cN((\gamma-\sigma^2 /2)\,\mathrm e_m,\sigma^2 I_m)$. We also assume that~$\xi_1=100\,\mathrm e_m$ is deterministic. The price of the option is given by the risk-neutral expected net present value of the payoffs under the optimal exercise strategy, which can be found by solving an optimal stopping problem as described in Section~\ref{sec:motivational_ex}. If~$x_t$ is the expected value of the option at time~$t+1$ conditional on time-$t$ information, then we have $x_{t-1}=\EE_{t-1}[\max\{g_t(\xi_t), e^{-\gamma} x_t\}]$ for every $t\in[T]$. Thus, computing today's price~$x_0$ amounts to solving the instance of~\eqref{problem:MCCE} with $f_t(\xi_t, x_t) = \max\{  g_t(\xi_t), e^{-\gamma} x_t \}$ for all~$t\in[T-1]$ and~$f_T(\xi_T,x)=g_T(\xi_T)$. For a more detailed discussion of Bermudan options see \cite{jain2012pricing}.

The option pricing problem under consideration does not admit a closed-form solution and must therefore be tackled using sampling-based methods. Approximate prices obtained with classical approaches are reported in \cite[Table~4]{jain2012pricing}; for instance, least squares Monte Carlo and the stochastic grid method yield estimates of $2.163(0.001)$ and $2.141(0.008)$, respectively, where numbers in parentheses denote standard errors and do not account for bias. In contrast, a policy improvement method~\cite{bender2006policy} produces a $95\%$ confidence interval of $[2.154,2.164]$ for the true option value. All of these classical methods are specifically designed for option pricing and do not readily extend to general MCCE problems. Furthermore, as noted in \cite{zhou2022unbiased}, their computational cost typically scales superlinearly with the dimension~$m$, which limits their applicability to high-dimensional basket options, where $m$ can be as large as $5{,}000$ \cite{becker2021solving}. In \cite{zhou2022unbiased} an untruncated MLMC estimator is shown to produce an unbiased price estimate of $2.161$ for the Bermudan option under consideration. Further numerical experiments suggest that the computational cost of this estimator scales sublinearly with~$m$. However, the analysis of this estimator heavily exploits the special structure of optimal stopping problems. As for any MLMC estimator, its efficiency depends critically on the rate parameters~$r_t$ governing the geometric distributions of the log-branching factors. If~$r_t\leq0.5$, then the estimator's expected cost diverges. In addition, there exists a positive tolerance $\delta\le 10^{-4}$ such that the estimator’s variance cannot be guaranteed to be finite for any $r_t\ge 0.5+\delta$~\cite[Theorem 2]{zhou2022unbiased}. Within the theoretically admissible narrow range of~$r_t$ close to~$0.5$, the untruncated MLMC estimator typically incurs high computational costs. For instance, setting $r_t=0.5001$ for $t\in[T-1]$ ($\delta=10^{-4}$) results in an expected cost of $2500.5^{T-1}$ \emph{per tree}.

We now compute the option price using three truncated MLMC estimators with $M_t=9,10,11$, as proposed in this paper, and two untruncated MLMC estimators with~$M_t=\infty$ proposed in~\cite{zhou2022unbiased}. The selected truncation points for the truncated estimators empirically provide the best balance between bias and computational cost. For the truncated MLMC estimators, we adopt a heuristic proposed in~\cite{zhou2022unbiased} to choose identical rate parameters $r_t=r$ for all geometric distributions governing the log-branching factors. As discussed after Lemma~\ref{lemma:bound_on_cost_rtmlmc}, the rate~$r$ should approximately minimize the expected sampling cost while ensuring that the estimator's variance does not exceed~$\epsilon^2/2$. This is achieved by minimizing the work-normalized second moment \cite{blanchet2015unbiased}, defined as $\mu_1^2(x)\prod_{t=1}^{T-1}\sum_{\ell=0}^{M_t} q_t(\ell)2^\ell$ (see also the proof of Lemma~\ref{lemma:optimal_b}). We approximate this quantity for each $r\in\{0.51,0.52,\ldots,0.70\}$ by estimating the expectation in the definition~\eqref{eq:mu} of~$\mu_1^2(x)$ via a sample average over $10^6$ independent copies of $\hat H_1(x)$, each computed as in Definition~\ref{def:MCCO_RTMLMC} but from a simplified and computationally cheap option pricing problem in which the average asset price $\frac{1}{m}e^\top\xi_t$ is replaced by a standard normal distribution. Next, we fit a piecewise linear convex function of~$r$ to the resulting approximate work-normalized second moments via a least squares procedure and select the rate parameter for all geometric distributions as a minimizer of this function. This yields~$r=0.59$ for the truncation points $M_t=9$ and~$M_t=11$ and~$r=0.58$ for~$M_t=10$. These rates are consistent with the theory of Section~\ref{sec:rtmlmc}, that is, they yield estimators with finite biases, variances and expected costs. As in~\cite{zhou2022unbiased}, we use the same heuristic approach to find a near-optimal rate of~$r=0.60$ for the first untruncated MLMC estimator. However, this rate is inconsistent with the theory of~\cite{zhou2022unbiased}, that is, for this choice of~$r$ the variance of the estimator is not guaranteed to be finite. The rate of the second untruncated MLMC estimator is set to~$r=0.5001$, which is the largest rate that guarantees the estimator to have a finite variance. Table~\ref{tab:option} reports the option prices predicted by all truncated and untruncated MLMC estimators when the number of scenario trees is set to~$n_1=5\times 10^6$. Even though they only use half as many scenario trees as the methods in~\cite{bender2006policy, zhou2022unbiased}, our truncated MLMC estimators provide comparable $95\%$~confidence intervals. Note that the untruncated MLMC estimator with~$r=0.6$ offers no confidence interval because its variance is not guaranteed to be finite, while the untruncated MLMC estimator with~$r=0.5001$ suffers from an excessive sampling cost per tree.

\begin{table}[htb]
\centering
\begin{tabular}{crrlc}
\hline
$M_t$    & \multicolumn{1}{c}{$r_t$} & \multicolumn{1}{c}{Cost per tree} & \multicolumn{1}{c}{Estimate (se)} & Confidence interval                 \\ \hline
9        & 0.5900                    & $22.6084 $               & $2.1684 (0.0076)$                 & $[2.1536, 2.1832]$ \\
10       & 0.5800                    & $29.5795$                & $2.1562 (0.0072)$                 & $[2.1421, 2.1703]$ \\
11       & 0.5900                    & $26.3283$                & $2.1641 (0.0080)$                 & $[2.1485, 2.1797]$ \\
$\infty$ & 0.5001                    & $1.5634 \times 10^{10}$  & $2.1547 (0.0073)$                         & $[2.1403, 2.1690]$ \\
$\infty$ & 0.6000                    & $27.0000$                & $2.1589 (-)$                      & $-$        
        \\ \hline
\end{tabular}
\caption{MLMC estimators for pricing Bermudan basket options: expected costs per tree, estimated option prices with standard errors (se) and $95\%$ confidence intervals for different truncation points and rates.}
\label{tab:option}
\end{table}

\paragraph{Contextual Bandits} 

In the last experiment we solve an instance of the off-policy learning problem~\eqref{eq:dro-offline-bandits-dual-softmax} borrowed from~\citep{shen2024wasserstein}, which determines one of two treatments to be administered to hospital patients suffering from acute ischaemic stroke. Each patient is characterized by a context vector~$c \in \mathcal{C} \subseteq \mathbb{Z}_+^6$. The first context feature $c_1 \in \{0,1,2\}$ represents the patient's level of consciousness (0: fully conscious, 1: drowsy, 2: unconscious), while $c_2 \in \{0, \dots, 4\}$, $c_3, c_4 \in \{0, 1\}$,  $c_5 \in \{0, \dots, 5\}$ and $c_6 \in \{0, \dots, 3\}$ capture other categorical features such as the patient's age etc. The overall context space~$\mathcal{C}$ is thus finite, and its cardinality equals~1{,}440. As in~\citep{shen2024wasserstein}, we assume that there are two actions $a\in\mathcal A=\{1,2\}$ (1: prescribing both aspirin and heparin, 2: not administering any treatment), and we parametrize the context-dependent policy~$\pi_\theta$ with $\theta \in[0,1]^2$, where the agent selects action~1 with probability~$\theta_1$ if~$c_1 \neq 0$ and with probability~$\theta_2$ if~$c_1 = 0$. 

We will solve the resulting instance of~\eqref{problem:MCCO} with a variant of the Adam optimizer \cite{kingma2014adam} equipped with the MLMC gradient estimator of Definition~\ref{def:rumlmc-grad}. This allows us to assess our MLMC gradient estimator in an interesting scenario where we have no theoretical results guaranteeing its superiority over basic SAA estimators. Indeed, as  explained in Section~\ref{sec:application_dro}, the integrands~$f_1$, $f_2$ and~$f_3$ involve logarithms as well as quadratic and exponential terms. Thus, they fail to satisfy the Lipschitz continuity and smoothness conditions detailed in Assumptions~\ref{assump:lipschitz} and~\ref{assump:smooth}, which are needed for all theoretical results in Section~\ref{sec:mcco_optimization}.

We prescribe the joint distribution of~$c$, $u$ and~$y$ under~$\mathbb P$ synthetically, which allows us to solve problem~\eqref{eq:dro-offline-bandits-dual-softmax} exactly. The resulting ground-truth solution serves as a benchmark against which any data-driven policies can be compared. First, we assume that the context~$c$ as well as the auxiliary random vector~$u$ follow independent uniform distributions on~$\cC$. Next, the mean of the cost vector~$y$ conditional on~$c$ is set to
\[
    \mathbb E[y|c] = \left\{ \begin{array}{ll}
    (3.0 + 5.0 c_5+p(c), 5.5 + 1.0 c_5+p(c)) & \text{if }c_1 = 0 ,\\
     (1.7 + 3.5 c_5+p(c) , 3.0 + 1.0 c_5+p(c)) & \text{if }c_1 \neq 0,
    \end{array} \right.
\]
where $p(c)=2.4+1.92(c_5/5-2.5)^2$ if $c_2=4$, $c_3=1$, $c_4=1$ and $c_6=3$; $p(c)=0$ otherwise. Conditional on~$c$, the actual cost vector~$y$ follows a bivariate log-normal distribution, that is, $\log(y)$ follows the normal distribution with mean $\log(\mathbb E[y|c]) - \frac{1}{2}\text{diag}(\Sigma)$ and covariance matrix~$\Sigma$, where the diagonal and off-diagonal elements of~$\Sigma$ are set to~$5.0$ and~$2.5$, respectively. The heterogeneity of $\mathbb E[y|c]$ ensures that the exact optimal solution of problem~\eqref{eq:dro-offline-bandits-dual-softmax} displays a sufficient level of complexity, whereas the positive correlation between the costs of the two actions captures the effect of an unobserved confounder that simultaneously affects both treatment outcomes. Finally, we set the smoothness temperature in~\eqref{eq:dro-offline-bandits-dual-softmax} to~$\mu=2$, and we set the ambiguity radii to~$r_c=0.4$ and~$r_y = 0.15$. This choice roughly accounts for a covariate shift whereby the subpopulation of patients with~$c_1=0$ is over-sampled by~$50\%$. 

By leveraging known properties of log-normal distributions, the innermost integral in~\eqref{eq:dro-offline-bandits-dual-softmax} can be evaluated analytically. As~$\cC$ is finite, the objective of~\eqref{eq:dro-offline-bandits-dual-softmax} thus reduces to a standard log-sum-exp function. Hence, the exact global minimizer $(\lambda^*, \theta_1^*, \theta_2^*) \approx (11.829, 0.589, 0.713)$ of problem~\eqref{eq:dro-offline-bandits-dual-softmax}
can be computed efficiently to any desired accuracy with any off-the-shelf convex optimization solver.

We now address problem~\eqref{eq:dro-offline-bandits-dual-softmax} with oracle-based methods that have only sample access to~$\mathbb P$. Specifically, we solve~\eqref{eq:dro-offline-bandits-dual-softmax} with a variant of the Adam optimizer using learning rates of 0.75 for~$\lambda$ and 0.025 for~$\theta_1$ and~$\theta_2$, alongside default decay rate parameters. To prevent Adam's momentum from accumulating against a hard boundary and stalling convergence, we enforce the $\lambda \geq 0$ constraint via a softplus reparameterization rather than a standard projection. Concretely, we set $\lambda=\log(1+\exp(\lambda'))$ for some $\lambda'\in\RR$ and run Adam entirely on~$\lambda'$. However, we simply project~$\theta_1$ and~$\theta_2$ to~$[0,1]$ in each iteration of the algorithm. We also add the $L_2$~regularization term~$0.005\cdot\|\theta\|_2^2$ to the objective of~\eqref{eq:dro-offline-bandits-dual-softmax} to enhance stability. As stochastic gradient estimators we use the SAA-type estimator proposed in~\cite{shen2024wasserstein} (adapted to account for stochastic costs) as well as the MLMC gradient estimator of Definition~\ref{def:rumlmc-grad}. As the Lipschitz continuity and smoothness conditions required to control the bias and variance of the MLMC estimator fail to hold in this example, we clip the norms of the truncated MLMC gradient estimators with respect to~$\lambda$, $\theta_1$ and~$\theta_2$ at 100, 50 and 50, respectively. Note that problem~\eqref{eq:dro-offline-bandits-dual-softmax} as described in~\citep{shen2024wasserstein} constitutes an instance of~\eqref{problem:MCCO} with~$T=3$. However, the numerical experiments in~\citep{shen2024wasserstein} focus on a simplified model with deterministic costs and~$T=2$---presumably due to the hardness of the original three-stage problem. In contrast, our experiments faithfully capture the three-stage nature of the more challenging original problem with stochastic costs.

Figure~\ref{fig:contextual} illustrates the convergence of~$\lambda$, $\theta_1$ and~$\theta_2$ as a function of the cumulative number of scenarios over 2,000 Adam iterations and for different choices of the estimators' hyperparameters. Solid lines and shaded regions represent means as well as corresponding 95\% confidence intervals obtained from 20 independent simulation runs, whereas dotted lines represent ground-truth minimizers. We observe that Adam converges significantly faster when equipped with MLMC estimators instead of SAA estimators. Recall from Lemma~\ref{lemma:bias_bound_rtmlmc_grad} that the truncated MLMC estimator is biased, and note that gradient clipping introduces an additional bias. Nevertheless, MLMC estimators lead to less biased optimizers than SAA estimators. Even though Assumptions~\ref{assump:lipschitz} and~\ref{assump:smooth} fail to hold, MLMC estimators thus enable faster and tighter convergence to the true optimal solution than SAA estimators equipped with a comparable computational budget.

\begin{figure}[ht]
    \centering
    \includegraphics[width=1\linewidth]{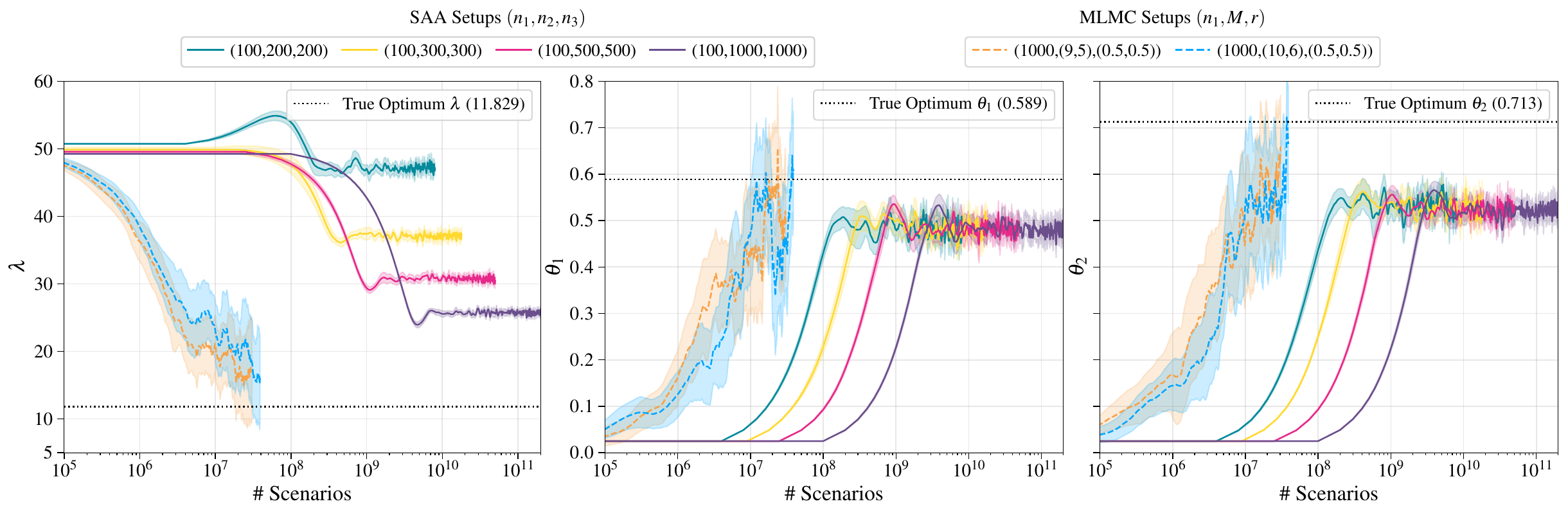}
    \caption{Convergence of the decision variables $\lambda$ (left), $\theta_1$ (middle), and $\theta_2$ (right) generated by Adam. Results show mean trajectories with 95\% confidence intervals based on 20 independent simulation runs.}
    \label{fig:contextual}
\end{figure}


\paragraph{Acknowledgments}
This work was supported as a part of NCCR Automation, a National Centre of Competence in Research, funded by the Swiss National Science Foundation (grant number 51NF40\_225155). We thank Guanyang Wang for providing the code for the experiments in~\cite{syed2023optimal}, as well as Yi Shen and Pan Xu for the code corresponding to~\cite{shen2024wasserstein}. We also thank Guanyang Wang, Lie He and Niao He for helpful discussions.

\medskip
\bibliographystyle{plainnat}
\bibliography{ref.bib}

\appendix
\section*{Appendix}
\section{Auxiliary Results} \label{appendix:auxiliary}
\renewcommand{\thelemma}{A.\arabic{lemma}}  
The following lemma establishes a uniform deviation bound based on covering numbers. It is a standard result in stochastic programming, and we include a concise proof to keep the presentation self-contained.

\begin{lemma}[{\citep[Proof of Theorem~1]{shapiro2006complexity}}]\label{lemma:v-net}
Consider the problem of minimizing an $L$-Lipschitz continuous function $F:\RR^d\to\RR$ over a compact set $\cX\subseteq\RR^d$ with diameter $D_{\cX}$. Let $\hat F:\RR^d\to\RR$ be an $L$-Lipschitz continuous random function that approximates~$F$. Then, for every $v>0$, there is a $v$-net $\cX_v\subseteq \cX$ such that $|\cX_v|\leq\lceil 2D_{\cX}/v+1\rceil^{d}$ and $\min_{x^\prime\in\cX_v}\|x-x^\prime\|_2\leq v$ for every $x\in\cX$. If $\epsilon>0$ and $v=\frac{\epsilon}{4L}$,  then we have
    \begin{align*} 
        \PP\Big(\sup_{x\in\mathcal{X}} \lvert \hat F(x) - F(x) \rvert > \epsilon \Big)\leq\sum_{x^\prime\in\cX_v} \PP\Big( \lvert  \hat F(x^\prime) - F(x^\prime) \rvert  > \frac{\epsilon}{2}\Big).
    \end{align*}
\end{lemma}

\begin{proof}
The set~$\cX$ is contained in a ball of radius~$D_\cX$ and can thus be covered by $\lceil 2D_{\cX}/v+1\rceil^{d}$ smaller balls of diameter~$v$~\cite[Lemma~5.7]{wainwright2019high}. By construction, the centroids of these smaller balls, projected onto~$\cX$, constitute a $v$-net $\cX_v$. That is, for every $x\in\cX$, there is an $x'\in\cX_v$ in its closed $v$-neighborhood. Assume now that $v=\frac{\epsilon}{4L}$ for some $\epsilon>0$. Then, for every $x\in\cX$ and $x^\prime\in\argmin_{x^\prime\in\cX_v}\|x-x^\prime\|_2$, we have
\begin{align*}
    \lvert \hat F(x) - F(x) \rvert \leq \lvert \hat F(x) - \hat F(x^\prime) \rvert + \lvert  \hat F(x^\prime) - F(x^\prime) \rvert + \lvert F(x^\prime) - F(x) \rvert \leq \lvert  \hat F(x^\prime) - F(x^\prime) \rvert+ \frac{\epsilon}{2} ,
\end{align*}
where the second inequality follows from the $L$-Lipschitz continuity of~$F$ and~$\hat F$ and from the observation that $\|x'-x\|_2\leq v$. As the above bound holds for every $x\in\cX$, we readily obtain
\begin{align*}
    \sup_{x\in\cX} \lvert \hat F(x) - F(x) \rvert  \leq \frac{\epsilon}{2} + \max_{x^\prime\in\cX_v} \lvert  \hat F(x^\prime) - F(x^\prime) \rvert.
\end{align*}
This further implies via the union bound that
\begin{align*}
    \PP\left(\sup_{x\in\cX} \lvert \hat F(x) - F(x) \rvert > \epsilon \right) \leq \PP\left( \max_{x^\prime\in\cX_v} \lvert  \hat F(x^\prime) - F(x^\prime) \rvert > \frac{\epsilon}{2} \right) \leq \sum_{x^\prime\in\cX_v} \PP\left( \lvert  \hat F(x^\prime) - F(x^\prime) \rvert  > \frac{\epsilon}{2}\right).
\end{align*}
Hence, the claim follows.
\end{proof}

The following matrix norm bounds will be used repeatedly in proving the main results of this paper.

\begin{lemma}[{\cite[\S~5.4.P3]{Horn_Johnson1985matrixanalysis}}]\label{lemma:matrix-inequality-pq}
    If $A\in\RR^{m\times n}$ and $1\leq p <q\leq\infty$, then
    $\|A\|_q \leq \|A\|_p \leq (mn)^{\frac{1}{p}-\frac{1}{q}}\|A\|_q$.
\end{lemma}

\begin{lemma}[{\cite[\S~5.6.0.2]{Horn_Johnson1985matrixanalysis}}]\label{lemma:cauchy-schwarz-matrix}
    If $A\in\RR^{m\times n}$ and $B\in\RR^{n\times k}$, then $\|A B\|_2 \leq \|A\|_2 \|B\|_2$.
\end{lemma}

Finally, we state a well-known corollary of the Marcinkiewicz-Zygmund inequality as an auxiliary lemma, which is useful for deriving variance bounds of the estimators proposed in the main text.

\begin{lemma}
\label{lemma:moment-inequality-multidimensional} Assume that $Y$ and $Z$ are random matrices of arbitrary dimensions with finite $p$-th moments for some $p\geq 2$, and let $Z_i$, $i\in[n]$, be independent samples from the conditional distribution $\mathbb P_{Z|Y}$. Then there exists a constant $B_p$ depending only on $p$ such that
    \begin{equation*}
        \EE_{Z|Y}\bigg[\Big\|\frac{1}{n}\sum_{i=1}^n Z_i-\EE_{Z|Y}[Z]\Big\|^p_p\bigg] \leq \frac{B_p}{n^{p/2}}  \EE_{Z|Y}\left[ \|Z\|^{p}_p \right].
    \end{equation*}
\end{lemma}
The proof parallels that of~\cite[Corollary 2]{zhou2022unbiased}, generalized to matrices. Details are omitted for brevity. 

\section{Proofs of Section~\ref{sec:motivational_ex}}
\label{sec:proofs-section-2}
\begin{proof}[Proof of Lemma~\ref{lem:Q-lqr}]
We prove the claim by backward induction on~$t$. At $t=T$ we have $Q_T(s,a)=s^\top P_T s$ by assumption, which is of the form~\eqref{eq:Q-Bellman} if the coefficients satisfy~\eqref{eq:Q-initialization}. This establishes the base step. As for the induction step, assume that~\eqref{eq:Q-Bellman} holds for $Q_{t+1}$. Thus, we find 
\[
	\min_{a'\in\mathbb R^n} Q_{t+1}(s',a')  = {s'}^\top P_{t+1} s' + 2\,g_{t+1}^\top s' + d_{t+1},
\]
with the Schur-complement expressions for $P_{t+1}$, $g_{t+1}$ and $d_{t+1}$ given in~\eqref{eq:Q-Schur}. Substituting the above formula with $s'=A s + B a + \xi_{t+1}$ into the Bellman recursion~\eqref{eq:Qdef} and expanding the quadratic and linear terms shows that $Q_t(s,a)$ is again a quadratic function in $(s,a)$. Collecting coefficients of $s^\top s$, $s^\top a$, $a^\top a$, $s$, $a$ and the constant term then yields exactly the expressions~\eqref{eq:Q-recursion}. Thus the quadratic structure of the Q-function is preserved backward in time, completing the induction. 
\end{proof}

The proof of Lemma~\ref{lem:Q-lqr} readily reveals that the optimal control at time~$t$ is given by
\begin{equation*}
a_t^\star = -\big(Q_t^{aa}\big)^{-1} \big((Q_t^{sa})^\top s + b_t^a \big)
\end{equation*}
and thus constitutes an affine function of the state~$s$ with $\mathcal F_t$-measurable random coefficients.

\section{Proofs of Section~\ref{sec:mcco_estimation}}

To streamline the proofs in Section~\ref{sec:mcco_estimation}, we first introduce some key notation and conventions that will be useful for our derivations. Throughout this discussion we use $\mathcal L_t$ as shorthand for the space $\mathcal L^1(\Omega,\mathcal F,\PP;\RR^{d_t})$ of all $\PP$-integrable $d_t$-dimensional random vectors on $(\Omega,\mathcal F,\PP)$. 

\begin{defn}
\label{def:risk-mapping}
For any $t\in[T]$ and $s\in[t+1]$, we define the conditional risk mapping $F_s^t : \mathcal L_t\to \mathcal L_{s-1}$~via
\begin{align*}
    F_s^t(z_t)\coloneqq \begin{cases}
        \EE_{s-1} \big[f_{s} \big(\xi_{s},\ldots ,\EE_{t-1} \big[ f_t(\xi_t, z_t ) \big] \ldots \big) \big] & \text{if } s<t+1, \\ z_t & \text{if } s=t+1.
    \end{cases}
\end{align*}
\end{defn}

Note that $F_s^t$ maps $d_t$-dimensional random vectors to $d_{s-1}$-dimensional random vectors. By construction, we have $F(x)=F_1^T(x)$ for every (deterministic) vector $x\in\RR^d$. If Assumptions~\ref{assump:general} and~\ref{assump:lipschitz} hold, then $F_s^t(z_t)$ is well-defined and constitutes indeed an element of $\mathcal L_{s-1}$ for every $z_t\in \mathcal L_{t}$.

\begin{lemma}
\label{lem:well-defined}
If Assumptions~\ref{assump:general} and~\ref{assump:lipschitz} hold, then $F_s^t(z_t)\in \mathcal L_{s-1}$ for all $z_t\in \mathcal L_{t}$, $t\in[T]$ and $s\in[t+1]$.
\end{lemma}

\begin{proof}
We fix $t\in[T]$ and $z_t\in\mathcal L_t$ and prove the claim by induction on~$s$. If $s=t+1$, then we have $F_s^t(z_t)=z_t\in\mathcal L_t$, and thus the base case is trivially true. As for the induction step corresponding to any~$s\in[t]$, assume that $F_{s+1}^t(z_t)\in\mathcal L_s$. Then, $f_{s}(\xi_{s}, F_{s+1}^t(z_t))$ is $\mathcal F$-measurable thanks to \cite[Lemma~4.51]{aliprantisinfinite} and because measurability is preserved under Borel-measurable transformations. Also, we have
\begin{align*}
    \EE\left[ \big\| F_s^t(z_t)\big\|_2 \right] & = \EE\left[ \big\| \EE_{s-1} [ f_s(\xi_s,F_{s+1}^t(z_t))] \big\|_2 \right]\\
    & \leq \EE\left[ \big\| f_s(\xi_s,F_{s+1}^t(z_t)) \big\|_2 \right] \\
    & \leq \EE\left[ \big\| f_s(\xi_s,F_{s+1}^t(z_t))-f_s(\xi_s,0) \big\|_2 + \big\| f_s(\xi_s,0) \big\|_2 \right] \\
    & \leq L_s\, \EE\left[ \big\|F_{s+1}^t(z_t)\big\|_2\right] + \EE\left[\big\| f_s(\xi_s,0) \big\|_2 \right] <\infty,
\end{align*}
where the first inequality follows from Jensen's inequality and the law of total expectation, while the second and the third inequalities follow from the triangle inequality and Assumption~\ref{assump:lipschitz}, respectively. The resulting bound is finite by the induction hypothesis, which assumes that $F_{s+1}^t(z_t)\in\mathcal L_s$, and by Assumption~\ref{assump:general}. We may thus conclude that $F_{s}^t(z_t)\in\mathcal L_{s-1}$. This observation completes the induction step.
\end{proof}

Lemma~\ref{lem:well-defined} ensures that the function~$F(x)$ is well-defined. If Assumptions~\ref{assump:general} and~\ref{assump:lipschitz} hold, then all conditional risk mappings of Definition~\ref{def:risk-mapping} are Lipschitz continuous. If additionally Assumption~\ref{assump:smooth} holds, then the conditional risk mappings are also Lipschitz smooth.

\begin{lemma}\label{lem:lipschitz}
Fix any $t\in[T]$ and $s\in[t+1]$. If Assumptions~\ref{assump:general} and~\ref{assump:lipschitz} hold, then $F_s^t$ is an $L_{[s:t]}$-Lipschitz continuous function on the space $\mathcal L_t$ with $L_{[s:t]}\coloneqq\prod_{i=s}^t L_i$ if $s<t+1$; $\coloneqq1$ if $s=t+1$. This means that
\[
     \big\|F_s^t(z_t) - F_s^t(z'_t) \big\|_2 \leq L_{[s:t]} \, \EE_{s-1} \left[ \|z_t-z'_t\|_2\right] \quad \forall z_t,z'_t\in\mathcal L_t.
\]
If additionally Assumption~\ref{assump:smooth} holds, then $F_s^t$ is an $S_{[s:t]}$-smooth function on the space $\mathcal L_t$ with $S_{[s:t]}\coloneqq\sum_{i=s}^{t} S_i L_{[s:i-1]} L_{[i+1:t]}^2$ if $s<t+1$; $\coloneqq0$ if $s=t+1$. This means that
\[
     \big\|\nabla F_s^t(z_t) - \nabla F_s^t(z'_t) \big\|_2 \leq S_{[s:t]} \, \EE_{s-1} \left[ \|z_t-z'_t\|_2\right] \quad \forall z_t,z'_t\in\mathcal L_t.
\]
\end{lemma}

\begin{proof}
We fix $t\in[T]$ and use backward induction on~$s$ to prove that $F_s^t$ is $L_{[s:t]}$-Lipschitz continuous. In the base case corresponding to~$s=t+1$, $F_s^t$ collapses to the identity mapping on~$\mathcal L_t$, and thus the claim is trivially true. As for the induction step corresponding to any $s\in[t]$, assume that $F_{s+1}^t$ is $L_{[s+1:t]}$-Lipschitz continuous. Since $F_s^t(z_t)=\EE_{s-1}[f_s(\xi_s,F_{s+1}^t(z_t))]$, we thus have
\begin{align*}
   \big\| F_s^t(z_t)- F_s^t(z_t^\prime) \big\|_2 \leq & \EE_{s-1} \big[ \big\| f_s\left(\xi_s,F_{s+1}^t(z_t)\right)-f_s\left(\xi_s,F_{s+1}^t(z_t^\prime)\right) \big\|_2 \big]\\
   \leq  &L_s \; \EE_{s-1} \big[ \big\|F_{s+1}^t(z_t)-F_{s+1}^t(z_t^\prime) \big\|_2 \big] \\
   \leq& L_s L_{[s+1:t]} \; \EE_{s-1} \left[\| z_t-z_t^\prime\|_2\right] = L_{[s:t]} \;\EE_{s-1} \left[ \| z_t-z_t^\prime\|_2 \right]
\end{align*}
for any $z_t,z'_t\in\mathcal L_t$. Here, the first two inequalities follow from Jensen's inequality and from Assumption~\ref{assump:lipschitz}, respectively, while the third inequality exploits the induction hypothesis and the law of total expectation. This completes the induction step, and thus the claim follows.

Suppose now that Assumption~\ref{assump:smooth} holds, too. One readily verifies that the smoothness constants obey the recursion $S_{[s:t]} =  L_sS_{[s+1:t]} + S_s L_{[s+1:t]}^2$ for all~$s\in[t]$ and~$t\in[T]$. In the remainder we fix any~$t\in[T]$ and use backward induction on~$s$ to prove that $F_s^t$ is $S_{[s:t]}$-smooth. In the base case corresponding to~$s=t+1$, the claim is trivially true because~$F_s^t$ is the identity mapping on~$\cL_t$. As for the induction step corresponding to any $s\in[t]$, assume that $F_{s+1}^t$ is $S_{[s+1:t]}$-smooth. By the definition of $F_{s}^t(z_t)$, we have
\begin{align*}
    \nabla F_{s}^t(z_t)= \nabla \EE_{s-1} \big[ f_s\left(\xi_s,F_{s+1}^t(z_t)\right)\big]= \EE_{s-1} \big[ \nabla_{z_t} f_s\left(\xi_s,F_{s+1}^t(z_t)\right)\big].
\end{align*}
Note that the gradient may be interchanged with the conditional expectation by the dominated convergence theorem, which applies thanks to Assumption~\ref{assump:lipschitz}. Jensen's inequality and the chain rule then imply that
\begin{align*}
    & \big\| \nabla F_{s}^t(z_t) - \nabla F_{s}^t(z_t^\prime)\big\|_2  \leq \EE_{s-1} \big[\big\| \nabla_{z_t} f_s\left(\xi_s,F_{s+1}^t(z_t)\right)-\nabla_{z_t} f_s\left(\xi_s,F_{s+1}^t(z_t^\prime)\right) \big\|_2\big]\\
    & \hspace{2.5cm} = \EE_{s-1} \big[\big\| \nabla F_{s+1}^t(z_t)  \nabla_{x_s} f_s\left(\xi_s,F_{s+1}^t(z_t)\right) -\nabla F_{s+1}^t(z_t^\prime) \nabla_{x_s} f_s\left(\xi_s,F_{s+1}^t(z_t^\prime)\right) \big\|_2\big].
\end{align*}
Next, we may use the triangle inequality to conclude that
\begin{align*}
    \big\| \nabla F_{s}^t(z_t) - \nabla F_{s}^t(z_t^\prime)\big\|_2 & \leq \EE_{s-1} \big[ \big\| \left(\nabla  F_{s+1}^t(z_t) -\nabla  F_{s+1}^t(z_t^\prime)\right) \nabla_{x_s}  f_s\left(\xi_s,F_{s+1}^t(z_t)\right) \big\|_2 \big]\\
    & \quad +\EE_{s-1} \big[ \big\| \nabla  F_{s+1}^t(z_t^\prime) \left(\nabla_{x_s}  f_s\left(\xi_s,F_{s+1}^t(z_t)\right) -  \nabla_{x_s}  f_s\left(\xi_s,F_{s+1}^t(z_t^\prime)\right) \right)\big\|_2 \big]\\
    & \leq L_s\, \EE_{s-1} \big[ \big\|\nabla  F_{s+1}^t(z_t) -\nabla  F_{s+1}^t(z_t^\prime) \big\|_2 \big] \\
    &\quad + S_s \, \EE_{s-1} \big[ \big\|\nabla F_{s+1}^t(z_t^\prime) \big\|_2\,  \big\|F_{s+1}^t(z_t) - F_{s+1}^t(z_t^\prime) \big\|_2 \big] \\
    & \leq \big( L_s S_{[s+1:t]} + S_s L_{[s+1:t]}^2 \big) \big\| z_t - z_t^\prime \big\|_2 = S_{[s:t]}\, \big\| z_t - z_t^\prime \big\|_2,
\end{align*}
where the second inequality follows from Assumptions~\ref{assump:lipschitz} and~\ref{assump:smooth} and from the Cauchy-Schwarz inequality. The third inequality exploits the induction hypothesis, whereby $F_{s+1}^t$ is $S_{[s+1:t]}$-smooth, as well as the first part of the proof, whereby $F_{s+1}^t$ is $L_{[s+1:t]}$-Lipschitz continuous. Finally, the equality follows from the recursive formula for~$S_{[s:t]}$ established earlier. Hence, $F_{s}^t(z_t)$ is indeed $S_{[s:t]}$-smooth in~$z_t$.
\end{proof}

In analogy to Definition~\ref{def:risk-mapping}, we can now introduce {\em empirical} conditional risk mappings on $\mathcal L_t$. 
\begin{defn}
\label{def:risk-mapping-hat}
For any $t\in[T]$ and $s\in[t+1]$, we define the conditional risk mapping $\hat F_s^t : \mathcal L_t\to \mathcal L_{s-1}$~via
\begin{align*}
    \hat F_s^t(z_t)\coloneqq \begin{cases}
        \hat \EE_{s-1} \big[f_{s} \big(\xi_{s},\ldots ,\hat \EE_{t-1} \big[ f_t(\xi_t, z_t ) \big] \ldots \big) \big] & \text{if } s<t+1, \\ z_t & \text{if } s=t+1.
    \end{cases}
\end{align*}
Here, $\hat \EE_{t-1}[\cdot]$, $t\in[T]$, denote the empirical conditional expectation operators defined in~\eqref{eq:expected_psi}.
\end{defn}

Note that $\hat F(x)=\hat F_1^T(x)$ for every (deterministic) vector $x\in\RR^d$. If Assumptions~\ref{assump:general} and~\ref{assump:lipschitz} hold, one can proceed as in Lemmas~\ref{lem:well-defined} and~\ref{lem:lipschitz} to show that the conditional risk mapping~$\hat F_s^t$ is well-defined and $L_{[s:t]}$-Lipschitz continuous for any~$t\in [T]$ and~$s\in[t+1]$. In fact, the proofs simplify because  $\PP$-integrability and Lipschitz continuity are trivially preserved under finite sums. Details are omitted for brevity.
We are now ready to prove the formal results of Section~\ref{sec:saa}.

\begin{proof}[Proof of Lemma~\ref{lemma:bound_on_bias_saa}] 
Suppose first that only Assumptions~\ref{assump:general}, \ref{assump:lipschitz} and \ref{assump:finite_variance} hold. By the construction of the conditional risk mappings in Definitions~\ref{def:risk-mapping} and~\ref{def:risk-mapping-hat}, we have 
\[
    F(x) = F_1^T(x) = F_1^{T}(\hat F_{T+1}^T(x)) \quad \text{and} \quad \EE[\hat F(x)] = \EE[\hat F_1^T(x)]= F_1^1(\hat F_2^T(x))
\]
for all $x\in\RR^d$. We can thus represent the bias of the SAA estimator $\hat F(x)$ as a telescoping sum and then use the triangle inequality to conclude that
\begin{align}
    \label{eq:telescoping}
    \big|\EE [\hat F(x)] - F(x) \big| \leq \sum_{t=1}^{T-1} \left| F_1^t \big( \hat F_{t+1}^T(x) \big) - F_1^{t+1} \big( \hat F_{t+2}^T(x) \big) \right|. 
\end{align}
The $t$-th term in the resulting upper bound satisfies
\begin{align*}
    \left|F_1^t \big( \hat F_{t+1}^T(x) \big) - F_1^{t+1} \big( \hat F_{t+2}^T(x) \big) \right|= & \left| F_1^t \big( \hat F_{t+1}^T(x) \big) - F_1^{t} \big( \EE_t [\hat F_{t+1}^T(x) ]\big) \right| \\
    \leq & L_{[t]} \, \EE \big[ \big\| \hat F_{t+1}^T(x) - \EE_t[\hat F_{t+1}^T(x)] \big\|_2 \big]\\
    \leq & L_{[t]} \, \EE \big[ \big\| \hat F_{t+1}^T(x) - \EE_t[\hat F_{t+1}^T(x)] \big\|_2^2 \big]^{\frac{1}{2}} ,
\end{align*} 
where the two inequalities follow from the Lipschitz continuity of $F_1^t$ established in Lemma~\ref{lem:lipschitz} and from Jensen's inequality, respectively. By the definition of $\hat F_t^T$, we then have
\begin{align*}
& L_{[t]} \, \EE \big[ \big\| \hat F_{t+1}^T(x) - \EE_t[\hat F_{t+1}^T(x)] \big\|_2^2 \big]^{\frac{1}{2}}\\
    = & L_{[t]} \, \EE \Big[ \Big\| \hat \EE_t\big[ f_{t+1}\big( \xi_{t+1}, \hat F_{t+2}^T(x)\big) \big] - \EE_t\big[ f_{t+1}\big( \xi_{t+1}, \hat F_{t+2}^T(x)\big) \big] \Big\|_2^2 \Big]^{\frac{1}{2}} \\
    = & \frac{L_{[t]}}{\sqrt{n_{t+1}}} \, \EE \Big[ \Big\| f_{t+1}\big( \xi_{t+1}, \hat F_{t+2}^T(x)\big) - \EE_t\big[ f_{t+1}\big( \xi_{t+1}, \hat F_{t+2}^T(x)\big) \big] \Big\|_2^2 \Big]^{\frac{1}{2}} 
    \leq \frac{L_{[t]}\sigma_{t+1}}{\sqrt{n_{t+1}}},
\end{align*}   
where the second equality holds because, conditional on $\xi_{[t]}$, 
the variance of an average of $n_{t+1}$ i.i.d.\ random variables coincides with the variance of a single random variable divided by~$n_{t+1}$. The inequality in the last line follows from Assumption~\ref{assump:finite_variance}. Substituting the resulting estimate into~\eqref{eq:telescoping} yields
\begin{align*}
    \big|\EE [\hat F(x)] - F(x) | \leq \sum_{t=1}^{T-1} \frac{L_{[t]}\sigma_{t+1}}{\sqrt{n_{t+1}}}. 
\end{align*}

Suppose now that Assumption~\ref{assump:smooth} holds, too. Hence, the $t$-th term in~\eqref{eq:telescoping} admits the tighter estimate
\begin{equation}
\label{eq:smooth_case_lipschitz_step}
    \begin{aligned}
    & \left|F_1^t \big( \hat F_{t+1}^T(x) \big) - F_1^{t+1} \big( \hat F_{t+2}^T(x) \big) \right| \\
    = & \left| F_1^{t-1} \big( \EE_{t-1} \big[ f_t \big(\xi_t,\hat F_{t+1}^T(x) \big) \big] \big) - F_1^{t-1} \big( \EE_{t-1} \big[ f_t \big(\xi_t,\EE_t [\hat F_{t+1}^T(x)] \big) \big] \big) \right| \\
    \leq & L_{[t-1]} \; \EE \left[ \Big\|\EE_{t-1} \Big[ f_t \big(\xi_t,\hat F_{t+1}^T(x) \big) - f_t \big(\xi_t,\EE_t [\hat F_{t+1}^T(x)] \big) \Big] \Big\|_2 \right]
    \end{aligned}
\end{equation}
where the inequality follows from the Lipschitz continuity of $F_1^{t-1}$ established in Lemma~\ref{lem:lipschitz}.
\begin{align*}
    \EE_{t-1} \Big[\nabla_{x_t} f_t\big(\xi_t,\EE_t [\hat F_{t+1}^T(x)]\big) ^\top  \big(\hat F_{t+1}^T(x) - \EE_t [\hat F_{t+1}^T(x)] \big) \Big] =0
\end{align*}
as $\nabla_{x_t} f_t(\xi_t,\EE_t [\hat F_{t+1}^T(x)])$ is $\mathcal F_t$-measurable. The conditional expectation in the last line of~\eqref{eq:smooth_case_lipschitz_step} thus satisfies
\begin{align*}
    & \Big\|  \EE_{t-1} \Big[f_t \big(\xi_t,\hat F_{t+1}^T(x) \big) - f_t \big(\xi_t,\EE_t [\hat F_{t+1}^T(x)] \big)\Big]\Big\|_2  \\
    = &  \Big\|\EE_{t-1} \Big[ f_t \big(\xi_t,\hat F_{t+1}^T(x) \big) - f_t \big(\xi_t,\EE_t [\hat F_{t+1}^T(x)] \big) + \nabla_{x_t} f_t(\xi_t,\EE_t[\hat F_{t+1}^T(x)])^\top \big(\hat F_{t+1}^T(x) - \EE_t[\hat F_{t+1}^T(x)]\big) \Big]\Big\|_2  \\
   \leq &\frac{S_t}{2} \; \EE_{t-1}\Big[ \Big\|\hat F_{t+1}^T(x) - \EE_t[\hat F_{t+1}^T(x)]\Big\|_2^2 \Big] \\
   = & \frac{S_t}{2} \; \EE_{t-1} \Big[ \Big\| \hat\EE_t \Big[ f_{t+1}\big( \xi_{t+1}, \hat F_{t+2}^T(x) \big) \Big] - \EE_t \Big[ f_{t+1}\big( \xi_{t+1}, \hat F_{t+2}^T(x) \big) \Big] \Big\|_2^2\Big] \leq \frac{S_t\sigma_{t+1}^2}{2n_{t+1}}.
\end{align*}
Here, the first inequality follows from the bound in~\eqref{eq:inequality_for_smooth_funs}, which applies thanks to Assumption~\ref{assump:smooth}. The second inequality holds again because, conditional on $\xi_{[t]}$, 
the variance of an average of $n_{t+1}$ i.i.d.\ random variables coincides with the variance of a single random variable divided by~$n_{t+1}$. Substituting the above estimate into~\eqref{eq:smooth_case_lipschitz_step} and combining the resulting bound with~\eqref{eq:telescoping} finally yields 
\begin{align*}
    \big|\EE [\hat F(x)] - F(x) | \leq \sum_{t=1}^{T-1} \frac{L_{[t-1]}S_t\sigma_{t+1}^2}{2n_{t+1}}.
\end{align*}
This observation completes the proof.
\end{proof}

\begin{proof}[Proof of Lemma~\ref{lemma:bound_on_variance_saa}] 
Since $F(x)$ is deterministic, the variance of the SAA estimator $\hat F(x)$ satisfies
\begin{align*}
    \VV \big( \hat{F}(x) \big)= \VV \big( \hat{F}(x) - F(x) \big) \leq \EE\Big[ \big(\hat{F}(x) -F(x) \big)^2 \Big].
\end{align*}
To bound the variance of $\hat F(x)$, it thus suffices to bound the second moment of the estimation error $\hat{F}(x) -F(x)$.  By the construction of the conditional risk mappings in Definitions~\ref{def:risk-mapping} and~\ref{def:risk-mapping-hat}, we have 
\[
    F(x) = F_1^T(x) = \hat F_1^0(F_1^T(x)) \quad \text{and} \quad \hat{F}(x) = \hat F_1^T(x) = \hat F_1^T(F_{T+1}^T(x))
\]
for all $x\in\RR^d$. We can thus represent the estimation error as the telescoping sum 
\begin{align*}
     \hat{F}(x) - F(x)=\sum_{t=0}^{T-1} \Big(\hat F_1^t(F_{t+1}^T(x)) - \hat F_1^{t+1}(F_{t+2}^T(x))\Big).
\end{align*}
Therefore, the variance of $\hat{F}(x)$ satisfies
\begin{align*}
    \VV \big( \hat{F}(x) \big)^{\frac{1}{2}} \leq &  \EE\bigg[\sum_{t=0}^{T-1} \Big( \hat F_1^t(F_{t+1}^T(x)) - \hat F_1^{t+1}(F_{t+2}^T(x)) \Big)^2\bigg]^\frac{1}{2}\\
    \leq & \sum_{t=0}^{T-1} \EE\bigg[ \Big(\hat F_1^t(F_{t+1}^T(x)) - \hat F_1^{t+1}(F_{t+2}^T(x))\Big)^2\bigg]^\frac{1}{2}\\
    = & \sum_{t=0}^{T-1} \EE\bigg[ \bigg(\hat F_1^t \Big( \EE_t\big[ f_{t+1}\big( \xi_{t+1}, F_{t+2}^T(x) \big) \big] \Big) - \hat F_1^t \Big( \hat \EE_t\big[ f_{t+1}\big( \xi_{t+1}, F_{t+2}^T(x) \big) \big] \Big) \bigg)^2\bigg]^\frac{1}{2},
\end{align*}
where the second inequality follows from the Minkowski inequality. From now on, we use the notational shorthand $A_t\coloneqq f_{t}(\xi_{t}, F_{t+1}^T(x))$ for all $t\in[T]$. As $\hat{F}_1^t$ is $L_{[t]}$-Lipschitz continuous, we thus find
  \begin{align*}
    \VV \big( \hat{F}(x) \big)^{\frac{1}{2}} \leq & \sum_{t=0}^{T-1} L_{[t]}\; \EE\bigg[ \Big(\hat \EE\Big[\Big\|\hat \EE_t[ A_{t+1} ]- \EE_t [A_{t+1} ]\Big\|_2\Big]\Big)^2 \bigg]^\frac{1}{2}\\
    \leq &\sum_{t=0}^{T-1} L_{[t]}\; \EE\bigg[ \Big\|\hat \EE_t[ A_{t+1} ]- \EE_t [A_{t+1} ]\Big\|_2^2 \bigg]^\frac{1}{2} \\
    \leq &  \sum_{t=0}^{T-1} \frac{L_{[t]}}{\sqrt{n_{t+1}}}\; \EE\bigg[ \Big\|A_{t+1}- \EE_t [A_{t+1} ]\Big\|_2^2 \bigg]^\frac{1}{2}\leq \sum_{t=0}^{T-1} \frac{L_{[t]}\sigma_{t+1}}{\sqrt{n_{t+1}}}.
\end{align*}
Here, the second inequality exploits Jensen's inequality and the observation that $\EE[\hat\EE[\cdot]] = \EE[\cdot]$. The third inequality follows from the formula for the (conditional) variance of a sample average, and the last inequality follows from Assumption~\ref{assump:finite_variance}. Squaring both sides of the above inequality then yields the basic bound
\begin{align}
    \label{eq:basic-variance-bound}
    \VV \big( \hat{F}(x) \big) \leq  \bigg(\sum_{t=0}^{T-1} \frac{L_{[t]}\sigma_{t+1}}{\sqrt{n_{t+1}}}\bigg)^2.
\end{align}
This basic bound can be strengthened by recalling that $\hat F(x)=\frac{1}{n_1}\sum_{i_1=1}^{n_1} \hat F_{i_1}(x)$, where $\hat F_{i_1}(x)$, $i_1\in[n_1]$, constitute independent and identically distributed random variables that depend on the samples of the $i_1$-th scenario tree used in the construction of the SAA estimator; see Definition~\ref{def:MCCO_sample_average_approximation}. Thus, we have
\begin{align*}
     \VV \big( \hat{F}(x) \big) =\frac{1}{n_1^2}\sum_{i_1=1}^{n_1} \VV\big(\hat F_{i_1}(x)\big) \leq  \frac{1}{n_1}\bigg(\sum_{t=1}^{T-1} \frac{L_{[t]}\sigma_{t+1}}{\sqrt{n_{t+1}}}+\sigma_1\bigg)^2
\end{align*}
where the inequality uses~\eqref{eq:basic-variance-bound} with $n_1=1$ to bound the variance of $\hat F_{i_1}(x)$. Hence, the claim follows.
\end{proof}

The next lemma shows that sub-Gaussianity is preserved under averaging and concatenation as well as under transformations with Lipschitz continuous functions and mixtures. Among other things, this lemma enables us to identify easily verifiable conditions under which Assumption~\ref{assump:sub-gaussianity_saa} holds.

\begin{lemma}[Operations Preserving Sub-Gaussianity]
\label{lemma:sub-gaussianity-preserving-operations} The following hold.
\begin{enumerate}[label=(\roman*)]
    \item If $z_i$, $i\in[n]$, are i.i.d.\ sub-Gaussian random vectors with common variance proxy~$\zeta^2$, then their average $\frac{1}{n}\sum_{i=1}^n z_i$ constitutes another sub-Gaussian random vector with variance proxy~$\zeta^2/n$.
    \item If $z_1$ and $z_2$ are (possibly dependent)\ sub-Gaussian random vectors valued in $\RR^{m_1}$ and $\RR^{m_2}$ with respective variance proxies $\zeta_1^2$ and $\zeta_2^2$, then $(z_1,z_2)$ is sub-Gaussian with variance proxy $2\max\{\zeta_1^2,\zeta_2^2\}$.
    \item If $z$ is an $m$-dimensional sub-Gaussian random vector with variance proxy $\zeta^2$, and if $f:\RR^m\rightarrow \RR^{n}$ is an $L$-Lipschitz continuous function with $\|f(z)-f(z^\prime)\|_2\leq L\|z-z^\prime\|_2$ for all $z,z'\in\RR^m$, then $f(z)$ is a sub-Gaussian random vector with variance proxy that grows monotonically with $L^2m^2\zeta^2$.
    \item If $z_i$, $i\in[n]$ are (possibly dependent) sub-Gaussian random vectors valued in $\RR^{m}$ with respective variance proxies $\zeta_i^2$, $i\in[n]$, and if $\alpha$ is a random index independent of all other random objects with $\PP(\alpha=i)=w_i$ for all $i\in[n]$, then the mixture $z_{\alpha}$ is a sub-Gaussian random vector with variance proxy $\max_{i\in[n]}\{\zeta_i^2\} +\max_{i \in [n]} \{\|\EE[z_i] - \EE[z_\alpha]\|_2\}$.
\end{enumerate}
\end{lemma}
\begin{proof} Assertion~(i) follows from an immediate generalization of  \cite[Exercise~2.13]{wainwright2019high}.
    
As for assertion~(ii), observe that
\begin{align*}
    \EE\big[\exp\big((\lambda_1, \lambda_2)^\top  ((z_1,z_2)- \EE&[(z_1,z_2)])\big)\big] = \EE\big[\exp\big(\lambda_1^\top (z_1-\EE[z_1])\big)\exp\big(\lambda_2^\top(z_2-\EE[z_2])\big)\big]\\
    & \leq \EE\big[\exp\big(2\lambda_1^\top (z_1-\EE[z_1])\big)\big]^{1/2} \, \EE\big[\exp\big(2\lambda_2^\top(z_2-\EE[z_2])\big)\big]^{1/2}\\
    & \leq \exp(\|\lambda_1\|_2^2 \, \zeta_1^2)\exp(\|\lambda_2\|_2^2 \, \zeta_2^2 )\leq \exp(\|(\lambda_1,\lambda_2)\|_2^2 \, \max\{\zeta_1^2,\zeta_2^2\})
\end{align*}
for all $\lambda_1\in\mathbb R^{m_1}$ and $\lambda_2 \in\mathbb R^{m_2}$. Here, 
the three inequalities follow from Hölder's inequality, the sub-Gaussianity of $z_1$ and $z_2$ and the monotonicity of the exponential function, respectively. This shows that the combined random vector $(z_1,z_2)$ is indeed sub-Gaussian with variance proxy $2\max\{\zeta_1^2, \zeta_2^2\}$.

As for assertion~(iii), let $z^\prime$ be an independent copy of~$z$. Thus, $-z^\prime$ has the same variance proxy as $z$, which implies via~\cite[Exercise~2.13]{wainwright2019high} that $z - z^\prime$ is sub-Gaussian with variance proxy $2\zeta^2$. Hence, we find 
\begin{align*}
    \EE\big[ \exp\big(\lambda^\top(f(z)-\EE[f(z)])\big) \big] 
    & = \EE \big[ \exp\big(\lambda^\top(\EE [f(z)-f(z^\prime) | z])\big) \big] \\
    & \leq \EE \big[ \exp\big(\lambda^\top\big(f(z)-f(z^\prime)\big)\big) \big] \\ 
    & \leq \EE \big[ \exp\big(\|\lambda\|_2 \, \|f(z)-f(z^\prime)\|_2\big) \big] \\
    & \leq \EE \big[ \exp\big(L\,\|\lambda\|_2 \,\|z-z^\prime\|_2\big) \big] ,
\end{align*}
where we have used Jensen's inequality, the Cauchy-Schwarz inequality and the $L$-Lipschitz continuity of~$f$. Since $\|\cdot\|_2\leq\|\cdot\|_1$ and as the exponential function is monotonically increasing, we also find
\begin{align*}
    \exp(\|z-z^\prime\|_2)\leq\exp(\|z-z^\prime\|_1)=\prod_{i=1}^m\exp(|z_i-z^\prime_i|).
\end{align*}
Combining the above estimates and applying the multi-variable Hölder inequality then yields
\begin{align}
    \nonumber
    \EE\big[ \exp\big(\lambda^\top(f(z)-\EE[f(z)])\big) \big] & \leq \EE \bigg[ \prod_{i=1}^m\exp\big(L \, \|\lambda\|_2 \, |z_i-z^\prime_i|\big) \bigg] \\
    & \leq \prod_{i=1}^m \EE \big[ \exp\big(Lm \, \|\lambda\|_2 \, |z_i-z^\prime_i|\big) \big]^{1/m}.
    \label{eq:subgaussian_holder_step}
\end{align}
As $z^{1/m}$ increases with~$z\geq 0$, an upper bound on~\eqref{eq:subgaussian_holder_step} follows by bounding the underlying expectations as 
\begin{align*}
    \EE\big[ \exp\big(Lm\,\|\lambda\|_2 \,|z_i-z^\prime_i|\big) \big]
    & \leq \EE \big[ \exp\big(Lm \, \|\lambda\|_2 \, (z_i-z^\prime_i)\big) \big] + \EE \big[ \exp\big(Lm \, \|\lambda\|_2\, (z_i^\prime -z_i)\big) \big] \\
    & \leq 2\exp\big(\|\lambda\|_2^2 \, L^2m^2\zeta^2 \big).
\end{align*}
Here, the two inequalities hold because $\exp(|a|)\leq\exp(a)+\exp(-a)$ for any $a\in\RR$ and because $z_i - z_i^\prime$ is sub-Gaussian with variance proxy~$2\zeta^2$. Substituting this estimate into \eqref{eq:subgaussian_holder_step} shows that
\begin{align}
    \label{eq:preliminary-subgaussian}
    \EE\big[ \exp\big(\lambda^\top(f(z)-\EE[f(z)])\big) \big]
    & \leq 2\exp\big(\|\lambda\|_2^2 \, L^2m^2\zeta^2 \big).
\end{align}
It remains to be shown that the factor~$2$ can be absorbed in the variance proxy. For ease of notation, we henceforth use $\psi(\lambda)$ as shorthand for the left hand side of~\eqref{eq:preliminary-subgaussian}. That is, $\psi(\lambda)$ denotes the moment-generating function of $f(z)-\EE[f(z)]$. We must now prove that there exists $\varsigma>0$ such that $\psi(\lambda)\leq \exp(\|\lambda\|_2^2\varsigma^2/2)$ for all~$\lambda\in\RR^n$. To this end, note that~\eqref{eq:preliminary-subgaussian} ensures that $\psi(\lambda)$ is analytic on a neighborhood of~$0$ \cite[Section~2]{wainwright2019high}. Thus, there exists~$\varsigma^2_1>0$ with $\nabla^2\psi(0)\preceq \varsigma_1^2 I_m$. In addition, we have $\psi(0)=1$ and $\nabla\psi(0)=0$. Next, define $\phi_1(\lambda) \coloneqq \exp(\|\lambda\|^2 \varsigma_1^2 / 2)$, and note that $\phi_1(0)= 1$, $\nabla \phi_1(0) = 0$ and $\nabla^2 \phi_1(0) = \varsigma_1^2 I_m$. Thus, Taylor's theorem guarantees that there exists~$\epsilon>0$ such that $\psi(\lambda)\leq \phi_1(\lambda)$ whenever $\|\lambda\|_2\leq \epsilon$. Define now $\varsigma_2^2\coloneqq 2(\log(2)/\epsilon^2 + L^2m^2\zeta^2)$, and set $\phi_2(\lambda)\coloneqq \exp(\|\lambda\|^2 \varsigma_2^2 / 2)$. By~\eqref{eq:preliminary-subgaussian}, we then have
\begin{align*}
    \psi(\lambda) &\leq  \exp\big(\log(2) + \|\lambda\|_2^2L^2m^2\zeta^2\big) = \exp\Big(\|\lambda\|_2^2 \big(\log(2)/ \|\lambda\|_2^2 + L^2m^2\zeta^2 \big)\Big) 
    < \phi_2(\lambda)
\end{align*}
whenever $\|\lambda\|_2>\epsilon$. We may thus conclude that 
\[
    \psi(\lambda)\leq \max\{\phi_1(\lambda),\phi_2(\lambda)\} = \exp\left( \|\lambda\|^2 \max\{\varsigma_1^2, \varsigma_2^2\}/2 \right) \quad \forall \lambda\in\mathbb R^m .
\]
Consequently, $f(z)$ is sub-Gaussian with variance proxy $\max\{\varsigma_1^2, \varsigma_2^2\}$.

As for assertion~(iv), observe that
\begin{equation*}
    \begin{aligned}
    \EE\big[ \exp(\lambda^\top (z_\alpha - \EE[z_\alpha]))\big] 
    & = \sum_{i=1}^n w_i\,\EE\big[\exp(\lambda^\top (z_i-\EE[z_\alpha])) \big] \\
    & = \sum_{i=1}^n w_i\,\exp(\lambda^\top (\EE[z_i]-\EE[z_\alpha])) \, \EE\big[\exp(\lambda^\top (z_i-\EE[z_i])) \big], 
\end{aligned}
\end{equation*}
where the first equality holds because $\alpha$ is independent of~$z_i$, $i\in[n]$, and the second equality centers each $z_i$ around its own mean. The sub-Gaussianity of each $z_i$ implies that $\mathbb{E}[\exp(\lambda^\top (z_i - \EE[z_i]))] \leq \exp(\|\lambda\|_2^2 \zeta_i^2 / 2)$. Factoring out a uniform bound involving $\zeta_{\max}^2 := \max_{i \in [n]} \{\zeta_i^2\}$, we obtain
\[
   \EE\big[ \exp(\lambda^\top (z_\alpha - \EE[z_\alpha]))\big]  \leq \exp(\|\lambda\|_2^2 \zeta_{\max}^2/2) \, \sum_{i=1}^n w_i \exp(\lambda^\top (\EE[z_i]-\EE[z_\alpha])) .
\]
Here, the sum can be interpreted as the moment generating function of a discrete random variable~$Y$ taking values $y_i=\lambda^\top (\EE[z_i]-\EE[z_\alpha])$ with probabilities~$w_i$. Note that $\mathbb{E}[Y] = \sum w_i \lambda^\top (\EE[z_i] - \EE[z_\alpha]) = 0$. Note also that $|y_i |\leq \|\lambda\|_2 R$ for $R := \max_{i \in [n]} \{\|\EE[z_i] - \EE[z_\alpha]\|_2\}$, that is, the variable $Y$ is bounded. Hoeffding’s Lemma (see, {\em e.g.}, \cite[Example~2.4]{wainwright2019high}) thus implies that $\mathbb{E}[\exp(Y)] \leq \exp(\|\lambda\|_2^2 R^2 / 2)$. In summary, we find
\[
\EE\big[ \exp(\lambda^\top (z_\alpha - \EE[z_\alpha]))\big] \leq \exp\big(\|\lambda\|_2^2 (\zeta_{\max}^2 + R^2) /2 \big),
\]
which establishes that $z_\alpha$ is sub-Gaussian with variance proxy $\zeta_{\max}^2 + R^2$.
\end{proof}

\begin{proof}[Proof of Lemma~\ref{lemma:unif_conv_saa}] 
Assumptions~\ref{assump:general} and~\ref{assump:lipschitz} ensure that~$\cX$ is compact and that both~$F$ and~$\hat{F}$ are $L_{[T]}$-Lipschitz continuous; see Lemma~\ref{lem:lipschitz} as well as the discussion after Definition~\ref{def:risk-mapping-hat}. Next, set $v=\epsilon/(4L_{[T]})$. By Lemma~\ref{lemma:v-net}, there exists a $v$-net $\cX_v\subseteq \cX$ with $|\cX_v|\leq \lceil 2D_{\cX}/v + 1\rceil^d$ such that
\begin{align}
    \label{eq:sumoverJ}
    \PP\Big(\sup_{x\in \mathcal{X}} \lvert \hat{F}(x) - F(x) \rvert > \epsilon \Big) \leq \sum_{x^\prime\in\cX_v} \PP\Big( \lvert  \hat{F}(x^\prime) - F(x^\prime) \rvert  > \frac{\epsilon}{2}\Big).
\end{align}
Recall that the SAA estimator is given by $\hat{F}(x) = \frac{1}{n_1} \sum_{i_1=1}^{n_1} \hat{F}_{i_1}(x)$, where the~$\hat{F}_{i_1}(x)$, $i_1\in[n_1]$, can be interpreted as independent and identically distributed estimators for~$F(x)$, each of which is constructed from a single scenario tree; see Definition~\ref{def:MCCO_sample_average_approximation}. In the following we introduce the auxiliary bias function
\begin{align*}
    \mu(x)\coloneqq\EE\big[\hat{F}_{i_1}(x)-F(x)\big] = \EE\big[\hat{F}(x)-F(x)\big].
\end{align*}
Given any $\epsilon>0$, we now seek sufficiently large sample sizes to ensure that $|\mu(x')|\leq \frac{\epsilon}{4}$ for every~$x'\in\cX_v$. If the integrands are Lipschitz continuous but not smooth ({\em i.e.}, only Assumption~\ref{assump:lipschitz} holds), then we have
\begin{align*}
    n_{t+1} =\bigg\lceil\left( \frac{4L_{[t]}\sigma_{t+1} (T-1)}{\epsilon} \right)^2\bigg\rceil \quad \forall t\in[T-1]\quad \implies\quad  \frac{L_{[t]}\sigma_{t+1}}{\sqrt{n_{t+1}}} \leq \frac{\epsilon}{4(T-1)} \quad \forall t\in[T-1],
\end{align*}
which ensures via Lemma~\ref{lemma:bound_on_bias_saa} that
\begin{align*}
    \left| \mu(x^\prime) \right| \leq \sum_{t=1}^{T-1} \frac{L_{[t]}\sigma_{t+1}}{\sqrt{n_{t+1}}} \leq \frac{\epsilon}{4} \quad \forall x'\in\cX_v.
\end{align*}
Similarly, if the integrands are Lipschitz continuous and smooth ({\em i.e.}, both Assumptions~\ref{assump:lipschitz} and~\ref{assump:smooth} hold), then we have
\begin{align*}
    n_{t+1}=\bigg\lceil \frac{2L_{[t-1]}S_t\sigma_{t+1}^2 (T-1)}{\epsilon} \bigg\rceil \quad \forall t\in[T-1] \quad \implies \quad \frac{L_{[t-1]}S_t\sigma_{t+1}^2}{ 2 n_{t+1}} \leq \frac{\epsilon}{4(T-1)} \quad \forall t\in[T-1],
\end{align*}
which ensures via Lemma~\ref{lemma:bound_on_bias_saa} that
\begin{align*}
    \left| \mu(x^\prime) \right|\leq \sum_{t=1}^{T-1} \frac{L_{[t-1]}S_t\sigma_{t+1}^2}{2n_{t+1}} \leq \frac{\epsilon}{4} \quad \forall x'\in\cX_v.
\end{align*}
In both cases, we may conclude that 
\begin{align*}
    \PP\left( \hat{F}(x^\prime) - F(x^\prime) > \frac{\epsilon}{2}\right) 
    &\leq \PP\left( \hat{F}(x^\prime) - F(x^\prime)  > \frac{\epsilon}{4} + \mu(x^\prime) \right)  \\
    &= \PP \left( \frac{1}{n_1}\sum_{i_1=1}^{n_1}\big(\hat F_{i_1}(x^\prime)-F(x^\prime)-\mu(x^\prime) \big) > \frac{\epsilon}{4} \right) \leq \exp\left(- \frac{n_1\epsilon^2}{32 \zeta^2} \right)
\end{align*}
for all $x'\in\cX_v$. The second inequality follows from a standard Hoeffding bound \cite[Proposition~2.5]{wainwright2019high}, which applies because the i.i.d.\ centered random variables $\hat F_{i_1}(x)-F(x)-\mu(x)$, $i\in[n_1]$, are sub-Gaussian with variance proxy~$\zeta^2$ thanks to Assumption~\ref{assump:sub-gaussianity_saa}. Following a similar reasoning, one can also show that 
\begin{align*}
    \PP\left( F(x^\prime)- \hat{F}(x^\prime)  > \frac{\epsilon}{2}\right) \leq \exp\left(- \frac{n_1\epsilon^2}{32\zeta^2} \right)
\end{align*}
for all $x'\in\cX_v$. Combining the above probability bounds with~\eqref{eq:sumoverJ} finally yields
\begin{align*}
    \PP\left(\sup_{x\in \mathcal{X}} \lvert \hat{F}(x) - F(x) \rvert > \epsilon \right) 
    &\leq 2|\cX_v| \exp\bigg(- \frac{n_1\epsilon^2}{32 \zeta^2} \bigg) \leq 2\left\lceil 8L_{[T]} D_\mathcal{X}/\epsilon + 1 \right\rceil^{d} \exp\bigg(- \frac{n_1\epsilon^2}{32\zeta^2} \bigg),
\end{align*}
and thus the claim follows.
\end{proof}

\begin{proof}[Proof of Corollary~\ref{cor:high_prob_bound_saa_epsilon_optimal}] 
Problems~\eqref{problem:MCCO} and~\eqref{prob:min_hatF(x)} are solvable by virtue of Weierstrass' extreme value theorem, which applies thanks to Assumptions~\ref{assump:general} and~\ref{assump:lipschitz} and Lemma~\ref{lem:lipschitz}. Hence, $x^*$ and $\hat x^*$ exist. Elementary manipulations then show that
\begin{align*}
    \PP\left(   F(\hat{x}^*) - F(x^*)  > \epsilon \right) 
    &= \PP\left(   F(\hat{x}^*) - \hat{F}(\hat{x}^*) + \hat{F}(\hat{x}^*) - \hat{F}(x^*) + \hat{F}(x^*) - F(x^*)  > \epsilon \right) \\
    &\leq \PP\left( F(\hat{x}^*) - \hat{F}(\hat{x}^*) > \epsilon/2 \;\text{ or }\;  \hat{F}(x^*) - F(x^*)  > \epsilon/2 \right) \\
    &\leq   \PP\left( F(\hat{x}^*) - \hat{F}(\hat{x}^*) > \epsilon/2 \right) + \PP\left( \hat{F}(x^*) - F(x^*)  > \epsilon/2 \right)   \\
    &\leq 4\left\lceil 8L_{[T]} D_\mathcal{X}/\epsilon +1\right\rceil^{d} \exp\bigg(- \frac{n_1\epsilon^2}{128\zeta^2} \bigg),
\end{align*}
where the first inequality holds because $\hat{F}(\hat{x}^*) - \hat{F}(x^*) \leq 0$. The second and the third inequalities follow from the union bound and from Lemma~\ref{lemma:unif_conv_saa}, respectively.
\end{proof}

\begin{proof}[Proof of Theorem~\ref{thm:SAA_sample_complexity_high_probability}]
    The condition on $n_1$ ensures via Corollary~\ref{cor:high_prob_bound_saa_epsilon_optimal} that $ \PP\left( F(\hat{x}^*) - F(x^*)  > \epsilon \right) \leq \beta$ for any minimizers~$x^*$ and~$\hat{x}^*$ of problems \eqref{problem:MCCO} and~\eqref{prob:min_hatF(x)}, respectively. Hence, any minimizer of problem~\eqref{prob:min_hatF(x)} is an $\epsilon$-optimal solution to~\eqref{problem:MCCO} with probability at least $1-\beta$. The scenario complexity of the SAA estimator is thus obtained by multiplying the given expressions for the sample sizes $n_t$, $t\in[T]$.
\end{proof}


Next, we provide the proofs of all formal results in Section~\ref{sec:rtmlmc}.

\begin{proof}[Proof of Lemma~\ref{lemma:bound_on_variance_rtmlmc}]
Since $x$ is fixed, we simplify notation by suppressing the dependence on $x$ for all functions of $x$ throughout the proof. Suppose first that only Assumptions~\ref{assump:general}, \ref{assump:lipschitz} and~\ref{assump:mu_bar} hold. We prove the upper bound on~$\mu_t^2$ by backward induction on~$t$. As for the base case corresponding to~$t=T$, note that 
\[
    \mu_T^2=\EE_{T-1}[\|\hat{H}_T\|_2^2]= \EE_{T-1}[\|f_T(\xi_T,x)\|_2^2]\leq \overline \mu_T^2,
\]
where the inequality follows from the definition of~$\overline \mu_T^2$. As for the induction step corresponding to any~$t$ with~$t+1\in[T]$, we use the law of total expectation and the definition of $\hat H_t(x)$ in~\eqref{eq:H_hat} to obtain
\begin{equation}
    \label{eq:moment_of_H_telescoping}
    \begin{aligned}
    \mu_t^2 = \EE_{t-1}\big[\|\hat{H}_t\|^2_2 \big] & =\sum_{\ell=0}^{M_{t}} q_t(\ell)\,\EE_{t-1}\left[ \left. \frac{1}{q_t(\lambda_t)^2} \big\| \hat{h}_{t}^{\lambda_t}-\frac{1}{2}\hat{h}_t^{\lambda_t, {\rm e}}-\frac{1}{2}\hat{h}_t^{\lambda_t, {\rm o}}\big\|^2_2 \,\right| \lambda_t=\ell\right] \\
    & =\sum_{\ell=0}^{M_{t}} \frac{1}{q_t(\ell)} \EE_{t-1}\Big[\big\|\hat{h}_{t}^{\ell}-\frac{1}{2}\hat{h}_t^{\ell, {\rm e}}-\frac{1}{2}\hat{h}_t^{\ell, {\rm o}}\big\|^2_2\Big].
\end{aligned}
\end{equation}
In the second line we have eliminated the condition on the event $\lambda_t=\ell$, which is allowed because~$\lambda_t$ is independent of all other random objects. Next, define $\Delta^\ell\coloneqq\hat{h}_{t}^{\ell}-\frac{1}{2}\hat{h}_t^{\ell, {\rm e}}-\frac{1}{2}\hat{h}_t^{\ell, {\rm o}}$ to avoid clutter, and note~that
\begin{align*}
    \|\Delta^\ell\|_2 & \leq \frac{L_t}{2}\big( \big\|\hat\EE_t^{\ell,{\rm e}}[\hat{H}_{t+1}] - \hat\EE_t^{\ell}[\hat{H}_{t+1}]\big\|_2 +\big\|\hat \EE_t^{\ell,{\rm o}}[\hat{H}_{t+1}] -\hat\EE_{t}^{\ell}[\hat{H}_{t+1}] \big\|_2 \big)\\
    & = \frac{L_t}{2} \big\|\hat \EE_t^{\ell,{\rm e}}[\hat{H}_{t+1}] -\hat\EE_{t}^{\ell,{\rm o}}[\hat{H}_{t+1}]\big\|_2 \\
    & \leq \frac{L_t}{2} \big(\big\|\hat \EE_t^{\ell,{\rm e}}[\hat{H}_{t+1}] -\EE_{t}[\hat{H}_{t+1}] \big\|_2 + \big\|\hat \EE_t^{\ell,{\rm o}}[\hat{H}_{t+1}] - \EE_{t}[\hat{H}_{t+1}] \big\|_2\big),
\end{align*}
where the first inequality holds because of Assumption~\ref{assump:lipschitz}, which ensures that~$f_t(\xi_t,x_t)$ is $L_t$-Lipschitz continuous in~$x_t$, and because of the triangle inequality. The equality follows from the antithetic sample average identity~\eqref{eq:even-odd-averages-identity}, and the second inequality follows again from the triangle inequality. As $(a+b)^2\leq 2a^2+2b^2$ for all $a,b\in\mathbb R$, the conditional second moment of $\Delta^\ell$ admits the following upper bound.
\begin{align*}
    \EE_{t-1}\big[\|\Delta^\ell\|_2^2\big] & \leq \frac{L_t^2}{2} \left(\EE_{t-1}\Big[ \big\| \hat \EE_t^{\ell, {\rm e}}[\hat{H}_{t+1}] -\EE_{t}[\hat{H}_{t+1}] \big\|_2^2 \Big]
    + \EE_{t-1} \Big[ \big\|\hat \EE_t^{\ell,{\rm o}}[\hat{H}_{t+1}] -
    \EE_{t}[\hat{H}_{t+1}] \big\|_2^2 \Big] \right)\\
    & = L_t^2 \;\EE_{t-1}\Big[\big\|\hat \EE_t^{\ell, {\rm e}}[\hat{H}_{t+1}] -\EE_{t}[\hat{H}_{t+1}] \big\|_2^2 \Big] \leq \frac{B_2 L_t^2}{2^{\ell-1}} \; \EE_{t-1}\big[\|\hat{H}_{t+1}\|_2^2\big] =\frac{B_2 L_t^2}{2^{\ell-1}} \; \EE_{t-1}\big[\mu_{t+1}^2\big].
\end{align*}
Here, the first equality holds because the even and the odd sample averages of $\hat H_t$ have the same distribution conditional on~$\mathcal F_{t}$, and the second inequality follows from Lemma~\ref{lemma:moment-inequality-multidimensional}. The second equality follows from the definition of~$\mu_{t+1}^2$ and the law of total expectation. Substituting this bound back into~\eqref{eq:moment_of_H_telescoping} yields
\begin{align*}
    \mu_t^2 &\leq 2B_2L_t^2 \; \EE_{t-1}\big[ \mu_{t+1}^2 \big]\sum_{\ell=0}^{M_{t}}\frac{1}{q_t(\ell) 2^{
    \ell}}  \\
    &\leq 2B_2L_t^2 \overline\mu_T^2 \prod_{s=t+1}^{T-1}\bigg(2B_2L^2_s\sum_{\ell=0}^{M_s}\frac{1}{2^s q_s(\ell)} \bigg) \sum_{\ell=0}^{M_{t}}\frac{1}{q_t(\ell)2^{\ell} } = \overline\mu_T^2 \prod_{s=t}^{T-1}\bigg(2B_2 L_s^2 \sum_{\ell=0}^{M_s}\frac{1}{2^s q_s(\ell)} \bigg),
\end{align*}
where the second inequality follows from the induction hypothesis. This establishes the first claim.

Suppose now that Assumption~\ref{assump:smooth} holds, too. In this case, we prove the upper bound on $\mu_t^{2^t}$ by backward induction on~$t$. The base case corresponding to~$t=T$ is trivially true because
\[
    \mu_T^{2^T} =\EE_{T-1}\big[\|\hat{H}_T\|_{2^T}^{2^T}\big] =\EE_{T-1}\big[\|f_T(\xi_T,x)\|_{2^T}^{2^T}\big]\leq \overline\mu_T^{2^T},
\]
where the inequality uses the definition of~$\overline \mu_T^{2^T}$. As for the induction step corresponding to any~$t$ with~$t+1\in[T]$, we can use the law of total expectation and define $\Delta^\ell$ as in the first part of the proof to demonstrate that
\begin{align}
    \label{eq:mu_t^2^t}
    \mu_t^{2^t} = \EE_{t-1}[\|\hat{H}_t\|^{2^t}_{2^t}] 
    =\sum_{\ell=0}^{M_{t}}\frac{\EE_{t-1}\big[\|\Delta^\ell\|^{2^t}_{2^t}\big]}{q_t(\ell)^{2^t-1}}.
\end{align}
Next, we can use the antithetic sample average identity~\eqref{eq:even-odd-averages-identity} to reformulate $\Delta^\ell$ as
\begin{align*}
    \Delta^\ell & = \hat{h}_{t}^{\ell} - f_t(\xi_t,\EE_t[\hat{H}_{t+1}]) -\nabla_{x_t} f_t(\xi_t,\EE_t[\hat{H}_{t+1}])^\top \big(\hat\EE_t^\ell[\hat{H}_{t+1}] -\EE_t[\hat{H}_{t+1}] \big)\\
    &\quad -\frac{1}{2}\Big[\hat{h}_t^{\ell, {\rm e}} - f_t(\xi_t,\EE_t[\hat{H}_{t+1}]) -\nabla_{x_t} f_t(\xi_t,\EE_t[\hat{H}_{t+1}])^\top \big(\hat\EE_t^{\ell,{\rm e}}[\hat{H}_{t+1}] -\EE_t[\hat{H}_{t+1}] \big)\Big]\\
    &\quad -\frac{1}{2}\Big[\hat{h}_t^{\ell, {\rm o}} - f_t(\xi_t,\EE_t[\hat{H}_{t+1}]) -\nabla_{x_t} f_t(\xi_t,\EE_t[\hat{H}_{t+1}])^\top \big(\hat\EE_t^{\ell,{\rm o}}[\hat{H}_{t+1}] -\EE_t[\hat{H}_{t+1}] \big)\Big].
\end{align*}
Taking norms on both sides and replacing $\hat{h}_{t}^{\ell}$, $\hat{h}_t^{\ell, {\rm e}}$ and $\hat{h}_t^{\ell, {\rm o}}$ by their definitions then yields
\begin{align*}
    \big\|\Delta^\ell \big\|_{2^t} \leq \big\| \Delta^\ell \big\|_2 & \leq \frac{S_t}{2}\, \Big\|\hat\EE_t^\ell[\hat{H}_{t+1}] - \EE_t[\hat{H}_{t+1}] \Big\|_2^2 \\ &\quad + \frac{S_t}{4} \, \Big\|\hat\EE_t^{\ell, {\rm e}}[\hat{H}_{t+1}] - \EE_t[\hat{H}_{t+1}] \Big\|_2^2 + \frac{S_t}{4}\, \Big\|\hat\EE_t^{\ell, {\rm o}}[\hat{H}_{t+1}] - \EE_t[\hat{H}_{t+1}] \Big\|_2^2.
\end{align*}
Here, the two inequalities follow from Lemma~\ref{lemma:matrix-inequality-pq} and from a combination of the triangle inequality and the inequality~\eqref{eq:inequality_for_smooth_funs} for smooth functions, respectively. By the Minkowski inequality, we then obtain
\begin{align*}
    &\EE_{t-1}\big[\|\Delta^\ell\|_{2^t}^{2^t}\big]^{\frac{1}{2^t}} \\
    & \leq \frac{S_t}{2}\Big(\EE_{t-1} \Big[\Big\|\hat\EE_t^\ell[\hat{H}_{t+1}] - \EE_t[\hat{H}_{t+1}] \big\|_2^{2^{t+1}}\Big]\Big)^{\frac{1}{2^t}} 
    + \frac{S_t}{4}\Big(\EE_{t-1} \Big[ \big\|\hat\EE_t^{\ell, {\rm e}}[\hat{H}_{t+1}] - \EE_t[\hat{H}_{t+1}]\big\|_2^{2^{t+1}}\Big]\Big)^{\frac{1}{2^t}} \\
    & \quad + \frac{S_t}{4}\Big(\EE_{t-1} \Big[ \big\|\hat\EE_t^{\ell, {\rm o}}[\hat{H}_{t+1}] - \EE_t[\hat{H}_{t+1}]\big\|_2^{2^{t+1}} \Big]\Big)^{\frac{1}{2^t}}\\
    & = \frac{S_t}{2}\Big(\EE_{t-1} \Big[\big\|\hat\EE_t^\ell[\hat{H}_{t+1}] - \EE_t[\hat{H}_{t+1}]\big\|_2^{2^{t+1}} \Big]\Big)^{\frac{1}{2^t}} + \frac{S_t}{2}\Big(\EE_{t-1} \Big[ \big\|\hat\EE_t^{\ell, {\rm e}}[\hat{H}_{t+1}] - \EE_t[\hat{H}_{t+1}]\big\|_2^{2^{t+1}}\Big]\Big)^{\frac{1}{2^t}},
\end{align*}
where the equality holds because the even and the odd sample averages of $\hat H_t$ have the same distribution conditional on~$\mathcal F_t$. Applying Lemma~\ref{lemma:matrix-inequality-pq} with $p=2$, $q=2^{t+1}$, $m=1$ and~$n=d_t$ to the two differences of conditional expectations in the above expression then yields 
\begin{align*}
    & \EE_{t-1}\big[\|\Delta^\ell \|_{2^t}^{2^t}\big]^{\frac{1}{2^t}} \\
    &\leq \frac{S_t}{2}\Big((d_t)^{2^t-1}\EE_{t-1} \Big[\big\|\hat\EE_t^\ell[\hat{H}_{t+1}] - \EE_t[\hat{H}_{t+1}]\big\|_{2^{t+1}}^{2^{t+1}} \Big]\Big)^{\frac{1}{2^t}} \\
    & \hspace{1cm} + \frac{S_t}{2}\Big((d_t)^{2^t-1}\EE_{t-1} \Big[ \big\|\hat\EE_t^{\ell, {\rm e}}[\hat{H}_{t+1}] - \EE_t[\hat{H}_{t+1}]\big\|_{2^{t+1}}^{2^{t+1}}\Big]\Big)^{\frac{1}{2^t}}\\
    &\leq \frac{S_t}{2}\Big((d_t)^{2^t-1}\frac{B_{2^{t+1}}}{2^{2^{t}\ell}}\EE_{t-1}\Big[\big\|  \hat{H}_{t+1}\big\|_{2^{t+1}}^{2^{t+1}}\Big] \Big)^{\frac{1}{2^t}} + \frac{S_t}{2}\Big((d_t)^{2^t-1}\frac{B_{2^{t+1}}}{2^{2^t(\ell-1)}}\EE_{t-1}\Big[\big\|  \hat{H}_{t+1}\big\|_{2^{t+1}}^{2^{t+1}}\Big]\Big)^{\frac{1}{2^t}} \\
    &= \frac{3S_t}{2}\Big((d_t)^{2^t-1}\frac{B_{2^{t+1}}}{2^{2^t\ell}}\EE_{t-1}\Big[\big\|  \hat{H}_{t+1}\big\|_{2^{t+1}}^{2^{t+1}}\Big]\Big)^{\frac{1}{2^t}}= \frac{3S_t}{2}\Big((d_t)^{2^t-1}\frac{B_{2^{t+1}}}{2^{2^t\ell}}\EE_{t-1}\Big[ \mu_{t+1}^{2^{t+1}}\Big]\Big)^{\frac{1}{2^t}}.
\end{align*}
Here, the second inequality follows from Lemma~\ref{lemma:moment-inequality-multidimensional} with $p=2^{t+1}$ and $n=2^\ell$ (for sample averages involving even {\em as well as} odd samples) or $n=2^{\ell-1}$ (for sample averages involving only even {\em or} odd samples). Next, by raising both sides of the resulting inequality to the $2^t$-th power, we find
\begin{align*}
    \EE_{t-1}\big[\|\Delta^\ell \|_{2^t}^{2^t}\big] &\leq \bigg(\frac{3S_t}{2}\bigg)^{2^t} (d_t)^{2^{t}-1} B_{2^{t+1}} \frac{1}{2^{2^t\ell}} \, \EE_{t-1}\Big[ \mu_{t+1}^{2^{t+1}} \Big]\\
    &\leq \bigg(\frac{3S_t}{2}\bigg)^{2^t} (d_t)^{2^{t}-1} B_{2^{t+1}} \frac{1}{2^{2^t\ell}}  \; \overline\mu_T^{2^T} \prod_{s=t+1}^{T-1} \bigg(\bigg(\frac{3S_s}{2}\bigg)^{2^s} d_s^{2^s-1} B_{2^{s+1}} \sum_{\ell^\prime=0}^{M_s} \frac{1}{2^{2^s \ell^\prime} q_s(\ell^\prime)^{2^{s}-1}} \bigg),
\end{align*}
where the second inequality uses the induction hypothesis. Substituting this back into~\eqref{eq:mu_t^2^t} finally yields
\begin{align*}
    \mu_t^{2^t} \leq \overline\mu_T^{2^T} \prod_{s=t}^{T-1} \bigg(\bigg(\frac{3S_s}{2}\bigg)^{2^s} d_s^{2^s-1} B_{2^{s+1}} \sum_{\ell=0}^{M_s} \frac{1}{2^{2^s \ell} q_s(\ell)^{2^{s}-1}} \bigg).
\end{align*}
This observation completes the proof.
\end{proof}


The random variable~$\hat{H}_t(x)$ represents an estimator for the conditional risk mapping~$F_t^T(x)$ introduced in Definition~\ref{def:risk-mapping}. Lemma~\ref{lemma:bias_bound_rtmlmc_H_t} below bounds the conditional bias of the estimator~$\hat{H}_t(x)$ for any~$t\in[T]$. As~$\EE[\hat{F}(x)]=\EE[\hat{H}_1(x)]$ for any $x\in\cX$, Lemma~\ref{lemma:bias_bound_rtmlmc} follows immediately from Lemma~\ref{lemma:bias_bound_rtmlmc_H_t}. Its proof is thus omitted for brevity. We point out that the stronger result of Lemma~\ref{lemma:bias_bound_rtmlmc_H_t} will also be used in Section~\ref{sec:mcco_optimization}.

\begin{lemma}
\label{lemma:bias_bound_rtmlmc_H_t} 
If Assumptions~\ref{assump:general}, \ref{assump:lipschitz} and~\ref{assump:mu_bar} hold and if $t\in[T]$, then we have
\begin{align*}
    \big\|\EE_{t-1}[\hat{H}_t(x)] - F_t^T(x)\big\|_2 \leq \begin{cases}
        \sum_{s=t}^{T-1}\frac{L_{[t:s]}\;\EE_{t-1}[\mu_{s+1}^2(x)]^\frac{1}{2}}{2^{M_s/2}}&\text{if Assumption~\ref{assump:smooth} does not hold,}\\
         \sum_{s=t}^{T-1}\frac{L_{[t:s-1]}S_s\;\EE_{t-1}[\mu_{s+1}^{2}(x)]}{2^{M_s + 1}} &\text{if Assumption~\ref{assump:smooth} holds.}
    \end{cases}
\end{align*}
\end{lemma}

\begin{proof}
We simplify notation by suppressing the dependence on $x$ for all functions of $x$ throughout the proof. Suppose first that only Assumptions~\ref{assump:general} and~\ref{assump:lipschitz} hold. By the construction of the conditional risk mappings $F_s^t$ in Definition~\ref{def:risk-mapping} and the estimators $\hat H_s$ in Definition~\ref{def:MCCO_RTMLMC}, we have
\begin{align*}
    F_t^T = F_t^{T-1}\big(\EE_{T-1}[\hat{H}_{T}]\big) \quad \text{and}\quad \EE_{t-1}[\hat{H}_{t}]=F_t^{t-1}\big(\EE_{t-1}[\hat{H}_{t}]\big).
\end{align*}
The triangle inequality thus implies that the squared conditional bias of $\hat H_t$ admits the estimate
\begin{align}
    \label{eq:rtmlmc_telescoping_sum}
    \big\| \EE_{t-1}[\hat{H}_{t}] - F_t^T \big\|_2 \leq \sum_{s=t}^{T-1}\big\| F_t^{s-1}\big(\EE_{s-1}[\hat{H}_s]\big) - F_t^s\big(\EE_s[\hat{H}_{s+1}]\big)\big\|_2.
\end{align}
To reformulate the resulting upper bound, observe that for any~$s \in [T-1]$, we have
\begin{equation}
    \label{eq:rtmlmc_bias}
    \begin{aligned}
    \EE_{s-1}[\hat{H}_s] & = \EE_{s-1} \left[ \frac{1}{q_{s}(\lambda_{s})}\left( \hat{h}_s^{\lambda_s} -\frac{1}{2}\hat{h}_s^{\lambda_s, {\rm e}}-\frac{1}{2}\hat{h}_s^{\lambda_s,{\rm o}} \right)\right] \\
    & = \sum_{\ell=0}^{M_s} \EE_{s-1} \left[ \hat{h}_s^{\ell} -\frac{1}{2}\hat{h}_s^{\ell, {\rm e}}-\frac{1}{2}\hat{h}_s^{\ell,{\rm o}} \right] = \EE_{s-1}[\hat{h}_s^{M_s}] 
    = \EE_{s-1}\big[f_s(\xi_s, \hat\EE_s^{M_s}[\hat{H}_{s+1}])\big],
\end{aligned}
\end{equation}
where the first equality follows from the definition of~$\hat H_s$ in~\eqref{eq:H_hat}, and the second uses the law of total expectation along with the independence of~$\lambda_s$ from all other random objects. The third and fourth equalities hold because~$\hat{h}_s^{\ell,{\rm e}}$ and~$\hat{h}_s^{\ell,{\rm o}}$ share the distribution of~$\hat{h}_s^{\ell-1}$ conditional on $\mathcal F_s$ for all~$\ell \in [M_s]$ and because~$\hat{h}_s^{0,{\rm e}} = \hat{h}_s^{0,{\rm o}} = 0$, respectively. Hence, the $s$-th term in the sum on the right hand side of~\eqref{eq:rtmlmc_telescoping_sum} satisfies
\begin{align*}
    & \big\| F_t^{s-1}\big(\EE_{s-1}[\hat{H}_s]\big) - F_t^s\big(\EE_s[\hat{H}_{s+1}]\big)\big\|_2\\
    &\quad = \big\| F_t^{s-1}\big(\EE_{s-1}[f_s(\xi_s,\hat\EE_s^{M_s}[\hat{H}_{s+1}])]\big) - F_t^{s-1}\big(\EE_{s-1}[f_s(\xi_s,\EE_s[\hat{H}_{s+1}])]\big)\big\|_2\\
    &\quad \leq L_{[t:s]} \; \EE_{t-1}\big[ \big\|\hat\EE_s^{M_s}[\hat{H}_{s+1}] - \EE_s[\hat{H}_{s+1}] \big\|_2 \big] \\
    &\quad \leq L_{[t:s]} \; \EE_{t-1}\big[ \big\|\hat\EE_s^{M_s}[\hat{H}_{s+1}] - \EE_s[\hat{H}_{s+1}] \big\|_2^2 \big]^\frac{1}{2} \\
    & \quad \leq \frac{L_{[t:s]}}{2^{M_s/2}} \; \EE_{t-1}\big[\big\|\hat{H}_{s+1} - \EE_s[\hat{H}_{s+1}]\big\|_2^2\big]^{\frac{1}{2}} \leq \frac{L_{[t:s]}}{2^{M_s/2}} \; \EE_{t-1}\big[ \mu^2_{s+1} \big]^\frac{1}{2},
\end{align*}
where the first inequality follows from Lemma~\ref{lem:lipschitz}, which guarantees that $F_t^{s-1}(z_{s-1})$ is $L_{[t:s-1]}$-Lipschitz continuous in~$z_{s-1}$, and from Assumption~\ref{assump:lipschitz}, which ensures that $f_s(\xi_s, x_s)$ is $L_s$-Lipschitz continuous in~$x_s$. The second and third inequalities follow from Jensen's inequality and the fact that the variance of the average of i.i.d.\ random variables equals the variance of one of these random variables divided by the sample size. The last inequality holds because the variance of any random variable is bounded by its second moment. The claim then follows by substituting the emerging bound into~\eqref{eq:rtmlmc_telescoping_sum}.

If Assumption~\ref{assump:smooth} holds, too,  then the $s$-th term on the right hand side of~\eqref{eq:rtmlmc_telescoping_sum} satisfies
\begin{equation}
\label{eq:telescopting-smooth-rtmlmc}
\begin{aligned}
    & \big\| F_t^{s-1}\big(\EE_{s-1}[\hat{H}_s]\big) - F_t^s\big(\EE_s[\hat{H}_{s+1}]\big)\big\|_2 \\
    & \quad =  \big\| F_t^{s-1}\big(\EE_{s-1}[f_s(\xi_s,\hat\EE_s^{M_s}[\hat{H}_{s+1}])]\big) - F_t^{s-1}\big(\EE_{s-1}[f_s(\xi_s,\EE_s[\hat{H}_{s+1}])]\big)\big\|_2\\
    & \quad \leq L_{[t:s-1]}\; \EE_{t-1}\big[\big\|\EE_{s-1}\big[f_s(\xi_s,\hat\EE_s^{M_s}[\hat{H}_{s+1}]) - f_s(\xi_s,\EE_s[\hat{H}_{s+1}])\big] \big\|_2\big].
\end{aligned}
\end{equation}
As $ \nabla_{x_s} f_s(\xi_s, \EE_{s} [\hat{H}_{s+1}])$ is $\cF_s$-measurable and as $\EE_{s}[\hat\EE_s^{M_s}[\cdot]] = \EE_{s}[\cdot]$, we may conclude that
\begin{align*}
    \EE_{s-1}\big[\nabla_{x_s}f_s(\xi_s,\EE_{s}[\hat{H}_{s+1}])^\top (\hat\EE_s^{M_s}[\hat{H}_{s+1}] - \EE_{s}[\hat{H}_{s+1}])\big] =0.
\end{align*}
Hence, the last expression in~\eqref{eq:telescopting-smooth-rtmlmc} equals
\begin{align*}
    & L_{[t:s-1]}\;\EE_{t-1}\big[\big\|\EE_{s-1}\big[f_s(\xi_s, \hat\EE_s^{M_s}[\hat{H}_{s+1}]) - f_s(\xi_s,\EE_s[\hat{H}_{s+1}]) \\
    &\qquad\;\qquad - \nabla_{x_s}f_s(\xi_s,\EE_s[\hat{H}_{s+1}])^\top (\hat\EE_s^{M_s}[\hat{H}_{s+1}] - \EE_s[\hat{H}_{s+1}]) \big]\big\|_2\big]\\
    & \quad \leq \frac{L_{[t:s-1]}S_s}{2}\;\EE_{t-1}\big[ \big\|\hat\EE_s^{M_s}[\hat{H}_{s+1}] - \EE_s[\hat{H}_{s+1}] \big\|_2^2 \big]\\
    & \quad \leq \frac{L_{[t:s-1]}S_s}{2^{M_s+1}}\; \EE_{t-1}\big[ \big\|\hat{H}_{s+1} - \EE_s[\hat{H}_{s+1}] \big\|_2^2 \big]\leq \frac{L_{[t:s-1]}S_s}{2^{M_s+1}}\; \EE_{t-1}\big[\mu_{s+1}^{2}\big].
\end{align*}
Here, the first inequality follows by applying Jensen's inequality to move the conditional expectation $\EE_{s-1}[\cdot]$ outside the norm, and then using the smoothness inequality~\eqref{eq:inequality_for_smooth_funs}, which holds by Assumption~\ref{assump:smooth}. The second inequality holds because the variance of an average of i.i.d.\ samples equals the variance of a single sample divided by the sample size. The third inequality follows from the fact that the variance of any random variable is bounded by its second moment, together with the definition of $\mu_{s+1}^{2}$ in~\eqref{eq:mu}. Substituting the emerging bound into~\eqref{eq:rtmlmc_telescoping_sum} yields the desired result.
\end{proof}

\begin{proof}[Proof of Lemma~\ref{lemma:bound_on_cost_rtmlmc}] 

The MLMC estimator of Definition~\ref{def:MCCO_RTMLMC} is built on~$n_1$ independent scenario trees, and thus, we have $\EE[\CC(\hat{F}(x))]=n_1\EE[\CC(\hat{H}_{1}(x))]$. Each tree accommodates~$T$ stages, $2^{\lambda_t}$ branches per stage-$t$ node for every~$t\in[T-1]$ and one sample per node. As each log-branching factor $\lambda_t$, $t\in[T-1]$, is independent of all other random objects, the expected number of scenarios per tree is given by
\begin{align*}
   \EE\big[\CC(\hat{H}_{1}(x))\big] =  \prod_{t=1}^{T-1}\sum_{\ell=0}^{M_t} q_t(\ell)2^{\ell}.
\end{align*}
Hence, the claim follows. 
\end{proof}

The following lemma shows that if the integrands fail to be smooth, then the rate parameters in Assumption~\ref{assump:rate-parameters-rtmlmc-funval} are optimal in a precise sense, as detailed in Remark~\ref{remark:optimal-prob}.

\begin{lemma}\label{lemma:optimal_b}
    If Assumptions~\ref{assump:general}, \ref{assump:lipschitz} and~\ref{assump:mu_bar} hold, then the choice $r = (r_1, \ldots, r_{T-1})$ with $r_t = 1/2$ for all $t \in [T-1]$ minimizes the approximate sampling cost~$c_{\mathrm{ns}}(r)$ over all~$r\in(0, 1)^{T-1}$.
\end{lemma}

\begin{proof}[Proof of Lemma~\ref{lemma:optimal_b}] 
We adopt a similar proof strategy as in~\cite[\S~3]{blanchet2015unbiased}, which proposes an {\em untruncated} MLMC estimator for a subclass of~\eqref{problem:MCCE} problems with $T = 2$. Using the definition of~$c_{\mathrm{ns}}(r)$ and the uniform upper bound on $\mu_t^2(x)$ for nonsmooth integrands established in Lemma~\ref{lemma:bound_on_variance_rtmlmc}, we obtain
\[
    c_{\mathrm{ns}}(r) = \frac{2}{\epsilon^2} \left(\prod_{t=1}^{T-1} \sum_{\ell=0}^{M_t} \frac{1}{q_t(\ell)2^{\ell}} \right) \left(\prod_{t=1}^{T-1} \sum_{\ell=0}^{M_t} q_t(\ell)2^{\ell}\right) = \frac{2}{\epsilon^2} \prod_{t=1}^{T-1} \left(\sum_{\ell=0}^{M_t} \frac{1}{q_t(\ell)2^{\ell}} \right) \left( \sum_{\ell=0}^{M_t} q_t(\ell)2^{\ell}\right).
\]
As $q_t(\ell) = r_t(1-r_t)^\ell/(1-(1-r_t)^{M_t+1})$ is independent of~$r_{t'}$ for all $t'\neq t$, we can find the optimal $r_t$ by minimizing the logarithm of the $t$-th term in the resulting product over $r_t\in(0,1)$. Thus, we must solve
\begin{align*}
    \min_{r_t\in(0,1) }\; \log\bigg(\sum_{\ell=0}^{M_t}\frac{1}{2^{\ell}q_t(\ell)} \bigg) +\log\bigg(\sum_{\ell=0}^{M_t}2^{\ell}q_t(\ell)  \bigg),
\end{align*}
where $q_t(\ell)$ is interpreted as a function of~$r_t$.  This minimization problem can be simplified by representing the rate as $r_t=1-2^{-b_t}$ using an auxiliary decision variable $b_t\in (0,\infty)$. Under this re-parametrization we have $q_t(\ell)=2^{-b_t\ell}/Z_t$, and the normalization constant amounts to $Z_t=(1-2^{-b_t(M_t+1)}) / (1-2^{-b_t})$. Replacing~$q_t(\ell)$ with this formula allows us to reformulate the above minimization problem equivalently as
\begin{align*}
    \min_{b_t \in (0,\infty)}\; \log\bigg(\sum_{\ell=0}^{M_t}2^{(b_t-1)\ell} \bigg) +\log\bigg(\sum_{\ell=0}^{M_t}2^{(1-b_t)\ell}  \bigg).
\end{align*}
The objective function of the re-parametrized optimization problem is convex in~$b_t$ because logarithms of sums of exponential functions are known to be convex. An optimal~$b_t$ can thus be found from the problem's first-order optimality condition. Indeed, the gradient of the objective function vanishes if and only if
\begin{align*}
    \frac{\sum_{\ell=0}^{M_t} \ell\,2^{(b_t-1)\ell}}{\sum_{\ell'=0}^{M_t} 2^{(b_t-1)\ell'}}= \frac{\sum_{\ell=0}^{M_t} \ell\,2^{(1-b_t)\ell}}{\sum_{\ell'=0}^{M_t} 2^{(1-b_t)\ell'}}.
\end{align*}
Clearly, the above equation is solved by $b_t=1$ for every $t\in[T-1]$, which implies that $c_{\rm ns}(r)$ is minimized by setting $r_t=1-2^{-b_t}=1/2$ for every $t\in[T-1]$. Hence, the claim follows.
\end{proof}

\begin{proof}[Proof of Theorem~\ref{thm:rtmlmc_sample_complexity_mse}]
To guarantee a root mean squared error of at most~$\epsilon$, we require both the bias and the standard deviation of the MLMC estimator to be bounded by~$\epsilon/\sqrt{2}$ for any~$x\in\cX$. Suppose first that Assumption~\ref{assump:smooth} does not hold ({\em i.e.}, the integrands are only Lipschitz continuous). Hence, the bias satisfies
\begin{align*}
  \big|\EE[\hat{F}(x)] - F(x)\big| \leq \sum_{t=1}^{T-1}\frac{L_{[t]}\EE[\mu_{t+1}^2(x)]^\frac{1}{2}}{2^{M_t/2}} 
\end{align*}
thanks to Lemma~\ref{lemma:bias_bound_rtmlmc}. Lemma~\ref{lemma:bound_on_variance_rtmlmc} further implies that
\[
    \EE[\mu_{t+1}^2(x)]^\frac{1}{2} \leq \sqrt{C_{t+1}} \prod_{s=t+1}^{T-1} \bigg( \sum_{\ell=0}^{M_s}\frac{1}{2^{\ell} q_s(\ell)} \bigg)^{\frac{1}{2}} = \sqrt{C_{t+1}} \prod_{s=t+1}^{T-1} \sqrt{Z_s} \sqrt{M_s+1} \quad \forall t\in[T-1],
\]
where the equality holds thanks to Assumption~\ref{assump:rate-parameters-rtmlmc-funval}, which implies that~$r_t=1/2$. Indeed, for this choice of the rate parameter~$r_t$ we have $q_t(\ell) = 2^{-\ell}/Z_t$, and the corresponding normalization constant simplifies to~$Z_t=2(1-2^{-(M_t+1})$. Combining the above estimates then yields
\begin{align*}
  \big|\EE[\hat{F}(x)] - F(x)\big| \leq \sum_{t=1}^{T-1} \frac{L_{[t]} \sqrt{C_{t+1}}}{2^{M_t/2}} \prod_{s=t+1}^{T-1} \sqrt{Z_s} \sqrt{M_s+1}.
\end{align*}
Recalling from the theorem statement that the $t$-th truncation point is given by
\begin{align*}
     M_{t} = \bigg\lceil 2\log_2\Big( \textstyle \sqrt{2}L_{[t]}\sqrt{C_{t+1}} \, \Big( \prod_{s=t+1}^{T-1} \sqrt{Z_s} \sqrt{M_s+1} \Big) (T-1)/\epsilon \Big) \bigg\rceil
\end{align*}
for every $t \in [T-1]$, we finally obtain the desired estimate $|\EE[\hat{F}(x)] - F(x)| \leq \epsilon/\sqrt{2}$ for every $x\in\cX$. Similarly, the variance of the MLMC estimator satisfies
\[
    \VV\big(\hat{F}(x) \big) = \VV\big(\hat{\EE}[\hat{H}_{1}(x)]\big) = \frac{1}{n_1}\VV\big(\hat{H}_{1}(x)\big) \leq \frac{\mu_1^2(x)}{n_1}\leq \frac{\epsilon^2}{2}
\]
for every $x\in\cX$, where the two equalities exploit the definitions of~$\hat F(x)$ and $\hat \EE[\cdot]$, respectively. The first inequality holds because the variance of any random variable is bounded by its second moment, and the second inequality follows from our choice of~$n_1=\lceil \sup_{x\in\cX} 2\mu_1^2/\epsilon^2 \rceil$. In summary, the above reasoning implies that the root mean squared error of $\hat F(x)$ is bounded by~$\epsilon$ uniformly across all~$x\in\cX$.

Next, Lemma~\ref{lemma:bound_on_cost_rtmlmc} implies that the expected sampling cost of the MLMC estimator amounts to
\begin{align*}
    \EE[\CC(\hat{F}(x))] = n_1 \prod_{t=1}^{T-1} \sum_{\ell=0}^{M_t} q_t(\ell)2^{\ell} = n_1 \prod_{t=1}^{T-1} Z_t^{-1} (M_t+1), 
\end{align*}
where the second equality follows from 
Assumption~\ref{assump:rate-parameters-rtmlmc-funval}. It remains to analyze how the resulting expression scales with respect to~$1/\epsilon$. The recursive definition of the truncation points readily implies that~$M_t = \cO(\log(\epsilon^{-1}))$ for all $t \in [T-1]$. By Lemma~\ref{lemma:bound_on_variance_rtmlmc} and the definition of~$n_1$, we further have
\begin{align*}
    n_1\leq \frac{2C_1}{\epsilon^2} \prod_{s=1}^{T-1} \bigg(\sum_{\ell=0}^{M_s}\frac{1}{2^{\ell} q_s(\ell)} \bigg) +1 = \frac{2C_1}{\epsilon^2} \prod_{t=1}^{T-1}Z_t(M_t+1) +1 = \cO(\log(\epsilon^{-1})^{T-1}\epsilon^{-2}).
\end{align*}
Hence, the expected sampling cost of the MLMC estimator satisfies $\EE[\CC(\hat{F}(x))] \leq \cO\big(\log(\epsilon^{-1})^{2(T-1)}\epsilon^{-2}\big)$.

Suppose now that Assumption~\ref{assump:smooth} holds, too. In this case, Lemma~\ref{lemma:bias_bound_rtmlmc} and Jensen's inequality imply that the bias of $\hat F(x)$ satisfies 
\begin{align*}
  \big|\EE[\hat{F}(x)] - F(x)\big| \leq \sum_{t=1}^{T-1}\frac{L_{[t-1]}S_t\EE[\mu_{t+1}^{2}(x)]}{2^{M_t + 1}}  \leq \sum_{t=1}^{T-1}\frac{L_{[t-1]}S_t\EE[\mu_{t+1}^{2^{t+1}}(x)]^\frac{1}{2^t}}{2^{M_t + 1}}.
\end{align*}
Lemma~\ref{lemma:bound_on_variance_rtmlmc} and Assumption~\ref{assump:rate-parameters-rtmlmc-funval} further imply that
\begin{align*}
    \EE[\mu_{t+1}^{2^{t+1}}(x)]^\frac{1}{2^t} \leq \Big(D_{t+1} \prod_{s=t+1}^{T-1}\sum_{\ell=0}^{M_s} Z_s^{2^s-1}2^{-\frac{\ell}{2^s}}\Big)^\frac{1}{2^t}
    =D_{t+1}^{2^{-t}} \prod_{s=t+1}^{T-1} Z_s^{\frac{2^s-1}{2^t}} \bigg(\frac{1-2^{-\frac{M_s+1}{2^s}}}{1-2^{-\frac{1}{2^s}}}\bigg)^\frac{1}{2^t} \quad \forall t\in[T-1].
\end{align*}
Indeed, setting the $t$-th rate parameter to $r_t=1-2^{-1-2^{-t}}$ implies that $q_t(\ell)=2^{-(1+2^{-t})\ell}/Z_t$ with normalization constant $Z_t=(1-2^{-(1+2^{-t})(M_t+1)})/(1-2^{-1-2^{-t}})$. Combining these estimates then yields
\begin{align*}
    \big|\EE[\hat{F}(x)] - F(x)\big| \leq \sum_{t=1}^{T-1}\frac{L_{[t-1]}S_tD_{t+1}^{2^{-t}}}{2^{M_t + 1}} \prod_{s=t+1}^{T-1} Z_s^{\frac{2^s-1}{2^t}} \bigg(\frac{1-2^{-\frac{M_s+1}{2^s}}}{1-2^{-\frac{1}{2^s}}}\bigg)^\frac{1}{2^t}.
\end{align*}
Recall now from the theorem statement that the $t$-th truncation point is given by
\begin{align*}
     \textstyle M_{t}  = \bigg\lceil\log_2\bigg( \sqrt{2}L_{[t-1]}S_{t}D_{t+1}^{2^{-t}} \, \bigg( \prod_{s=t+1}^{T-1}Z_s^{\frac{2^s-1}{2^t}} \Big(\frac{1-2^{-\frac{M_s+1}{2^s}}}{1-2^{-\frac{1}{2^s}}} \Big)^\frac{1}{2^t} \bigg)  (T-1)/(2\epsilon)\bigg)\bigg\rceil
\end{align*}
for every $t\in[T-1]$. This choice then yields the desired bias estimate $|\EE[\hat F(x)]-F(x)|\leq \epsilon/\sqrt{2}$ for every $x\in\cX$. To bound the variance of the MLMC estimator, we may use the exact same reasoning as in the first part of the proof. Thus, we may conclude that $\VV(\hat{F}(x)) \leq \epsilon^2/2$ uniformly across all $x\in\cX$. 

In addition, we may use a similar reasoning as in the first part of the proof to conclude that
\begin{align*}
    \EE[\CC(\hat{F}(x))] = n_1 \prod_{t=1}^{T-1} \sum_{\ell=0}^{M_t}Z_t^{-1} 2^{-2^{-t}\ell} = n_1 \prod_{t=1}^{T-1} Z_t^{-1}\frac{1-2^{-\frac{M_{t}+1}{2^{t}}}}{1-2^{-\frac{1}{2^t}}}. 
\end{align*}
The recursive definition of the truncation points readily implies that~$M_t = \cO(\log(\epsilon^{-1}))$ for all $t \in [T-1]$. Furthermore, the definition of~$n_1$ implies via Lemma~\ref{lemma:bound_on_variance_rtmlmc} that
\begin{align*}
    n_1 \leq \frac{2D_1}{\epsilon^2}\prod_{t=1}^{T-1} \sum_{\ell=0}^{M_t} Z_t^{2^t-1}2^{-2^{-t}\ell}+1
    =\frac{2D_1}{\epsilon^2}\prod_{t=1}^{T-1}Z_t^{2^t-1}\frac{1-2^{-\frac{M_t+1}{2^t}}}{1-2^{-\frac{1}{2^t}}} + 1=\cO(\epsilon^{-2}).
\end{align*}
Hence, the expected sampling cost of the MLMC estimator satisfies $\EE[\CC(\hat{F}(x))] \leq \cO(\epsilon^{-2})$.
\end{proof}
\begin{proof}[Proof of Theorem~\ref{thm:rtmlmc_sample_complexity_high_probability}] 
    The condition on~$n_1$ ensures via Corollary~\ref{cor:high_prob_conv_rtmlmc} that $ \PP\left( F(\hat{x}^*) - F(x^*)  > \epsilon \right) \leq \beta$ for any minimizers~$x^*$ and~$\hat{x}^*$ of problems~\eqref{problem:MCCO} and~\eqref{prob:min_hatF(x)}, respectively. Hence, any minimizer of problem~\eqref{prob:min_hatF(x)} constitutes an $\epsilon$-optimal solution to~\eqref{problem:MCCO} with probability at least $1-\beta$. 
    
    By Lemma~\ref{lemma:bound_on_cost_rtmlmc}, the expected sampling cost of the MLMC estimator amounts to
    \[
        \EE[\CC(\hat{F}(x))] = n_1 \prod_{t=1}^{T-1} \sum_{\ell=0}^{M_t} q_t(\ell)2^{\ell}.
    \] 
    By using a similar reasoning as in the proof of Theorem~\ref{thm:rtmlmc_sample_complexity_mse}, one can show that the product term is of order~$\cO(\log(\epsilon^{-1})^{T-1})$ for Lipschitz continuous integrands and of order~$\cO(1)$ for smooth integrands. Moreover, the number~$n_1$ of scenario trees appearing in the statement of the corollary is of order~$\cO(\log(\epsilon^{-1})\epsilon^{-2})$. The desired scenario complexity bounds then follow by multiplying these estimates. 
\end{proof}


\section{Proofs of Section~\ref{sec:mcco_optimization}}

\begin{proof}[Proof of Lemma~\ref{lem:gradient-recursion}]
    We use backward induction on~$t$ to prove that $\nabla F_t^T(x)= G_t(x)$ for all $x\in \RR^d$ and $t\in[T]$. As for the base case corresponding to~$t=T$, fix any $x\in\RR^d$, and observe that
    \begin{align*}
        \nabla F_T^T(x) = \nabla \EE_{T-1}[f_T(\xi_T,x)]=\EE_{T-1}[\nabla_x f_T(\xi_T,x)]=G_T(x).
    \end{align*}
    Here, the first equality follows from the construction of $F_T^T(x)$ in Definition~\ref{def:risk-mapping}. Next, the gradient may be interchanged with the conditional expectation by the dominated convergence theorem, which applies thanks to Assumptions~\ref{assump:lipschitz} and~\ref{assump:smooth}. The resulting expression coincides with the definition of $G_T(x)$. 
    
    As for the induction step corresponding to any $t\in[T-1]$, assume that $\nabla F_{t+1}^T(x) = G_{t+1}(x)$ for all~$x\in\RR^d$. Next, fix an arbitrary $x\in\RR^d$, and observe that
    \begin{align*}
        \nabla F_t^T(x) =& \nabla_{x_t} \EE_{t-1}[f_t(\xi_t,F_{t+1}^T(x))]\\
        =& \EE_{t-1}\left[ \nabla F_{t+1}^T(x)  \nabla_{x_t} f_t(\xi_t,F_{t+1}^T(x))\right] \\
        =& \EE_{t-1}\left[ G_{t+1}(x) \nabla_{x_t} f_t(\xi_t,F_{t+1}^T(x))\right] = G_t(x).
    \end{align*}
    Here, the first equality follows from the definition of~$F_t^T$, while the second equality follows from the dominated convergence theorem and the chain rule. The third and the fourth equalities then exploit the induction hypothesis and the definition of~$G_t$, respectively. The claim ultimately follows because $F(x)=F_1^T(x)$.
\end{proof}

\begin{proof}[Proof of Lemma~\ref{lemma:bound_on_variance_rumlmc_grad}] 

Throughout the proof we fix~$x$, and for simplicity we suppress the dependence on~$x$ in all functions. We then find an upper bound on~$\nu_t^{2^t}$ and a recursive formula for~$E_t$ by backward induction on~$t$. As for the base case corresponding to~$t=T$, the definition of~$\hat{G}_T$ and Assumption~\ref{assump:nu_bar} imply that
\begin{align*}
    \nu_T^{2^T} = \EE_{T-1}\big[ \|\hat{G}_T\|_2^{2^T} \big] \leq \overline\nu_T^{2^T} \eqqcolon E_T<\infty.
\end{align*}
As for the induction step corresponding to any~$t$ with~$t+1\in[T]$, we introduce the auxiliary variables
\begin{align}
\label{eq:S-definition}
    \cS_s \coloneqq \sum_{\ell=0}^{M_s}\frac{1}{2^{(2^{s-1})(\rho_s+1)\ell} q_s(\ell)^{2^{s}-1}} \quad\text{and} \quad\cS_s^\prime\coloneqq\sum_{\ell=0}^{M_s}\frac{1}{2^{2^s(\rho_{s-1}+1) \ell} q_s(\ell)^{2^{s}(\rho_{s-1}+1)-1}} \quad \forall s\in[T-1],
\end{align}
which allows us to express the induction hypothesis compactly as
\begin{equation}
  \begin{aligned}\label{eq:moment_lemma_inductive_hypothesis_grad}
    \nu_{t+1}^{2^{t+1}} &= \EE_t\big[\|\hat{G}_{t+1}\|_{2}^{2^{t+1}}\big]\leq E_{t+1}\prod_{s=t+1}^{T-1} \max\big\{ \cS_s, \cS_s^\prime \big\}.
\end{aligned}  
\end{equation}
In order to bound $\nu_t^{2^t}$, we use the law of total expectation and the definition of $\hat{G}_t$ in \eqref{eq:G_hat} to obtain
\begin{equation}
\begin{aligned} \label{eq:telescoping_step_rumlmc}
    \nu_t^{2^t}=\EE_{t-1} \big[ \| \hat{G}_t \|^{2^t}_{2} \big] 
    = & \sum_{\ell=0}^{M_t} q_t(\ell) \, \EE_{t-1} \bigg[ \left. \frac{1}{q_t(\lambda_t)^{2^t}} \Big\| \hat{g}_t^{\lambda_t}-\frac{1}{2}\hat{g}_t^{\lambda_t, {\rm e}}-\frac{1}{2}\hat{g}_t^{\lambda_t, {\rm o}} \Big\|^{2^t}_{2} \right| \lambda_t=\ell \bigg]\\
    = & \sum_{\ell=0}^{M_t} \frac{1}{q_t(\ell)^{2^t-1}} \, \EE_{t-1} \bigg[ \Big\| \hat{g}_t^{\ell}-\frac{1}{2}\hat{g}_t^{\ell, {\rm e}}-\frac{1}{2}\hat{g}_t^{\ell, {\rm o}} \Big\|^{2^t}_{2} \bigg].
\end{aligned}
\end{equation}
In the second line we have eliminated the condition on the event $\lambda_t=\ell$, which is allowed because $\lambda_t$ is independent of all other random objects. Next, we define $\Delta^\ell \coloneqq \hat{g}_t^{\ell}-\frac{1}{2}\hat{g}_t^{\ell, {\rm e}} - \frac{1}{2} \hat{g}_t^{\ell, {\rm o}}$ and
\begin{align*}
    A_t^\ell \coloneqq \nabla_{x_t} f_t \Big(\xi_t, \hat\EE_t^{\ell}\big[\hat{H}_{t+1}\big]\Big) 
    -\frac{1}{2}\nabla_{x_t} f_t \Big(\xi_t, \hat\EE_t^{\ell, {\rm e}}\big[\hat{H}_{t+1}\big]\Big) -\frac{1}{2}\nabla_{x_t} f_t \Big(\xi_t, \hat\EE_t^{\ell, {\rm o}}\big[\hat{H}_{t+1}\big]\Big)
\end{align*}
to avoid clutter. By the construction of $\hat g_t^\ell$, $\hat{g}_t^{\ell, {\rm e}}$ and $\hat{g}_t^{\ell, {\rm o}}$ in Definition~\ref{def:rumlmc-grad}, we thus have
\begin{align*}
    \Delta^\ell 
    & = \hat\EE_t^{\ell}\big[\hat{G}_{t+1}\big]  \nabla_{x_t} f_t \big(\xi_t, \hat\EE_t^{\ell}\big[\hat{H}_{t+1}\big]\big)-\frac{1}{2} \hat\EE_t^{\ell, {\rm e}}\big[\hat{G}_{t+1}\big]  \nabla_{x_t} f_t \big(\xi_t, \hat\EE_t^{\ell, {\rm e}}\big[\hat{H}_{t+1}\big]\big) \\
    & \quad -\frac{1}{2} \hat\EE_t^{\ell, {\rm o}}\big[\hat{G}_{t+1}\big] \nabla_{x_t} f_t \big(\xi_t, \hat\EE_t^{\ell, {\rm o}}\big[\hat{H}_{t+1}\big]\big) \\ 
    & = \Big( \hat\EE_t^{\ell}\big[\hat{G}_{t+1}\big] -\EE_t[\hat{G}_{t+1}] \Big) \Big[\nabla_{x_t} f_t \big(\xi_t, \hat\EE_t^{\ell}\big[\hat{H}_{t+1}\big]\big) - \nabla_{x_t} f_t\big(\xi_t, \EE_t[\hat{H}_{t+1}] \big)\Big] \\
    & \quad -\frac{1}{2}\Big( \hat\EE_t^{\ell, {\rm e}}\big[\hat{G}_{t+1}\big] -\EE_t[\hat{G}_{t+1}] \Big) \Big[\nabla_{x_t} f_t \big(\xi_t, \hat\EE_t^{\ell, {\rm e}}\big[\hat{H}_{t+1}\big]\big) - \nabla_{x_t} f_t\big(\xi_t, \EE_t[\hat{H}_{t+1}] \big)\Big] \\
    & \quad - \frac{1}{2}\Big( \hat\EE_t^{\ell, {\rm o}}\big[\hat{G}_{t+1}\big] -\EE_t[\hat{G}_{t+1}] \Big) \Big[ \nabla_{x_t} f_t \big(\xi_t, \hat\EE_t^{\ell, {\rm o}}\big[\hat{H}_{t+1}\big]\big) - \nabla_{x_t} f_t\big(\xi_t, \EE_t[\hat{H}_{t+1}] \big)\Big] +\EE_t[\hat{G}_{t+1}] \, A_t^\ell,
\end{align*}
where the second equality is obtained by adding terms that sum to~$0$ and using the antithetic sample average identity~\eqref{eq:even-odd-averages-identity}. Next, we observe that Jensen's inequality yields $|\sum_{i=1}^4 a_i |^p \leq 4^{p-1} \sum_{i=1}^4  |a_i |^p$ for all $a_1,\ldots, a_4\in\RR$ and $p\geq 1$. Hence, the $(p=2^t)$-th conditional moment of~$\Delta^\ell$ obeys the estimate
\begin{align*}
    & \EE_{t-1}\big[\|\Delta^\ell \|_{2}^{2^t}\big]\\
    & \leq {\textstyle 4^{2^t-1}} \, \EE_{t-1}\Big[\Big\|\Big( \hat\EE_t^{\ell}\big[\hat{G}_{t+1}\big] -\EE_t[\hat{G}_{t+1}] \Big) \Big[\nabla_{x_t} f_t \big(\xi_t, \hat\EE_t^{\ell}\big[\hat{H}_{t+1}\big]\big) - \nabla_{x_t} f_t\big(\xi_t, \EE_t[\hat{H}_{t+1}] \big)\Big] \Big\|_{2}^{2^t}\Big]\\
    & \quad +{\textstyle \frac{4^{2^t-1}}{2^{2^t}}} \, \EE_{t-1}\Big[\Big\|\Big( \hat\EE_t^{\ell, {\rm e}}\big[\hat{G}_{t+1}\big] -\EE_t[\hat{G}_{t+1}] \Big) \Big[\nabla_{x_t} f_t \big(\xi_t, \hat\EE_t^{\ell, {\rm e}}\big[\hat{H}_{t+1}\big]\big) - \nabla_{x_t} f_t\big(\xi_t, \EE_t[\hat{H}_{t+1}] \big)\Big] \Big\|_{2}^{2^t}\Big]\\
    & \quad +{\textstyle \frac{4^{2^t-1}}{2^{2^t}}} \, \EE_{t-1}\Big[\Big\|\Big( \hat\EE_t^{\ell, {\rm o}}\big[\hat{G}_{t+1}\big] -\EE_t[\hat{G}_{t+1}] \Big) \Big[ \nabla_{x_t} f_t \big(\xi_t, \hat\EE_t^{\ell, {\rm o}}\big[\hat{H}_{t+1}\big]\big) - \nabla_{x_t} f_t\big(\xi_t, \EE_t[\hat{H}_{t+1}] \big)\Big]\Big\|_{2}^{2^t}\Big]\\
    & \quad +{\textstyle 4^{2^t-1}} \, \EE_{t-1}\Big[\big\|\EE_t[\hat{G}_{t+1}] \; A^\ell_t\big\|_{2}^{2^t}\Big].
\end{align*}
To bound the first three terms above, we recall that $2ab\leq a^2 + b^2$ for all $a,b\in\RR$. Hence, we obtain
\begin{equation}
\label{eq:2t-th-moment-delta}
 \begin{aligned}
    & \EE_{t-1}\big[\|\Delta^\ell \|_{2}^{2^t}\big]\\
    & \leq {\textstyle \frac{4^{2^t-1}}{2}} \, \EE_{t-1}\Big[\Big\|\hat\EE_t^{\ell}\big[\hat{G}_{t+1}\big] \! - \! \EE_t[\hat{G}_{t+1}]  \Big\|_{2}^{2^{t+1}} \hspace{-2mm} + \Big\|\nabla_{x_t} f_t \big(\xi_t, \hat\EE_t^{\ell}\big[\hat{H}_{t+1}\big]\big) \! - \!  \nabla_{x_t} f_t\big(\xi_t, \EE_t[\hat{H}_{t+1}] \big) \Big\|_{2}^{2^{t+1}}\Big]\\
    &+{\textstyle \frac{4^{2^t-1}}{2^{2^t+1}}} \, \EE_{t-1}\Big[\Big\| \hat\EE_t^{\ell, {\rm e}}\big[\hat{G}_{t+1}\big] \! - \! \EE_t[\hat{G}_{t+1}] \Big\|_{2}^{2^{t+1}} \hspace{-2mm} + \Big\|\nabla_{x_t} f_t \big(\xi_t, \hat\EE_t^{\ell, {\rm e}}\big[\hat{H}_{t+1}\big]\big) \! - \! \nabla_{x_t} f_t\big(\xi_t, \EE_t[\hat{H}_{t+1}] \big)\Big\|_{2}^{2^{t+1}}\Big]\\
    &+ {\textstyle \frac{4^{2^t-1}}{2^{2^t+1}}} \, \EE_{t-1}\Big[\Big\| \hat\EE_t^{\ell, {\rm o}}\big[\hat{G}_{t+1}\big] \! - \! \EE_t[\hat{G}_{t+1}] \Big\|_{2}^{2^{t+1}} \hspace{-2mm} + \Big\| \nabla_{x_t} f_t \big(\xi_t, \hat\EE_t^{\ell, {\rm o}}\big[\hat{H}_{t+1}\big]\big) \! - \! \nabla_{x_t} f_t\big(\xi_t, \EE_t[\hat{H}_{t+1}] \big)\Big\|_{2}^{2^{t+1}}\Big]\\
    &+4^{2^t-1}\, \EE_{t-1}\Big[\big\|\EE_t[\hat{G}_{t+1}] \, A^\ell_t\big\|_{2}^{2^t}\Big].
 \end{aligned}
\end{equation}
The remainder of the proof proceeds in three steps, where we analyze separately the terms appearing in~\eqref{eq:2t-th-moment-delta}. In Step~(i), we bound the $2^t$-th conditional moment of $\mathbb{E}_t[\hat{G}_{t+1}] \, A_t^\ell$, which appears on the last line of~\eqref{eq:2t-th-moment-delta}. In Step~(ii), we bound the $2^{t+1}$-th conditional moments of the three gradient-difference terms involving~$\hat{H}_{t+1}$. Finally, in Step~(iii), we bound the $2^{t+1}$-th conditional moments of the three value-difference terms involving~$\hat{G}_{t+1}$. To complete the induction step, we then substitute all obtained bounds into~\eqref{eq:2t-th-moment-delta} and~\eqref{eq:telescoping_step_rumlmc}. %

\paragraph{Step (i)} To bound the $2^t$-th conditional moment of $\EE_t[\hat{G}_{t+1}] \, A^\ell_t$, we first observe that
\begin{align}
    \label{eq:last-term}
    \EE_{t-1}\Big[\big\|\EE_t[\hat{G}_{t+1}] \, A^\ell_t\big\|_{2}^{2^t}\Big]
   \leq \EE_{t-1}\Big[\big\|\EE_t[\hat{G}_{t+1}] \big\|_{2}^{2^t} \big\| A^\ell_t\big\|_{2}^{2^t}\Big],
\end{align}
where the inequality follows from the submultiplicativity of the Frobenius norm established in Lemma~\ref{lemma:cauchy-schwarz-matrix}. We emphasize that entry-wise matrix $p$-norms are not submultiplicative for $p>2$. This is our reason for using the Frobenius norm (corresponding to $p=2$) in the definition of the $p$-th conditional moment of $\hat{G}_t(x)$ in~\eqref{eq:nu}. 
Next, we bound $\| A^\ell_t\|_{2}^{2^t}$. Jensen's inequality readily implies that
\begin{equation}
    \label{eq:A^ell_t-norm-bound}
    \| A^\ell_t\|_{2}^{2^t}= \bigg(\sum_{i=1}^{d_{t-1}} \big\| A^\ell_t  e_i \big\|_{2}^2 \bigg)^{2^{t-1}}\leq (d_{t-1})^{2^{t-1}-1}\sum_{i=1}^{d_{t-1}} \| A^\ell_t e_i\|_{2}^{2^{t}},
\end{equation}
where $e_i$ stands for the $i$-th standard basis vector in~$\RR^{d_{t-1}}$. By the definition of $A_t^\ell$, we can rewrite $A_t^\ell e_i$ as
\begin{align*}
   & \nabla_{x_t} f_{t,i} \big(\xi_t, \hat\EE_t^{\ell}\big[\hat{H}_{t+1}\big]\big) 
    \! - \! \frac{1}{2}\nabla_{x_t} f_{t,i} \big(\xi_t, \hat\EE_t^{\ell, {\rm e}}\big[\hat{H}_{t+1}\big]\big) 
    \! - \! \frac{1}{2}\nabla_{x_t} f_{t,i} \big(\xi_t, \hat\EE_t^{\ell, {\rm o}}\big[\hat{H}_{t+1}\big]\big)\\
    = & \nabla_{x_t} f_{t,i} \big(\xi_t, \hat\EE_t^{\ell}\big[\hat{H}_{t+1}\big]\big) \! - \! \nabla_{x_t} f_{t,i} \big(\xi_t, \EE_t[\hat{H}_{t+1}]\big) \! - \! \nabla^2_{x_t} f_{t,i} \big(\xi_t,\EE_t[\hat{H}_{t+1}]\big) \big(\hat\EE_t^{\ell}\big[\hat{H}_{t+1}\big]\! - \!\EE_t[\hat{H}_{t+1}]\big)\\
    &+ \! \frac{1}{2}\Big[\nabla_{x_t} f_{t,i} \big(\xi_t, \EE_t[\hat{H}_{t+1}]\big) 
    \! - \! \nabla_{x_t} f_{t,i} \big(\xi_t, \hat\EE_t^{\ell, {\rm e}}\big[\hat{H}_{t+1}\big]\big) \! - \! \nabla^2_{x_t} f_{t,i} \big(\xi_t,\EE_t[\hat{H}_{t+1}]\big) \big(\hat\EE_t^{\ell, {\rm e}}\big[\hat{H}_{t+1}\big]\! - \!\EE_t[\hat{H}_{t+1}]\big) \Big]\\
    & + \! \frac{1}{2}\Big[\nabla_{x_t} f_{t,i} \big(\xi_t, \EE_t[\hat{H}_{t+1}]\big)  
    \! - \! \nabla_{x_t} f_{t,i} \big(\xi_t, \hat\EE_t^{\ell, {\rm o}}\big[\hat{H}_{t+1}\big]\big) \! - \! \nabla^2_{x_t} f_{t,i} \big(\xi_t,\EE_t[\hat{H}_{t+1}]\big) \big(\hat\EE_t^{\ell, {\rm o}}\big[\hat{H}_{t+1}\big]\! - \!\EE_t[\hat{H}_{t+1}]\big)\Big],
\end{align*}
where the equality follows from the antithetic sample average identity~\eqref{eq:even-odd-averages-identity}. The norm of $A_t^\ell e_i$ thus satisfies
\begin{align*}
   \|A_t^\ell e_i\|_{2} 
    \leq &\; \big\|\nabla_{x_t} f_{t,i} \big(\xi_t, \hat\EE_t^{\ell}\big[\hat{H}_{t+1}\big]\big) - \nabla_{x_t} f_{t,i} \big(\xi_t, \EE_t[\hat{H}_{t+1}]\big) \\
    &\qquad\;- \nabla^2_{x_t} f_{t,i} \big(\xi_t,\EE_t[\hat{H}_{t+1}]\big)\; \big(\hat\EE_t^{\ell}\big[\hat{H}_{t+1}\big]-\EE_t[\hat{H}_{t+1}]\big)\big\|_2\\
    & +\frac{1}{2}\big\|\nabla_{x_t} f_{t,i} \big(\xi_t, \EE_t[\hat{H}_{t+1}]\big) 
    - \nabla_{x_t} f_{t,i} \big(\xi_t, \hat\EE_t^{\ell, {\rm e}}\big[\hat{H}_{t+1}\big]\big) \\
    &\qquad\;  - \nabla^2_{x_t} f_{t,i} \big(\xi_t,\EE_t[\hat{H}_{t+1}]\big)\; \big(\hat\EE_t^{\ell, {\rm e}}\big[\hat{H}_{t+1}\big]-\EE_t[\hat{H}_{t+1}]\big) \big\|_2\\
    & +\frac{1}{2}\big\|\nabla_{x_t} f_{t,i} \big(\xi_t, \EE_t[\hat{H}_{t+1}]\big)  
    - \nabla_{x_t} f_{t,i} \big(\xi_t, \hat\EE_t^{\ell, {\rm o}}\big[\hat{H}_{t+1}\big]\big) \\
    &\qquad\; - \nabla^2_{x_t} f_{t,i} \big(\xi_t,\EE_t[\hat{H}_{t+1}]\big)\; \big(\hat\EE_t^{\ell, {\rm o}}\big[\hat{H}_{t+1}\big]-\EE_t[\hat{H}_{t+1}]\big)\big\|_2
\end{align*}
by virtue of the triangle inequality. Next, we bound the three norm terms by using the estimate~\eqref{eq:inequality_for_holder_funs}, which is allowed because the Hessians are H\"older continuous thanks to 
Assumption~\ref{assump:holder}. We therefore obtain
\begin{align*}
   \|A_t^\ell e_i\|_{2} 
    \leq &\;\frac{R_t}{\rho_t+1} \, \big\|\hat\EE_t^{\ell}\big[\hat{H}_{t+1}\big] - \EE_t[\hat{H}_{t+1}]\big\|_2^{\rho_t+1} 
    + \frac{R_t}{2(\rho_t+1)} \, \big\|\hat\EE_t^{\ell, {\rm e}}\big[\hat{H}_{t+1}\big] - \EE_t[\hat{H}_{t+1}]\big\|_2^{\rho_t+1} \\
    &+ \frac{R_t}{2(\rho_t+1)} \, \big\|\hat\EE_t^{\ell, {\rm o}}\big[\hat{H}_{t+1}\big] - \EE_t[\hat{H}_{t+1}]\big\|_2^{\rho_t+1}.
\end{align*}
Two successive applications of Jensen's inequality yield $|a+b+c|^p \le 2^{p-1}(|a|^p+2^{p-1}(|b|^p+|c|^p))$ for all $a,b,c\in\RR$ and~$p\geq 1$. By raising both sides of the previous inequality to the power $p=2^t$ we thus find
\begin{align*}
    \|A_t^\ell e_i\|_{2}^{2^t}
    & \leq \Big(\frac{R_t}{\rho_t+1}\Big)^{2^t}2^{2^t-1} \; \Big[
    \big\|\hat\EE_t^{\ell}\big[\hat{H}_{t+1}\big] - \EE_t[\hat{H}_{t+1}]\big\|_2^{2^t(\rho_t+1)} \\
    &\qquad\;\quad + \frac{1}{2} \; \big\|\hat\EE_t^{\ell, {\rm e}}\big[\hat{H}_{t+1}\big] - \EE_t[\hat{H}_{t+1}]\big\|_2^{2^t(\rho_t+1)} + \frac{1}{2} \; \big\|\hat\EE_t^{\ell, {\rm o}}\big[\hat{H}_{t+1}\big] - \EE_t[\hat{H}_{t+1}]\big\|_2^{2^t(\rho_t+1)} \Big]\\
    & \leq \Big(\frac{R_t}{\rho_t+1}\Big)^{2^t} 2^{2^t-1} (d_t)^{2^{t-1}(\rho_t+1)-1} \; \Big[
    \big\|\hat\EE_t^{\ell}\big[\hat{H}_{t+1}\big] - \EE_t[\hat{H}_{t+1}]\big\|_{2^t(\rho_t+1)}^{2^t(\rho_t+1)} \\
    &\qquad\;\quad + \frac{1}{2} \; \big\|\hat\EE_t^{\ell, {\rm e}}\big[\hat{H}_{t+1}\big] - \EE_t[\hat{H}_{t+1}]\big\|_{2^t(\rho_t+1)}^{2^t(\rho_t+1)} + \frac{1}{2} \; \big\|\hat\EE_t^{\ell, {\rm o}}\big[\hat{H}_{t+1}\big] - \EE_t[\hat{H}_{t+1}]\big\|_{2^t(\rho_t+1)}^{2^t(\rho_t+1)} \Big],
\end{align*}
where the second inequality follows from Lemma~\ref{lemma:matrix-inequality-pq} with $p=2$, $q=2^{t}(\rho_t+1)$, $m=1$ and $n=d_t$. Substituting the above bound on~$\|A_t^\ell e_i\|_{2}^{2^t}$ back into~\eqref{eq:A^ell_t-norm-bound} then results in the estimate
\begin{equation}
    \label{eq:A_ell_t_bound}
    \| A^\ell_t\|_{2}^{2^t} \leq \Big(\frac{R_t}{\rho_t+1}\Big)^{2^t} 2^{2^t-1} (d_{t-1})^{2^{t-1}}\; (d_t)^{2^{t-1}(\rho_t+1)-1} \bar{A}_t^\ell,
\end{equation}
where we have used the shorthand 
\begin{align*}
    \bar{A}_t^\ell & = \big\|\hat\EE_t^{\ell}\big[\hat{H}_{t+1}\big] - \EE_t[\hat{H}_{t+1}]\big\|_{2^t(\rho_t+1)}^{2^t(\rho_t+1)} \\
    &\quad + \frac{1}{2} \; \big\|\hat\EE_t^{\ell, {\rm e}}\big[\hat{H}_{t+1}\big] - \EE_t[\hat{H}_{t+1}]\big\|_{2^t(\rho_t+1)}^{2^t(\rho_t+1)} + \frac{1}{2} \; \big\|\hat\EE_t^{\ell, {\rm o}}\big[\hat{H}_{t+1}\big] - \EE_t[\hat{H}_{t+1}]\big\|_{2^t(\rho_t+1)}^{2^t(\rho_t+1)}.
\end{align*}
Substituting~\eqref{eq:A_ell_t_bound} into~\eqref{eq:last-term} and applying the Cauchy-Schwarz and Jensen inequalities then yields
\begin{equation}
\label{eq:last-term-grad-proof}
\begin{aligned}
    & \EE_{t-1}\Big[\big\|\EE_t[\hat{G}_{t+1}] \big\|_{2}^{2^t} \big\| A^\ell_t\big\|_{2}^{2^t}\Big]\\
    & \leq \Big(\frac{R_t}{\rho_t+1}\Big)^{2^t} 2^{2^t-1} (d_{t-1})^{2^{t-1}} (d_t)^{2^{t-1}(\rho_t+1)-1} \EE_{t-1}\Big[\big\|\EE_t[\hat{G}_{t+1}] \big\|_{2}^{2^t}  \bar{A}_t^\ell \Big] \\
    & \leq  \Big(\frac{R_t}{\rho_t+1}\Big)^{2^t} 2^{2^t-1} (d_{t-1})^{2^{t-1}} (d_t)^{2^{t-1}(\rho_t+1)-1} \; \EE_{t-1}\Big[\big\|\hat{G}_{t+1} \big\|_{2}^{2^{t+1}}\Big]^\frac{1}{2} \EE_{t-1}\Big[ \big(\bar{A}_t^\ell\big)^2\Big]^\frac{1}{2} .
\end{aligned}
\end{equation}
Note that $\hat G_{t+1}$ and $\hat H_{t+1}$ are constructed from the same sets of samples, which implies that~$\EE_t[\hat{G}_{t+1}]$ and~$\bar{A}_t^\ell$ are dependent. 
To further simplify the bound in~\eqref{eq:last-term-grad-proof}, we use the Minkowski inequality to obtain
\begin{align*}
    & \EE_{t-1}\Big[ \big(\bar{A}_t^\ell\big)^2\Big]^\frac{1}{2} \\
    & \leq \EE_{t-1}\Big[\big\|\hat\EE_t^{\ell}\big[\hat{H}_{t+1}\big] - \EE_t[\hat{H}_{t+1}]\big\|_{2^t(\rho_t+1)}^{2^{t+1}(\rho_t+1)}\Big]^\frac{1}{2} \\
    &\quad +\frac{1}{2} \, \EE_{t-1}\Big[\big\|\hat\EE_t^{\ell, {\rm e}}\big[\hat{H}_{t+1}\big] - \EE_t[\hat{H}_{t+1}]\big\|_{2^t(\rho_t+1)}^{2^{t+1}(\rho_t+1)}\Big]^\frac{1}{2} +\frac{1}{2} \, \EE_{t-1}\Big[\big\|\hat\EE_t^{\ell, {\rm o}}\big[\hat{H}_{t+1}\big] - \EE_t[\hat{H}_{t+1}]\big\|_{2^t(\rho_t+1)}^{2^{t+1}(\rho_t+1)}\Big]^\frac{1}{2}\\
    & \leq (d_t)^\frac{1}{2} \bigg( \EE_{t-1}\Big[\big\|\hat\EE_t^{\ell}\big[\hat{H}_{t+1}\big] - \EE_t[\hat{H}_{t+1}]\big\|_{2^{t+1}(\rho_t+1)}^{2^{t+1}(\rho_t+1)}\Big]^\frac{1}{2} \\
    & \quad +\frac{1}{2} \, \EE_{t-1}\Big[\big\|\hat\EE_t^{\ell, {\rm e}}\big[\hat{H}_{t+1}\big] - \EE_t[\hat{H}_{t+1}]\big\|_{2^{t+1}(\rho_t+1)}^{2^{t+1}(\rho_t+1)}\Big]^\frac{1}{2} +\frac{1}{2} \, \EE_{t-1}\Big[\big\|\hat\EE_t^{\ell, {\rm o}}\big[\hat{H}_{t+1}\big] - \EE_t[\hat{H}_{t+1}]\big\|_{2^{t+1}(\rho_t+1)}^{2^{t+1}(\rho_t+1)}\Big]^\frac{1}{2} \bigg).
\end{align*}
Here, the second inequality follows from Lemma~\ref{lemma:matrix-inequality-pq} with $p=2^t(\rho_t+1)$, $q=2^{t+1}(\rho_t+1)$, $m=1$ and~$n=d_t$.
By using Lemma~\ref{lemma:moment-inequality-multidimensional} with $p=2^{t+1}(\rho_t+1)$ and $n=2^\ell$ or $n=2^{\ell-1}$, we further obtain
\begin{align*}
    \EE_{t-1}\Big[ \big(\bar{A}_t^\ell\big)^2\Big]^\frac{1}{2} & \leq (d_t)^\frac{1}{2}\,\frac{B_{2^{t+1}(\rho_t+1)}^\frac{1}{2}}{2^{(2^{t-1})(\rho_t+1)\ell}} \;\EE_{t-1}\Big[\big\|\hat{H}_{t+1}\big\|_{2^{t+1}(\rho_t+1)}^{2^{t+1}(\rho_t+1)}\Big]^\frac{1}{2} \\
    &\quad + 
    (d_t)^\frac{1}{2} \; \frac{B_{2^{t+1}(\rho_t+1)}^\frac{1}{2}}{2^{(2^{t-1)}(\rho_t+1)(\ell-1)}}\, \EE_{t-1}\Big[\big\|\hat{H}_{t+1}\big\|_{2^{t+1}(\rho_t+1)}^{2^{t+1}(\rho_t+1)}\Big]^\frac{1}{2}\\
    & = \Big(1 + 2^{2^{t-1}(\rho_t+1)}\Big) \, (d_t)^\frac{1}{2} \frac{B_{2^{t+1}(\rho_t+1)}^\frac{1}{2}}{2^{(2^{t-1})(\rho_t+1)\ell}}\, \EE_{t-1}\Big[\big\|\hat{H}_{t+1}\big\|_{2^{t+1}(\rho_t+1)}^{2^{t+1}(\rho_t+1)}\Big]^\frac{1}{2}  .
\end{align*}
Substituting the resulting expression into \eqref{eq:last-term-grad-proof} and combining it with \eqref{eq:last-term} finally yields
\begin{equation*}
    \begin{aligned}
    \EE_{t-1}\Big[\big\|\EE_t[\hat{G}_{t+1}] A^\ell_t\big\|_{2}^{2^t}\Big] & \leq \Big(\frac{R_t}{\rho_t+1}\Big)^{2^t} 2^{2^t-1} (d_{t-1})^{2^{t-1}}(d_t)^{2^{t-1}(\rho_t+1)-1} \, \EE_{t-1}\Big[\big\| \hat{G}_{t+1} \big\|_{2}^{2^{t+1}}\Big]^\frac{1}{2} \\
    &\quad \times \Big(1 + 2^{2^{t-1}(\rho_t+1)}\Big) (d_t)^\frac{1}{2}\, \frac{B_{2^{t+1}(\rho_t+1)}^\frac{1}{2}}{2^{(2^{t-1})(\rho_t+1)\ell}} \, \EE_{t-1}\Big[\big\|\hat{H}_{t+1}\big\|_{2^{t+1}(\rho_t+1)}^{2^{t+1}(\rho_t+1)}\Big]^\frac{1}{2} .
\end{aligned}
\end{equation*}
We have thus established an upper bound on the last term in~\eqref{eq:2t-th-moment-delta}, which depends on conditional moments of~$\hat G_{t+1}$ and~$\hat H_{t+1}$ and decays geometrically with~$\ell$. This concludes Step~(i).

\paragraph{Step (ii)} To bound the gradient difference terms in~\eqref{eq:2t-th-moment-delta} depending on~$\hat H_{t+1}$, we note that 
\begin{subequations}
\begin{equation*}
\begin{aligned}
    & \EE_{t-1} \Big[\Big\|\nabla_{x_t} f_t \big(\xi_t, \hat\EE_t^{\ell}\big[\hat{H}_{t+1}\big]\big) - \nabla_{x_t} f_t\big(\xi_t, \EE_t[\hat{H}_{t+1}] \big) \Big\|_{2}^{2^{t+1}} \Big]
    \leq S_t^{2^{t+1}} \EE_{t-1} \Big[ \big\|\hat\EE_t^{\ell}\big[\hat{H}_{t+1}\big] - \EE_t[\hat{H}_{t+1}] \big\|_{2}^{2^{t+1}} \Big] \\
    & \leq S_t^{2^{t+1}} (d_t)^{2^t-1} \EE_{t-1} \Big[ \big\|\hat\EE_t^{\ell}\big[\hat{H}_{t+1}\big] - \EE_t[\hat{H}_{t+1}] \big\|_{2^{t+1}}^{2^{t+1}} \Big] \leq S_t^{2^{t+1}} (d_t)^{2^t-1} \frac{B_{2^{t+1}}}{2^{2^{t}\ell}} \, \EE_{t-1}\Big[\big\|\hat{H}_{t+1}\big\|_{2^{t+1}}^{2^{t+1}}\Big],
\end{aligned}
\end{equation*}
where the first inequality exploits Assumption~\ref{assump:smooth}, and the second inequality follows from Lemma~\ref{lemma:matrix-inequality-pq} with~$p=2$, $q=2^{t+1}$, $m=1$ and~$n=d_t$. The third inequality follows from the law of total expectation and from Lemma~\ref{lemma:moment-inequality-multidimensional} with~$p=2^{t+1}$ and~$n=2^\ell$. A similar reasoning implies that
\begin{equation*}
\begin{aligned}
    \EE_{t-1} \Big[ \Big\|\nabla_{x_t} f_t \big(\xi_t, \hat\EE_t^{\ell, {\rm e}}\big[\hat{H}_{t+1}\big]\big) - \nabla_{x_t} f_t\big(\xi_t, \EE_t[\hat{H}_{t+1}] \big) \Big\|_{2}^{2^{t+1}} \Big]
    &\! \leq  S_t^{2^{t+1}} (d_t)^{2^t-1} \frac{B_{2^{t+1}}}{2^{2^{t}(\ell-1)}}  \EE_{t-1}\Big[\big\|\hat{H}_{t+1}\big\|_{2^{t+1}}^{2^{t+1}}\Big] ,\\
    \EE_{t-1} \Big[ \Big\|\nabla_{x_t} f_t \big(\xi_t, \hat\EE_t^{\ell, {\rm o}}\big[\hat{H}_{t+1}\big]\big) - \nabla_{x_t} f_t\big(\xi_t, \EE_t[\hat{H}_{t+1}] \big) \Big\|_{2}^{2^{t+1}} \Big]
    & \! \leq  S_t^{2^{t+1}} (d_t)^{2^t-1} \frac{B_{2^{t+1}}}{2^{2^{t}(\ell-1)}} \EE_{t-1}\Big[\big\|\hat{H}_{t+1}\big\|_{2^{t+1}}^{2^{t+1}}\Big].
\end{aligned}
\end{equation*}
\end{subequations}
As the empirical averages here are taken only over samples with even or odd indices, we set~$n = 2^{\ell-1}$ when applying Lemma~\ref{lemma:moment-inequality-multidimensional}. In summary, we have found upper bounds on the gradient difference terms in~\eqref{eq:2t-th-moment-delta}, which depend on conditional moments of~$\hat H_{t+1}$ and decay geometrically with~$\ell$. This concludes Step~(ii).

\paragraph{Step (iii)} To bound the terms in~\eqref{eq:2t-th-moment-delta} depending on~$\hat{G}_{t+1}$, we note that
\begin{align*}
    \EE_{t-1} \Big[ \Big\|\hat\EE_t^{\ell}\big[\hat{G}_{t+1}\big] -\EE_t[\hat{G}_{t+1}]  \Big\|_{2}^{2^{t+1}}\Big] & \leq (d_t d)^{2^t-1} \, \EE_{t-1} \Big[ \Big\|\hat\EE_t^{\ell}\big[\hat{G}_{t+1}\big] -\EE_t[\hat{G}_{t+1}]  \Big\|_{2^{t+1}}^{2^{t+1}}\Big] \\
    & \leq (d_t d)^{2^t-1} \frac{B_{2^{t+1}}}{2^{2^{t}\ell}} \, \EE_{t-1}\Big[\big\|\hat{G}_{t+1}\big\|_{2^{t+1}}^{2^{t+1}}\Big]\\
    & \leq (d_t d)^{2^t-1} \frac{B_{2^{t+1}}}{2^{2^{t}\ell}} \, \EE_{t-1}\Big[\big\|\hat{G}_{t+1}\big\|_{2}^{2^{t+1}}\Big],
\end{align*}
where the first inequality follows from Lemma~\ref{lemma:matrix-inequality-pq} with~$p=2$, $q=2^{t+1}$, $m=d_{t}$ and~$n=d$, while the second inequality follows from the law of total expectation and from Lemma~\ref{lemma:moment-inequality-multidimensional} with~$p=2^{t+1}$ and~$n=2^\ell$. The third inequality follows from Jensen's inequality (or from Lemma~\ref{lemma:matrix-inequality-pq} with~$p=2$ and~$q=2^{t+1}$). 
A similar reasoning involving Lemma~\ref{lemma:moment-inequality-multidimensional} with~$p=2^{t+1}$ and~$n=2^{\ell-1}$ implies that
\begin{align*}
    \EE_{t-1}\Big[\big\|\hat\EE_t^{\ell, {\rm e}}\big[\hat{G}_{t+1}\big] - \EE_t[\hat{G}_{t+1}] \big\|_{2}^{2^{t+1}}\Big] & 
    \leq (d_t d)^{2^t-1} \frac{B_{2^{t+1}}}{2^{2^{t}(\ell-1)}} \EE_{t-1}\Big[\big\|\hat{G}_{t+1}\big\|_{2}^{2^{t+1}}\Big],\\
    \EE_{t-1}\Big[\big\|\hat\EE_t^{\ell, {\rm o}}\big[\hat{G}_{t+1}\big] - \EE_t[\hat{G}_{t+1}] \big\|_{2}^{2^{t+1}}\Big] 
    & \leq (d_t d)^{2^t-1} \frac{B_{2^{t+1}}}{2^{2^{t}(\ell-1)}} \EE_{t-1}\Big[\big\|\hat{G}_{t+1}\big\|_{2}^{2^{t+1}}\Big].
\end{align*}
These bounds depend on conditional moments of~$\hat G_{t+1}$ and decay geometrically with~$\ell$, completing Step~(iii).

Having analyzed all terms in~\eqref{eq:2t-th-moment-delta} in Steps~(i), (ii) and~(iii), we are now ready to combine the resulting estimates to a bound on the conditional moments of~$\Delta^\ell$. For notational convenience, we define the constants
\begin{equation*}
\begin{aligned}
    Y_t \coloneqq & 3 \cdot 2^{2^{t+1}-3}(d_t d)^{2^t-1} B_{2^{t+1}},\\
    Y^{\prime}_t \coloneqq & 3 \cdot 2^{2^{t+1}-3}  S_t^{2^{t+1}} (d_t)^{2^t-1}B_{2^{t+1}},\\
    Y^{\prime\prime}_t \coloneqq & 2^{2^{t+1}+2^t-3}\Big(\frac{R_t}{\rho_t+1}\Big)^{2^t} (d_{t-1})^{2^{t-1}} (d_t)^{2^{t-1}(\rho_t+1)-\frac{1}{2}}  \Big(1 + 2^{2^{t-1}(\rho_t+1)}\Big)  B_{2^{t+1}(\rho_t+1)}^\frac{1}{2},
\end{aligned}
\end{equation*}
which are independent of~$\ell$. Substituting the bounds from Steps~(i), (ii) and~(iii) into~\eqref{eq:2t-th-moment-delta} then yields
\begin{equation}\label{eq:delta_bound_with_moments}
    \begin{aligned}
    \EE_{t-1}\big[\|\Delta^\ell \|_{2}^{2^t}\big]
    & \leq  \frac{Y_t}{2^{2^{t}\ell}} \; \EE_{t-1}\Big[\big\|\hat{G}_{t+1}\big\|_{2}^{2^{t+1}}\Big] 
    + \frac{Y_t^\prime}{2^{2^{t}\ell}} \;  \EE_{t-1}\Big[\big\|\hat{H}_{t+1}\big\|_{2^{t+1}}^{2^{t+1}}\Big] \\
    &\qquad + \frac{Y_t^{\prime\prime}}{2^{(2^{t-1})(\rho_t+1)\ell}} \; \EE_{t-1}\Big[\big\| \hat{G}_{t+1} \big\|_{2}^{2^{t+1}}\Big]^\frac{1}{2} \; \EE_{t-1}\Big[\big\|\hat{H}_{t+1}\big\|_{2^{t+1}(\rho_t+1)}^{2^{t+1}(\rho_t+1)}\Big]^\frac{1}{2}.
\end{aligned}
\end{equation}
Next, we bound the three conditional expectations appearing in~\eqref{eq:delta_bound_with_moments}. By Lemma~\ref{lemma:bound_on_variance_rtmlmc}, under Assumption~\ref{assump:smooth}, and using the fact that $2^{s-1}(\rho_s+1)\leq 2^{s}$, the second conditional expectation satisfies

\[
    \EE_{t-1}\big[\|\hat H_{t+1}\|_{2^{t+1}}^{2^{t+1}}\big]
    \leq D_{t+1}\prod_{s=t+1}^{T-1}\bigg(\sum_{\ell=0}^{M_s}\frac{1}{2^{2^{s}\ell} q_s(\ell)^{2^{s}-1}}\bigg) \leq D_{t+1}\prod_{s=t+1}^{T-1} \cS_s, 
\]  
where
\[
 D_{t+1} = \overline\mu_T^{2^T} \prod_{s=t+1}^{T-1} \bigg(\bigg(\frac{3S_s}{2}\bigg)^{2^s} d_s^{2^s-1} B_{2^{s+1}}\bigg)
\]
and~$\cS_s$ is defined in~\eqref{eq:S-definition}. By adapting Lemma~\ref{lemma:bound_on_variance_rtmlmc} in the obvious way, one can further show that 
\begin{align*}
     \EE_{t}\Big[\big\|\hat{H}_{t+1}\big\|_{2^{t+1}(\rho_t+1)}^{2^{t+1}(\rho_t+1)}\Big]  \leq 
        D_{t+1}^\prime \prod_{s=t+1}^{T-1} \cS_s^\prime, 
\end{align*}
where
\begin{align*}
    D_{t+1}^\prime\coloneqq\overline\mu_T^{2^T(\rho_{T-1}+1)} \prod_{s=t+1}^{T-1} \bigg(\bigg(\frac{3S_s}{2}\bigg)^{2^{s}(\rho_{s-1}+1)} d_s^{2^{s}(\rho_{s-1}+1)-1} B_{2^{s+1}(\rho_{s-1}+1)}\bigg)
\end{align*}
and~$\cS_s^\prime$ is defined in~\eqref{eq:S-definition}. Moreover, by the induction hypothesis, the conditional moment of $\hat{G}_{t+1}$ in~\eqref{eq:delta_bound_with_moments} satisfies the bound~\eqref{eq:moment_lemma_inductive_hypothesis_grad}. Substituting all  of these conditional moment bounds into~\eqref{eq:delta_bound_with_moments} then yields
\begin{equation*}
    \begin{aligned}
    \EE_{t-1}\big[\|\Delta^\ell \|_{2}^{2^t}\big] &\leq   \frac{Y_t E_{t+1} }{2^{2^{t}\ell}}\prod_{s=t+1}^{T-1}\max\big\{\cS_s, \cS_s^\prime \big\}
    + \frac{ Y^\prime_t D_{t+1} }{2^{2^t \ell}}\prod_{s=t+1}^{T-1} \cS_s \\
    &\quad + \frac{Y_t^{\prime\prime} \sqrt{E_{t+1} {D_{t+1}^\prime}}}{2^{(2^{t-1})(\rho_t+1)\ell}}\prod_{s=t+1}^{T-1}\max\big\{\cS_s, \cS_s^\prime \big\}^\frac{1}{2} \big(\cS_s^\prime\big)^\frac{1}{2} \leq \frac{E_t }{2^{(2^{t-1})(\rho_t+1)\ell}} \prod_{s=t+1}^{T-1}\max\big\{\cS_s, \cS_s^\prime \big\},
\end{aligned}
\end{equation*}
where the second inequality follows from the trivial observations that $2^{t-1}(\rho_t+1)\leq 2^{t}$ and that both~$\cS_s$ as well as~$\cS_s'$ are dominated by $\max\{\cS_s, \cS_s^\prime \}$. The inequality also exploits the (recursive) definition of~$E_t$ as
\[
    E_t \coloneqq Y_t E_{t+1} + Y^\prime_t D_{t+1} + Y_t^{\prime\prime} \sqrt{E_{t+1} {D_{t+1}^\prime}}.
\]
Replacing the conditional moment $\EE_{t-1}[\|\Delta^\ell \|_{2}^{2^t}]$ in~\eqref{eq:telescoping_step_rumlmc} with the above bound finally yields
\begin{align*}                  
    \nu_t^{2^t} & = \sum_{\ell=0}^{M_t} \frac{\EE_{t-1}\big[\|\Delta^\ell\|^{2^t}_{2} \big]}{q_t(\ell)^{2^t-1}} 
    \leq E_t\, \cS_t \prod_{s=t+1}^{T-1} \max\big\{\cS_s, \cS_s^\prime \big\} \leq E_t\prod_{s=t}^{T-1}\max\big\{\cS_s, \cS_s^\prime \big\}.
\end{align*}
This completes the induction step. By their recursive definition and thanks to Assumption~\ref{assump:nu_bar}, the constants~$E_t$, $t \in [T-1]$, are finite whenever~$E_T$ is finite. Hence, the claim follows.
\end{proof}

\begin{proof}[Proof of Lemma~\ref{lemma:bias_bound_rtmlmc_grad}] 
We simplify notation by suppressing the dependence on~$x$ for all functions of~$x$ in this proof. In the following we will show that the conditional bias of the estimator~$\hat{G}_t(x)$ obeys the estimate
\begin{equation}
    \label{eq:bias-G_t}
    \begin{aligned}
    & \big\|\EE_{t-1}[\hat{G}_{t}]-G_{t} \big\|_2 \\ 
    & \qquad \leq \sum_{s=t}^{T-1}L_{[t:s-1]}S_s \EE_{t-1} [\nu_{s+1}^{2}]^\frac{1}{2}\bigg(\frac{ \EE_{t-1}[\mu_{s+1}^{2}]^\frac{1}{2}}{2^{M_s/2}} + \sum_{k=s+1}^{T-1} \frac{L_{[s+1:k-1]}S_k\EE_{t-1} [(\mu_{k+1}^{2})^2]^\frac{1}{2}}{2^{M_k+1}} \bigg)
    \end{aligned}
\end{equation}
for all~$t\in[T]$. As~$\EE[\hat{G}]=\EE[\hat{G}_1]$, the claim then follows immediately. We prove~\eqref{eq:bias-G_t} by backward induction on~$t$. As for the base case corresponding to~$t = T$, note that~$\hat{G}_T$ constitutes an unbiased estimator for~$G_T$. Thus, the left hand side of~\eqref{eq:bias-G_t} evaluates to~$0$, and the claim holds trivially (in fact, the right hand side of~\eqref{eq:bias-G_t} vanishes, too). As for the induction step corresponding to any~$t$ with~$t+1\in[T]$, assume that
\begin{align*}\label{eq:induction_hypothesis_bias_grad}
     \big\|\EE_{t}[\hat{G}_{t+1}]-G_{t+1} \big\|_2\leq \! \sum_{s=t+1}^{T-1}L_{[t+1:s-1]}S_s \EE_t[\nu_{s+1}^{2}]^\frac{1}{2}\bigg(\frac{ \EE_t[\mu_{s+1}^{2}]^\frac{1}{2}}{2^{M_s/2}}\! + \!\sum_{k=s+1}^{T-1} \!\frac{L_{[s+1:k-1]}S_k\EE_t[(\mu_{k+1}^{2})^2]^\frac{1}{2}}{2^{M_k+1}} \bigg).
\end{align*}
In addition, note that
\begin{equation}
    \begin{aligned}\label{eq:rtmlmc_grad_bias}
    \EE_{t-1}[\hat{G}_t] & =  \EE_{t-1}\left[\frac{1}{q_t(\lambda_t)} \left(\hat{g}_t^{\lambda_t}-\frac{1}{2}\hat{g}_t^{\lambda_t, {\rm e}} -\frac{1}{2}\hat{g}_t^{\lambda_t, {\rm o}} \right)\right] = \sum_{\ell=0}^{M_t} \EE_{t-1}\left[\hat{g}_t^{\ell}-\frac{1}{2}\hat{g}_t^{\ell, {\rm e}} -\frac{1}{2}\hat{g}_t^{\ell, {\rm o}} \right] \\
    &  = \EE_{t-1} [\hat{g}_t^{M_t}] 
    = \EE_{t-1}\Big[ \hat\EE_t^{M_t}[\hat{G}_{t+1}] \nabla_{x_t} f_t \big(\xi_t, \hat\EE_t^{M_t}[\hat{H}_{t+1}] \big)\Big],
\end{aligned}
\end{equation}
where the first equality follows from the definition of~$\hat{G}_t$ in~\eqref{eq:G_hat}, and the second uses the law of total expectation along with the independence of~$\lambda_t$ from all other random objects. The third equality holds because~$\hat{g}_t^{\ell,{\rm e}}$ and~$\hat{g}_t^{\ell,{\rm o}}$ share the distribution of~$\hat{g}_t^{\ell-1}$ conditional on $\mathcal F_t$ for all~$\ell \in [M_t]$ and because~$\hat{g}_t^{0,{\rm e}} = \hat{g}_t^{0,{\rm o}} = 0$. Finally, the fourth equality follows from Definition~\ref{def:rumlmc-grad}. Next, we introduce the shorthands
\begin{align*}
    A_1 & \coloneqq \Big\|\EE_{t-1}\Big[\hat\EE_t^{M_t}[\hat{G}_{t+1}] \left(\nabla_{x_t} f_t \big(\xi_t,\hat\EE_t^{M_t}[\hat{H}_{t+1}] \big) -\nabla_{x_t} f_t \left(\xi_t,F_{t+1}^T\right)\right)\Big] \Big\|_2\\
    A_2 & \coloneqq \Big\|\EE_{t-1}\Big[\Big(\EE_{t}[\hat{G}_{t+1}]-G_{t+1}\Big) \nabla_{x_t} f_t \left(\xi_t,F_{t+1}^T\right) \Big] \Big\|_2
\end{align*}
to simplify notation. Then, we have 
\begin{equation}\label{eq:bias_t_th_step}
\begin{aligned}
        \big\| \EE_{t-1}[\hat{G}_t] - G_t\big\|_2 & = \Big\|\EE_{t-1}\Big[\hat\EE_t^{M_t}[\hat{G}_{t+1}] \nabla_{x_t} f_t \big(\xi_t,\hat\EE_t^{M_t}[\hat{H}_{t+1}] \big) \Big] - \EE_{t-1}\Big[G_{t+1}\nabla_{x_t} f_t \big(\xi_t,F_{t+1}^T\big)\Big] \Big\|_2 \\
        &\leq \Big\|\EE_{t-1}\Big[\hat\EE_t^{M_t}[\hat{G}_{t+1}]\nabla_{x_t} f_t \big(\xi_t,\hat\EE_t^{M_t}[\hat{H}_{t+1}] \big) - \hat\EE_t^{M_t}[\hat{G}_{t+1}]\nabla_{x_t} f_t \big(\xi_t,F_{t+1}^T\big)\Big] \Big\|_2\\
        &\quad + \Big\|\EE_{t-1}\Big[\hat\EE_t^{M_t}[\hat{G}_{t+1}]\nabla_{x_t} f_t \big(\xi_t,F_{t+1}^T\big) - G_{t+1}  \nabla_{x_t} f_t \left(\xi_t,F_{t+1}^T\right) \Big]\Big\|_2 \! = A_1 \! + A_2,
\end{aligned}
\end{equation}
where the first equality exploits~\eqref{eq:rtmlmc_grad_bias} and the definition of~$G_t$ in~\eqref{eq:G_t}. The second equality follows from the law of total expectation and the elementary identity $\EE_{t}[\hat\EE_t^{M_t}[\cdot]]=\EE_t[\cdot]$. Next, we separately bound the terms~$A_1$ and~$A_2$. As for~$A_1$, we use the Cauchy-Schwarz and Jensen inequalities to obtain
\begin{align*}
        A_1 & \leq \EE_{t-1}\Big[\hat\EE_t^{M_t}\Big[\big\|\hat{G}_{t+1}\big\|_2\Big] \big\|\nabla_{x_t} f_t \big(\xi_t,\hat\EE_t^{M_t}[\hat{H}_{t+1}] \big) -\nabla_{x_t} f_t \left(\xi_t,F_{t+1}^T\right)\big\|_2 \Big] \\
        & \leq S_t \, \EE_{t-1}\Big[\big\|\hat{G}_{t+1}\big\|_2 \big\| \, \hat\EE_t^{M_t}[\hat{H}_{t+1}]  - F_{t+1}^T \big\|_2 \Big] \\
        & \leq S_t \, \EE_{t-1}\Big[\big\|\hat{G}_{t+1}\big\|_2^2\Big]^\frac{1}{2} \EE_{t-1}\Big[\big\|\hat\EE_t^{M_t}[\hat{H}_{t+1}]  - F_{t+1}^T \big\|_2^2 \Big]^\frac{1}{2} \\
        & = S_t \, \EE_{t-1}\Big[\big\|\hat{G}_{t+1}\big\|_2^2\Big]^\frac{1}{2}\Big( \EE_{t-1}\Big[\big\|\hat\EE_t^{M_t}[\hat{H}_{t+1}]  - \EE_t[\hat{H}_{t+1}] \big\|_2^2 \Big] 
        +\EE_{t-1}\Big[\big\|\EE_t[\hat{H}_{t+1}] - F_{t+1}^T\big\|_2^2 \Big]\Big)^\frac{1}{2}.
\end{align*}
Here, the second inequality follows from the smoothness of the integrand~$f_t$, as stipulated in Assumption~\ref{assump:smooth}. Note that we have also removed the first sample average $\hat\EE_t^{M_t}[\cdot]$. This is permissible because the second norm term does not depend on~$\xi_{t+1}$ and because the underlying samples are independent and identically distributed conditional on the information available at time~$t-1$. The third inequality and the equality follow from the Cauchy-Schwarz inequality and from a standard bias-variance decomposition of the mean squared error. 
By Lemma~\ref{lemma:bias_bound_rtmlmc_H_t}, the bias term admits the following upper bound.
\begin{align*}
    &\EE_{t-1}\Big[\big\|\EE_t[\hat{H}_{t+1}] - F_{t+1}^T\big\|_2^2 \Big] \leq \EE_{t-1}\bigg[\bigg(\sum_{s=t+1}^{T-1} \frac{L_{[t+1:s-1]} S_s \EE_{t}[\mu_{s+1}^2]}{2^{M_s + 1}} \bigg)^2\bigg]
\end{align*}
In addition, the variance term satisfies
\begin{align*}
    \EE_{t-1}\Big[\big\|\hat\EE_t^{M_t}[\hat{H}_{t+1}]  - \EE_t[\hat{H}_{t+1}] \big\|_2^2 \Big] = 2^{-M_t} \EE_{t-1}\Big[\big\|\hat{H}_{t+1}  - \EE_t[\hat{H}_{t+1}] \big\|_2^2 \Big] \leq 2^{-M_t} \EE_{t-1}[\mu_{t+1}^{2}],
\end{align*}
where the equality holds because the variance of the sample average $\hat\EE_t^{M_t}[\hat{H}_{t+1}]$ equals the variance of~$\hat{H}_{t+1}$ divided by the sample size~$2^{M_t}$. The inequality follows from the observation that the variance of~$\hat{H}_{t+1}$ is bounded below by its second moment and from the definition of~$\mu_{t+1}^{2}$ in~\eqref{eq:mu}. Substituting the bounds on these mean and the variance terms into the above bound on~$A_1$ then yields
\begin{align*}
    A_1 &\leq S_t \, \EE_{t-1}\Big[\big\|\hat{G}_{t+1}\big\|_2^2\Big]^\frac{1}{2} \Bigg( \EE_{t-1}\bigg[\bigg(\sum_{s=t+1}^{T-1} \frac{L_{[t+1:s-1]} S_s \EE_t[\mu_{s+1}^2]}{2^{M_s + 1}}\bigg)^2\bigg] + \frac{\EE_{t-1}[\mu_{t+1}^2]}{2^{M_t}} \Bigg)^\frac{1}{2}\\
    & \leq S_t \, \EE_{t-1}[\nu_{t+1}^{2}]^\frac{1}{2} \Bigg( \EE_{t-1}\bigg[\bigg(\sum_{s=t+1}^{T-1} \frac{L_{[t+1:s-1]} S_s \EE_t[\mu_{s+1}^2]}{2^{M_s + 1}}\bigg)^2\bigg]^\frac{1}{2} + \frac{\EE_{t-1}[\mu_{t+1}^{2}]^\frac{1}{2}}{2^{M_t/2}} \Bigg),
\end{align*}
where the second inequality holds because of the definition of~$\nu_{t+1}^2$ in~\eqref{eq:nu} and because $(a^2+b^2)^\frac{1}{2} \leq |a| + |b|$ for all~$a,b\in\mathbb R$. The Minkowski and Jensen inequalities further imply that
\begin{align*}
    \EE_{t-1}\bigg[\bigg(\sum_{s=t+1}^{T-1} \frac{L_{[t+1:s-1]} S_s \EE_t[\mu_{s+1}^{2}]}{2^{M_s + 1}}\bigg)^2\bigg]^\frac{1}{2} 
    &\leq \sum_{s=t+1}^{T-1}\EE_{t-1}\bigg[ \frac{L_{[t+1:s-1]}^2 S_s^2 \EE_t[\mu_{s+1}^{2}]^2 }{2^{2M_s + 2}}\bigg]^\frac{1}{2} \\
    &\leq \sum_{s=t+1}^{T-1}\frac{L_{[t+1:s-1]} S_s \EE_{t-1}[(\mu_{s+1}^{2})^2]^\frac{1}{2}}{2^{M_s + 1}}.
\end{align*}
Using this inequality to simplify that above bound on~$A_1$ yields
\begin{align*}
    A_1 \leq S_t \, \EE_{t-1}[\nu_{t+1}^{2}]^\frac{1}{2} \bigg( \sum_{s=t+1}^{T-1} \frac{L_{[t+1:s-1]} S_s \EE_{t-1}[(\mu_{s+1}^{2})^2]^\frac{1}{2}}{2^{M_s + 1}} + \frac{\EE_{t-1}[\mu_{t+1}^{2}]^\frac{1}{2}}{2^{M_t/2}} \bigg).
\end{align*}
As for the term~$A_2$, we use the Cauchy-Schwarz and Jensen inequalities to obtain
\begin{align*}
    &A_2 \leq  \EE_{t-1}\Big[\big\|\nabla_{x_t} f_t \left(\xi_t,F_{t+1}^T\right)\big\|_2 \, \big\|\EE_{t}[\hat{G}_{t+1}]-G_{t+1}  \big\|_2 \Big]\\
    &\quad \leq  L_t \; \EE_{t-1} \Big[\big\|\EE_{t}[\hat{G}_{t+1}]-G_{t+1} \big\|_2\Big]\\
    &\quad \leq L_t \;\EE_{t-1}\bigg[  \sum_{s=t+1}^{T-1}L_{[t+1:s-1]}S_s \EE_t[\nu_{s+1}^{2}]^\frac{1}{2} \bigg(\frac{\EE_{t}[\mu_{s+1}^{2}]^\frac{1}{2} }{2^{M_s/2}} + \bigg(\sum_{k=s+1}^{T-1} \frac{L_{[s+1:k-1]}S_k\EE_{t}[(\mu_{k+1}^{2})^2]^\frac{1}{2}}{2^{M_k + 1}}\bigg) \bigg) \bigg],
\end{align*}
where the second inequality follows from Assumption~\ref{assump:lipschitz}, and the third inequality exploits the induction hypothesis. Jensen’s inequality together with the law of total expectations further yields
\begin{align*}
&A_2 \leq  L_t \sum_{s=t+1}^{T-1}L_{[t+1:s-1]}S_s \EE_{t-1}[\nu_{s+1}^{2}]^\frac{1}{2} \bigg(\frac{\EE_{t-1}[\mu_{s+1}^{2}]^\frac{1}{2} }{2^{M_s/2}} + \bigg(\sum_{k=s+1}^{T-1} \frac{L_{[s+1:k-1]}S_k\EE_{t-1}[(\mu_{k+1}^{2})^2]^\frac{1}{2}}{2^{M_k + 1}}\bigg) \bigg) .
\end{align*}
Combining the bounds on~$A_1$ and~$A_2$ with~\eqref{eq:bias_t_th_step} and noting that $L_{[t:t-1]}=1$, we finally obtain
\begin{align*}
    &\big\| \EE_{t-1}[\hat{G}_t] - G_t\big\|_2\\
    &\quad \leq S_t \, \EE_{t-1}[\nu_{t+1}^{2}]^\frac{1}{2} \bigg(\frac{\EE_{t-1}[\mu_{t+1}^2]^\frac{1}{2}}{2^{M_t/2}} +\bigg(\sum_{s=t+1}^{T-1} \frac{L_{[t+1:s-1]} S_s \EE_{t-1}[(\mu_{s+1}^{2})^2]^\frac{1}{2}}{2^{M_s + 1}}\bigg) \bigg)   \\
    &\qquad\;\quad + L_t \sum_{s=t+1}^{T-1}L_{[t+1:s-1]}S_s \EE_{t-1}[\nu_{s+1}^{2}]^\frac{1}{2}\bigg(\frac{ \EE_{t-1}[\mu_{s+1}^{2}]^\frac{1}{2} }{2^{M_s/2}} + \bigg(\sum_{k=s+1}^{T-1} \frac{L_{[s+1:k-1]}S_k\EE_{t-1}[(\mu_{k+1}^{2})^2]^\frac{1}{2}}{2^{M_k + 1}}\bigg) \bigg)\\
    &\quad = \sum_{s=t}^{T-1} L_{[t:s-1]}S_s \EE_{t-1}[\nu_{s+1}^{2}]^\frac{1}{2} \bigg(\frac{ \EE_{t-1}[\mu_{s+1}^{2}]^\frac{1}{2}}{2^{M_s/2}} 
    + \sum_{k=s+1}^{T-1} \frac{L_{[s+1:k-1]}S_k\EE_{t-1}[(\mu_{k+1}^{2})^2]^\frac{1}{2}}{2^{M_k+1}}\bigg).
\end{align*}
This completes the induction step, and thus the claim follows.
\end{proof}

\begin{proof}[Proof of Theorem~\ref{thm:rtmlmc_grad_sample_complexity_mse}]
    To bound the root mean squared error of~$\hat{G}(x)$ by~$\epsilon$, it suffices to bound both the bias and the standard deviation of~$\hat{G}(x)$ by~$\epsilon/\sqrt{2}$ for any $x\in\cX$. As for the bias,  Lemma~\ref{lemma:bias_bound_rtmlmc_grad} implies that
    \begin{align*}
      & \big\| \EE[\hat{G}(x)]-G_1(x) \big\|_2 \leq \sum_{t=1}^{T-1} L_{[t-1]}S_t \EE[\nu_{t+1}^{2}(x)]^\frac{1}{2}\bigg( \frac{ \EE[\mu_{t+1}^{2}(x)]^\frac{1}{2}}{2^{M_t/2}} + \sum_{s=t+1}^{T-1} \frac{L_{[t+1:s-1]}S_s\EE[(\mu_{s+1}^{2}(x))^2]^\frac{1}{2}}{2^{M_s+1}} \bigg) \\
      & \qquad \leq \sum_{t=1}^{T-1} L_{[t-1]}S_t \EE \big[\nu_{t+1}^{2^{t+1}}(x) \big]^\frac{1}{2^{t+1}}  \bigg(\frac{\EE \big[\mu_{t+1}^{2^{t+1}}(x) \big]^\frac{1}{2^{t+1}}}{2^{M_t/2}} + \sum_{s=t+1}^{T-1} \frac{L_{[t+1:s-1]}S_s\EE \big[(\mu_{s+1}^{2^{s+1}}(x))^2 \big]^\frac{1}{2^{s+1}}}{2^{M_s+1}} \bigg),
    \end{align*}
    where the last expression follows from repeated application of Jensen's inequality and from the definitions of the conditional moments~$\mu_t^p(x)$ and~$\nu_t^p(x)$ in~\eqref{eq:mu} and~\eqref{eq:nu}, respectively. Lemma~\ref{lemma:bound_on_variance_rumlmc_grad} further implies that
    \begin{align*}
       \EE \big[\nu_{t+1}^{2^{t+1}}(x) \big]^\frac{1}{2^{t+1}} & \leq E_{t+1}^\frac{1}{2^{t+1}} \prod_{s=t+1}^{T-1} \max\big\{\cS_s, \cS_s^\prime \big\}^\frac{1}{2^{t+1}},
    \end{align*}
    where the auxiliary variables~$\cS_s$ and~$\cS_s^\prime$ are defined as in~\eqref{eq:S-definition} in the proof of Lemma~\ref{lemma:bound_on_variance_rumlmc_grad}. Thus, we have
    \begin{align*}
        \cS_s & = \sum_{\ell=0}^{M_s}\frac{1}{2^{(2^{s-1})(\rho_s+1)\ell} q_s(\ell)^{2^{s}-1}} = \sum_{\ell=0}^{M_s}\frac{Z_s^{2^s - 1} }{ 2^{(2^{s-1})(\rho_s+1)\ell} (1-r_s)^{(2^s-1)\ell} }.
    \end{align*}
    Here, the second equality holds because $q_s(\ell) = (1-r_s)^\ell/Z_s$, and~$Z_s \coloneqq  (1-(1-r_s)^{M_s+1})/r_s$ is shorthand for the normalization constant of the truncated geometric distribution~$\mathrm{Geo}(r_s | M_s)$. Note that~$Z_s\leq 1/r_s$. The first upper bound on~$r_s$ in Assumption~\ref{assump:rate-parameters-rumlmc-grad} thus ensures that~$\cS_s$ remains bounded as~$M_s$ tends to infinity. Consequently, there exists a constant~$\bar\cS_s\geq \cS_s$ that is independent of~$M_s$. Similarly, one can show that there exists a constant~$\bar{\cS}_s^\prime\geq \cS_s^\prime$ that is independent of~$M_s$. In summary, we may then conclude that
    \begin{align*}
        \EE \big[\nu_{t+1}^{2^{t+1}}(x) \big]^\frac{1}{2^{t+1}}
       & \leq E_{t+1}^\frac{1}{2^{t+1}} \prod_{s=t+1}^{T-1} \max\big\{ \bar\cS_s, \bar\cS_s^\prime \big\}^\frac{1}{2^{t+1}}
    \end{align*}
    uniformly across all~$x\in\cX$ and~$M_s$, $s=t+1,\ldots, T-1$. As~$\rho_{s}\leq 1$, Lemma~\ref{lemma:bound_on_variance_rtmlmc} further implies that
    \begin{align*}
        \mu_{t+1}^{2^{t+1}}(x) \leq  D_{t+1} \prod_{s=t+1}^{T-1} \bar\cS_s
    \end{align*}
    uniformly across all~$x\in\cX$ and~$M_s$, $s=t+1,\ldots, T-1$. Combining the above estimates then yields
    \begin{align*}
        \|\EE[\hat{G}(x)]-G_1(x)\|_2 \leq & \sum_{t=1}^{T-1} L_{[t-1]}S_t  E_{t+1}^\frac{1}{2^{t+1}} \Big(\prod_{s=t+1}^{T-1}\textstyle\max\big\{\bar\cS_s, \bar\cS_s^\prime\big\}^\frac{1}{2^{t+1}}\Big)\\
        &\times\bigg( \frac{ D_{t+1}^{\frac{1}{2^{t+1}}} \prod_{s=t+1}^{T-1} \bar\cS_s^\frac{1}{2^{t+1}}}{2^{M_t/2}}  + \sum_{s=t+1}^{T-1} \frac{L_{[t+1:s-1]}S_s D_{s+1}^{\frac{1}{2^{s}}} \prod_{s'=s+1}^{T-1} \bar\cS_{s'}^\frac{1}{2^{s}}}{2^{M_s+1}} \bigg).
    \end{align*}
    By construction, all constants in the resulting bound are independent of the truncation points. Since~$M_s \geq 0$, we can lower bound the denominator~$2^{M_s+1}$ by~$2^{M_s/2}$ and introduce auxiliary constants~$W_t' \geq 0$ that are independent of the truncation points to obtain a simpler but relaxed bias bound of the form
    \begin{align*}
        \|\EE[\hat{G}(x)] - G_1(x)\|_2 \leq \sum_{t=1}^{T-1} \frac{W_t'}{2^{M_t/2}}.
    \end{align*}
    Note that~$W_t'$ absorbs all dependencies of the bias on the Lipschitz, smoothness and Hölder constants, as well as on the dimensions, moment bounds and rate parameters. Setting $W_t=W^\prime_t \sqrt{2}(T-1)$, it is now evident that the bias drops below~$\epsilon / \sqrt{2}$ uniformly across all $x\in\cX$ if $M_{t}  =\lceil 2\log_2(W_t /\epsilon)\rceil$ for all $t\in[T-1]$. 

    Next, we show that the variance of the MLMC gradient estimator is bounded by $\epsilon^2/2$. From the first part of the proof we know that~$\nu_1^2(x)$ is bounded above by $E_1 \prod_{t=1}^{T-1}\max \{\bar\cS_t,\bar\cS_t^\prime \}$ uniformly across all~$x\in\cX$. Thus, both $\sup_{x\in\cX} \nu_1^2(x)$ and~$n_1$ are finite. Consequently, we may conclude that
    \[
        \VV(\hat{G}(x)) \leq \frac{\nu_1^2(x)}{n_1} \leq \frac{\epsilon^2}{2}
    \]
    uniformly across all $x\in\cX$, where the two inequalities follow from~\eqref{eq:variance-rumlmc-estimator} and the definition of~$n_1$.
    
    Finally, the expected sampling cost of the MLMC gradient estimator is given by
    \begin{align*}
        \EE[\CC(\hat{G}(x))]= n_1 \prod_{t=1}^{T-1} \sum_{\ell=0}^{M_t} q_t(\ell)2^{\ell} = n_1 \prod_{t=1}^{T-1} \sum_{\ell=0}^{M_t} \frac{2^\ell(1-r_t)^\ell}{Z_t} = n_1 \prod_{t=1}^{T-1} \frac{1-(2-2r_t)^{M_t+1}}{Z_t (2r_t - 1)},
    \end{align*}
    where the first equality can be proved as in Lemma~\ref{lemma:bound_on_cost_rtmlmc}. As~$2^{-1}< r_t <1$ by virtue of Assumption~\ref{assump:rate-parameters-rumlmc-grad}, the resulting expression remains bounded as~$M_t$ tends to infinity. 
    
    In summary, the above reasoning implies that the root mean squared error of the MLMC gradient estimator is bounded by~$\epsilon$ uniformly across all~$x\in\cX$ if $n_1=\cO(\epsilon^{-2})$ and $M_t= \cO(\log(\epsilon^{-1}))$ for all $t \in [T-1]$, in which case the expected sampling cost satisfies $\EE[\CC(\hat{G}(x))] \leq \cO(\epsilon^{-2})$. Hence, the claim follows.
\end{proof}

\begin{proof}[Proof of Theorem~\ref{theorem:sgd-num-steps_rt-mlmc}]
   Under the stated assumptions, Lemmas~\ref{lemma:bound_on_variance_rumlmc_grad}, \ref{lemma:bias_bound_rtmlmc_grad}, and~\ref{lemma:bound_on_cost_rumlmc_grad}, together with Theorem~\ref{thm:rtmlmc_grad_sample_complexity_mse}, are applicable. It is therefore easy to verify that the claim is a direct consequence of~\cite[Theorem~C.1]{hu2021bias}.
\end{proof}


\end{document}